\documentclass{amsart}

\usepackage{hyperref}

\usepackage{amsmath,amssymb,amsfonts,amsthm}
\usepackage[all]{xy}
\usepackage{enumerate}
\usepackage{tikz}
\usepackage{array}
\usepackage{mathrsfs}
\usepackage{csquotes}
\usepackage{fullpage}

\theoremstyle{theorem}
\newtheorem{thm}{Theorem}[section]
\newtheorem{prop}[thm]{Proposition}
\newtheorem{lem}[thm]{Lemma}

\theoremstyle{remark}
\newtheorem{rem}[thm]{Remark}
\newtheorem{ex}[thm]{Example}

\theoremstyle{definition}
\newtheorem{defi}[thm]{Definition}

\newcommand{\D}{\mathcal{D}}
\renewcommand{\H}{\mathcal{H}}
\renewcommand{\P}{\mathcal{P}}
\newcommand{\Q}{\mathbb{Q}}
\newcommand{\C}{\mathbb{C}}
\renewcommand{\c}{\mathscr{C}}
\renewcommand{\k}{\mathscr{K}}

\newcommand{\s}{\mathscr{S}}
\newcommand{\I}{\mathbb{I}}
\newcommand{\gr}{\mathrm{gr}}
\newcommand{\sgn}{\mathrm{sgn}}
\newcommand{\wt}[1]{\widetilde{#1}}
\renewcommand{\r}{0.8}

\newcommand{\dddegfourexintro}[1]{\begin{tikzpicture}
\def \r {#1};
\draw (0,0) circle (\r);
\foreach \k in {1,...,4} \node[circle, inner sep=1.5 pt, fill=black] (\k) at (-90-360*\k/5:\r) {};
\node[circle,inner sep=1.5pt,draw=black,fill=white] (root) at (-90:\r) {};
\draw (1) to (3); \draw (2) to (1); \draw (3) to (root); \draw (4) to (3);
\end{tikzpicture}}

\newcommand{\decorateddddegreetwointro}[1]{\begin{tikzpicture}
\def \r {#1};
\draw (0,0) circle (\r);
\foreach \k in {1,2} \node[circle,inner sep=1.5pt,fill=black] (\k) at (-90-360*\k/3:\r){};
\node[circle,inner sep=1.5 pt,draw=black,fill=white] (root) at (-90:\r){};
\draw (1) to node[right]{\small $u$} (root);
\draw (2) to node[below]{\small $v$} (1);
\node at (-90-360/6:1.3*\r){$-a$};
\node at (-90-360/6-360/3:1.2*\r){$0$};
\node at (-90-360/6-360*2/3:1.2*\r){$b$};
\end{tikzpicture}}

\newcommand{\corolla}[2]{\begin{tikzpicture}
\def \n {#1};
\def \r {#2};
\pgfmathsetmacro\nn{\n+1}; 
\draw (0,0) circle (\r);
\foreach \k in {1,...,\n} \node[circle,inner sep=1.5pt,fill=black]  (\k) at (-90-360*\k/\nn:\r) {};
\node[circle,inner sep=1.5pt,draw=black,fill=white] (root) at (-90:\r) {};
\foreach \k in {1,...,\n} \draw (\k) to  (root);
\end{tikzpicture}}

\newcommand{\directedcorolla}[2]{\begin{tikzpicture}
\def \n {#1};
\def \r {#2};
\pgfmathsetmacro\nn{\n+1}; 
\draw (0,0) circle (\r);
\foreach \k in {1,...,\n} \node[circle,inner sep=1.5pt,fill=black]  (\k) at (-90-360*\k/\nn:\r) {};
\node[circle,inner sep=1.5pt,draw=black,fill=white] (root) at (-90:\r) {};
\foreach \k in {1,...,\n} \draw[->,>=latex] (\k) to  (root);
\end{tikzpicture}}

\newcommand{\dddegtwoleft}[1]{\begin{tikzpicture}
\def \r {#1};
\draw (0,0) circle (\r);
\foreach \k in {1,2} \node[circle, inner sep=1.5 pt, fill=black] (\k) at (-90-360*\k/3:\r) {};
\node[circle,inner sep=1.5pt,draw=black,fill=white] (root) at (-90:\r) {};
\draw (1) to (root);
\draw (2) to (1);
\end{tikzpicture}}

\newcommand{\dddegtworight}[1]{\begin{tikzpicture}
\def \r {#1};
\draw (0,0) circle (\r);
\foreach \k in {1,2} \node[circle, inner sep=1.5 pt, fill=black] (\k) at (-90-360*\k/3:\r) {};
\node[circle,inner sep=1.5pt,draw=black,fill=white] (root) at (-90:\r) {};
\draw (2) to (root);
\draw (1) to (2);
\end{tikzpicture}}

\newcommand{\dddegthreeexOne}[1]{\begin{tikzpicture}
\def \r {#1};\draw (0,0) circle (\r);
\foreach \k in {1,...,3} \node[circle,inner sep=1.5pt,fill=black]  (\k) at (-90-360*\k/4:\r) {};
\node[circle,inner sep=1.5pt,draw=black,fill=white] (root) at (-90:\r) {};
\draw (1) to (2);\draw (2) to (3);\draw (3) to (root);
\end{tikzpicture}}

\newcommand{\dddegthreeexTwo}[1]{\begin{tikzpicture}
\def \r {#1};\draw (0,0) circle (\r);
\foreach \k in {1,...,3} \node[circle,inner sep=1.5pt,fill=black]  (\k) at (-90-360*\k/4:\r) {};
\node[circle,inner sep=1.5pt,draw=black,fill=white] (root) at (-90:\r) {};
\draw (1) to (root);\draw (2) to (1);\draw (3) to (2);
\end{tikzpicture}}

\newcommand{\dddegthreeexThree}[1]{\begin{tikzpicture}
\def \r {#1};\draw (0,0) circle (\r);
\foreach \k in {1,...,3} \node[circle,inner sep=1.5pt,fill=black]  (\k) at (-90-360*\k/4:\r) {};
\node[circle,inner sep=1.5pt,draw=black,fill=white] (root) at (-90:\r) {};
\draw (1) to (3);\draw (2) to (1);\draw (3) to (root);
\end{tikzpicture}}

\newcommand{\dddegthreeexFour}[1]{\begin{tikzpicture}
\def \r {#1};\draw (0,0) circle (\r);
\foreach \k in {1,...,3} \node[circle,inner sep=1.5pt,fill=black]  (\k) at (-90-360*\k/4:\r) {};
\node[circle,inner sep=1.5pt,draw=black,fill=white] (root) at (-90:\r) {};
\draw (1) to (root);\draw (2) to (3);\draw (3) to (1);
\end{tikzpicture}}

\newcommand{\dddegthreeexFive}[1]{\begin{tikzpicture}
\def \r {#1};\draw (0,0) circle (\r);
\foreach \k in {1,...,3} \node[circle,inner sep=1.5pt,fill=black]  (\k) at (-90-360*\k/4:\r) {};
\node[circle,inner sep=1.5pt,draw=black,fill=white] (root) at (-90:\r) {};
\draw (1) to (root);\draw (2) to (1);\draw (3) to (root);
\end{tikzpicture}}

\newcommand{\dddegthreeexSix}[1]{\begin{tikzpicture}
\def \r {#1};\draw (0,0) circle (\r);
\foreach \k in {1,...,3} \node[circle,inner sep=1.5pt,fill=black]  (\k) at (-90-360*\k/4:\r) {};
\node[circle,inner sep=1.5pt,draw=black,fill=white] (root) at (-90:\r) {};
\draw (1) to (root);\draw (2) to (3);\draw (3) to (root);
\end{tikzpicture}}

\newcommand{\dddegthreeexSeven}[1]{\begin{tikzpicture}
\def \r {#1};\draw (0,0) circle (\r);
\foreach \k in {1,...,3} \node[circle,inner sep=1.5pt,fill=black]  (\k) at (-90-360*\k/4:\r) {};
\node[circle,inner sep=1.5pt,draw=black,fill=white] (root) at (-90:\r) {};
\draw (1) to (2);\draw (2) to (root);\draw (3) to (root);
\end{tikzpicture}}

\newcommand{\dddegthreeexEight}[1]{\begin{tikzpicture}
\def \r {#1};\draw (0,0) circle (\r);
\foreach \k in {1,...,3} \node[circle,inner sep=1.5pt,fill=black]  (\k) at (-90-360*\k/4:\r) {};
\node[circle,inner sep=1.5pt,draw=black,fill=white] (root) at (-90:\r) {};
\draw (1) to (root);\draw (2) to (root);\draw (3) to (2);
\end{tikzpicture}}

\newcommand{\dddegthreeexNine}[1]{\begin{tikzpicture}
\def \r {#1};\draw (0,0) circle (\r);
\foreach \k in {1,...,3} \node[circle,inner sep=1.5pt,fill=black]  (\k) at (-90-360*\k/4:\r) {};
\node[circle,inner sep=1.5pt,draw=black,fill=white] (root) at (-90:\r) {};
\draw (1) to (3);\draw (2) to (3);\draw (3) to (root);
\end{tikzpicture}}

\newcommand{\dddegthreeexTen}[1]{\begin{tikzpicture}
\def \r {#1};\draw (0,0) circle (\r);
\foreach \k in {1,...,3} \node[circle,inner sep=1.5pt,fill=black]  (\k) at (-90-360*\k/4:\r) {};
\node[circle,inner sep=1.5pt,draw=black,fill=white] (root) at (-90:\r) {};
\draw (1) to (root);\draw (2) to (1);\draw (3) to (1);
\end{tikzpicture}}

\newcommand{\dddegthreeexEleven}[1]{\begin{tikzpicture}
\def \r {#1};\draw (0,0) circle (\r);
\foreach \k in {1,...,3} \node[circle,inner sep=1.5pt,fill=black]  (\k) at (-90-360*\k/4:\r) {};
\node[circle,inner sep=1.5pt,draw=black,fill=white] (root) at (-90:\r) {};
\draw (1) to (2);\draw (2) to (root);\draw (3) to (2);
\end{tikzpicture}}

\newcommand{\dddegthreeexTwelve}[1]{\begin{tikzpicture}
\def \r {#1};\draw (0,0) circle (\r);
\foreach \k in {1,...,3} \node[circle,inner sep=1.5pt,fill=black]  (\k) at (-90-360*\k/4:\r) {};
\node[circle,inner sep=1.5pt,draw=black,fill=white] (root) at (-90:\r) {};
\draw (1) to (root);\draw (2) to (root);\draw (3) to (root);
\end{tikzpicture}}

\newcommand{\tabledissectiondiagrams}[1]{
\begin{tabular}{|c|}\hline
	\begin{tabular}{c|c|c}
	degree $0$ & degree $1$ & degree $2$ \\ 
	\corolla{0}{#1} & \corolla{1}{#1} & \dddegtworight{#1}  \hspace{.5cm} \dddegtwoleft{#1}\hspace{.5cm}\corolla{2}{#1} 
	\end{tabular}\\ \hline
	degree $3$\\
	\dddegthreeexOne{#1} \hspace{.8cm} \dddegthreeexTwo{#1} \hspace{.8cm} \dddegthreeexThree{#1} \hspace{.8cm} \dddegthreeexFour{#1}\\ 
	\dddegthreeexFive{#1} \hspace{.8cm} \dddegthreeexSix{#1} \hspace{.8cm} \dddegthreeexSeven{#1} \hspace{.8cm} \dddegthreeexEight{#1}\\  
	\dddegthreeexNine{#1} \hspace{.8cm} \dddegthreeexTen{#1} \hspace{.8cm} \dddegthreeexEleven{#1} \hspace{.8cm} \dddegthreeexTwelve{#1}\\\hline
\end{tabular}
}

\newcommand{\isoternarytrees}[1]{\begin{tikzpicture}
\def \r {#1};
\draw (0,0) circle (\r);
\node[circle,inner sep=1.5pt,fill=black] (rho) at (-90-360*2/5:\r){};
\foreach \k in {1,2} \node at (-90-360*2/5:1.15*\r){$\rho$};
\node[circle,inner sep=1.5 pt,draw=black,fill=white] (root) at (-90:\r){};
\draw (rho) to node[right]{$c$} (root);
\node[circle,draw] (D1) at (180:.7*\r){$D_1$};
\node[circle,draw](D2) at (60:.5*\r){$D_2$};
\node[circle,draw](D3) at (-50:.5*\r){$D_3$};
\draw (D1) to (rho); \draw (D2) to (rho); \draw (D3) to (root);
\end{tikzpicture}}

\newcommand{\ternarytree}{
\begin{tikzpicture}[grow=north,level distance=8mm]
  \node {}
    child{node[circle,inner sep=1.5 pt,draw=black,fill=black]{}
    	child{node[circle,draw]{$T_3$}}
    	child{node[circle,draw]{$T_2$}}
    	child{node[circle,draw]{$T_1$}}
    };
\end{tikzpicture}}

\newcommand{\dddegthree}[1]{\begin{tikzpicture}
\def \r {#1};
\draw (0,0) circle (\r);
\foreach \k in {1,...,3} \node[circle,inner sep=1.5pt,fill=black]  (\k) at (-90-360*\k/4:\r) {};
\node[circle,inner sep=1.5pt,draw=black,fill=white] (root) at (-90:\r) {};
\draw (1) to (root);
\draw (2) to (1);
\draw (3) to (1);
\end{tikzpicture}}

\newcommand{\dddegthreedirected}[1]{\begin{tikzpicture}
\def \r {#1};
\foreach \k in {1,...,3} \node[circle,inner sep=1.5pt,fill=black]  (\k) at (-90-360*\k/4:\r) {};
\node[circle,inner sep=1.5pt,draw=black,fill=white] (root) at (-90:\r) {};
\draw[->,>=latex] (1) to (root);
\draw[->,>=latex] (2) to (1);
\draw[->,>=latex] (3) to (1);
\draw[bend left=40,->,>=latex] (root) to (1);
\draw[bend left=40,->,>=latex] (1) to (2);
\draw[bend left=40,->,>=latex] (2) to (3);
\draw[bend left=40,->,>=latex] (3) to (root);
\end{tikzpicture}}

\newcommand{\dddegthreelabelingarrows}[1]{\begin{tikzpicture}
\def \r {#1};
\draw (0,0) circle (\r);
\foreach \k in {1,...,3} \node[circle,inner sep=1.5pt,fill=black]  (\k) at (-90-360*\k/4:\r) {};
\foreach \k in {1,...,3} \node at (-90-360*\k/4:1.3*\r) {$\k$};
\node[circle,inner sep=1.5pt,draw=black,fill=white] (root) at (-90:\r) {};
\draw[->,>=latex] (1) to node[above,near end]{$1$} (root) ;
\draw[->,>=latex] (2) to node[below, near start]{$2$} (1);
\draw[->,>=latex] (3) to node[above, near start]{$3$} (1);
\end{tikzpicture}}

\newcommand{\dddegthreelabelingpolygons}[1]{\begin{tikzpicture}
\def \r {#1};
\draw (0,0) circle (\r);
\foreach \k in {1,...,3} \node[circle,inner sep=1.5pt,fill=black]  (\k) at (-90-360*\k/4:\r) {};
\foreach \k in {1,...,3} \node at (-90-360*\k/4:1.3*\r) {$\k$};
\node at (-135:1.3*\r){$0$};
\foreach \k in {1,...,3} \node at (-135-360*\k/4:1.2*\r) {$\k$};
\node[circle,inner sep=1.5pt,draw=black,fill=white] (root) at (-90:\r) {};
\end{tikzpicture}}

\newcommand{\dddegthreeexample}[1]{\begin{tikzpicture}
\def \r {#1};
\draw (0,0) circle (\r);
\foreach \k in {1,...,3} \node[circle,inner sep=1.5pt,fill=black]  (\k) at (-90-360*\k/4:\r) {};
\foreach \k in {1,...,3} \node at (-90-360*\k/4:1.3*\r) {$\k$};
\node at (-135:1.2*\r){$0$};
\foreach \k in {1,...,3} \node at (-135-360*\k/4:1.2*\r) {$\k$};
\node[circle,inner sep=1.5pt,draw=black,fill=white] (root) at (-90:\r) {};
\draw[->,>=latex] (1) to node[above,near end]{$1$} (root) ;
\draw[->,>=latex] (2) to node[below, near start]{$2$} (1);
\draw[->,>=latex,ultra thick] (3) to node[above, near start]{$3$} (1);
\end{tikzpicture}}

\newcommand{\dddegthreeexamplefinal}[1]{\begin{tikzpicture}
\def \r {#1};
\draw (0,0) circle (\r);
\foreach \k in {1,...,3} \node[circle,inner sep=1.5pt,fill=black]  (\k) at (-90-360*\k/4:\r) {};
\foreach \k in {1,...,3} \node at (-90-360*\k/4:1.3*\r) {$\k$};
\node at (-135:1.2*\r){$0$};
\foreach \k in {1,...,3} \node at (-135-360*\k/4:1.2*\r) {$\k$};
\node[circle,inner sep=1.5pt,draw=black,fill=white] (root) at (-90:\r) {};
\draw[->,>=latex] (1) to node[above,near end]{$1$} (root) ;
\draw[->,>=latex] (2) to node[below, near start]{$2$} (1);
\draw[->,>=latex] (3) to node[above, near start]{$3$} (1);
\end{tikzpicture}}

\newcommand{\dddegthreeexampleQ}[1]{\begin{tikzpicture}
\def \r {#1};
\draw (0,\r) circle (\r);
\draw (0,-\r) circle (\r);
\node[circle,inner sep=1.5pt,fill=black]  (1) at (0,2*\r) {};
\node[circle,inner sep=1.5pt,fill=black]  (root1) at (0,0) {};
\node[circle,inner sep=1.5pt,draw=black,fill=white]  (root) at (0,-2*\r) {};
\draw[->,>=latex] (1) to node[right]{$2$} (root1) ;
\draw[->,>=latex] (root1) to node[right]{$1$} (root);
\node at (-1.3*\r,\r) {$1$};
\node at (1.3*\r,\r) {$2$};
\node at (-1.3*\r,-\r) {$0$};
\node at (1.3*\r,-\r) {$3$};
\end{tikzpicture}}

\newcommand{\dddegthreeexampleQbis}[1]{\begin{tikzpicture}
\def \r {#1};
\draw (0,0) circle (\r);
\node[circle,inner sep=1.5pt,fill=black]  (1) at (0,\r) {};
\node[circle,inner sep=1.5pt,draw=black,fill=white]  (root) at (0,-\r) {};
\node at (-1.2*\r,0) {$1$};
\node at (1.2*\r,0) {$2$};
\draw[->,>=latex] (1) to node[right]{$2$} (root);
\end{tikzpicture}}

\newcommand{\dddegthreeexampleQter}[1]{\begin{tikzpicture}
\def \r {#1};
\draw (0,0) circle (\r);
\node[circle,inner sep=1.5pt,fill=black]  (1) at (0,\r) {};
\node[circle,inner sep=1.5pt,draw=black,fill=white]  (root) at (0,-\r) {};
\node at (-1.2*\r,0) {$0$};
\node at (1.2*\r,0) {$3$};
\draw[->,>=latex] (1) to node[right]{$1$} (root);
\end{tikzpicture}}

\newcommand{\dddegthreeexampleR}[1]{\begin{tikzpicture}
\def \r {#1};
\draw (0,0) circle (\r);
\foreach \k in {1,...,3} \node[circle,inner sep=1.5pt,fill=black]  (\k) at (-90-360*\k/4:\r) {};
\foreach \k in {1,...,3} \node at (-90-360*\k/4:1.3*\r) {$\k$};
\node at (-135:1.3*\r){$0$};
\foreach \k in {1,...,3} \node at (-135-360*\k/4:1.2*\r) {$\k$};
\node[circle,inner sep=1.5pt,draw=black,fill=white] (root) at (-90:\r) {};
\draw[->,>=latex] (3) to node[above]{$3$} (1);
\end{tikzpicture}}

\newcommand{\dddegthreeexampleRbis}[1]{\begin{tikzpicture}
\def \r {#1};
\draw (0,0) circle (\r);
\node[circle,inner sep=1.5pt,fill=black]  (1) at (0,\r) {};
\node[circle,inner sep=1.5pt,draw=black,fill=white]  (root) at (0,-\r) {};
\node at (-1.2*\r,0) {$0$};
\node at (1.2*\r,0) {$1$};
\draw[->,>=latex] (root) to node[right]{$3$} (1);
\end{tikzpicture}}

\newcommand{\examplestable}[1]{
\begin{tabular}{|c|c|>{\centering\arraybackslash}m{2.3cm}|>{\centering\arraybackslash}m{1.3cm}|c|c|}\hline
$C$ & $\s_C^+$ & $q_C(D)$ & $r_C(D)$ & $\mathscr{K}_C(D)$ & $k_C(D)$ \\ \hline

$\{1,2,3\}$ & $\varnothing$ & $1$ & \dddegthree{#1} & $\varnothing$ & $0$ \\ \hline
$\{1,2\}$ & $\{3\}$ & \corolla{1}{#1} & \dddegtwoleft{#1} & $\varnothing$ & $0$ \\ \hline
$\{1,3\}$ & $\{2\}$ & \corolla{1}{#1} & \dddegtwoleft{#1} & $\varnothing$ & $0$ \\ \hline
$\{2,3\}$ & $\{3\}$ & \corolla{1}{#1} & \dddegtwoleft{#1} & $\{3\}$ & $1$ \\ \hline
$\{1\}$ & $\{2,3\}$ & \corolla{2}{#1} & \corolla{1}{#1} & $\varnothing$ & $0$ \\ \hline
$\{2\}$ & $\{2,3\}$ & \dddegtwoleft{#1} & \corolla{1}{#1} & $\{2\}$ & $1$ \\ \hline
$\{3\}$ & $\{2,3\}$ & \corolla{1}{#1} \hspace{.1cm}\corolla{1}{#1} & \corolla{1}{#1} & $\{3\}$ & $1$ \\ \hline
$\varnothing$ & $\{1,2,3\}$ & \dddegthree{#1} & $1$ & $\varnothing$ & $0$ \\ \hline
\end{tabular}}

\newcommand{\vc}[1]{\vcenter{\hbox{#1}}}

\newcommand{\exampleformulacoproduct}[1]{
\begin{eqnarray*}
\Delta\left( \vc{\dddegthree{#1}} \right)&  = & 1\otimes \vc{\dddegthree{#1}}\\
&+& \vc{\corolla{1}{#1}} \otimes \vc{\dddegtwoleft{#1}} \\
&+& \left(\vc{\corolla{2}{#1}} - \vc{\dddegtwoleft{#1}} - \vc{\corolla{1}{#1}}\hspace{.1cm}\vc{\corolla{1}{#1}} \right) \otimes \vc{\corolla{1}{#1}}\\
&+&  \vc{\dddegthree{#1}}\otimes 1.
\end{eqnarray*}}

\newcommand{\pathtreeFour}[1]{\begin{tikzpicture}
\def \r {#1};
\draw (0,0) circle (\r);
\foreach \k in {1,...,4} \node[circle, inner sep=1.5pt, fill=black] (\k) at (-90-360*\k/5:\r) {};
\node[circle,inner sep=1.5pt,draw=black,fill=white] (root) at (-90:\r) {};
\draw (1) to (2); \draw (2) to (3); \draw (3) to (4); \draw (4) to (root);
\end{tikzpicture}}

\newcommand{\changedirection}{\begin{tikzpicture}
\foreach \k in {-2,-1,1,2} \node[circle,inner sep=1.5pt,fill=black] (\k) at (\k,0){};
\draw[->,>=latex] (-2) to node[above]{\small $\alpha$} (-1);
\draw[->,>=latex] (2) to node[above]{\small $-\alpha$} (1);
\node at (0,0){$\leadsto$};
\end{tikzpicture}}

\newcommand{\contraction}{\begin{tikzpicture}
\node[circle, inner sep=1.5pt,fill=black] (vm) at (0,0){};
\node at (0,.4){$v_-$};
\node[circle, inner sep=1.5pt,fill=black] (vp) at (2,0){};
\node at (2,0.3){$v_+$};
\draw[->,>=latex] (vm) to node[above]{\small $\alpha$} (vp);
\draw (vp) -- ++(0.3,0); \draw (vp) -- ++(0.3,0.2); \draw (vp) -- ++(0.3,-.2);
\draw[->,>=latex,bend right=10] (-1,-1) to node[left]{\small $x'$} (vm);
\draw[->,>=latex,bend left=10] (-1,1) to node[left]{\small $x$} (vm);
\draw[->,>=latex,bend left=10] (vm) to node[right]{\small $y$} (1,1);
\draw[->,>=latex,bend right=10] (vm) to node[right]{\small $y'$} (1,-1);
\node at (3.5,0){$\leadsto$};
\node[circle, inner sep=1.5pt,fill=black] (v) at (6,0){};
\node at (6,0.3){$v$};
\draw (v) -- ++(0.3,0); \draw (v) -- ++(0.3,0.2); \draw (v) -- ++(0.3,-.2);
\draw[->,>=latex,bend right=10] (5,-1) to node[left]{\small $x'+\alpha$} (v);
\draw[->,>=latex,bend left=10] (5,1) to node[left]{\small $x+\alpha$} (v);
\draw[->,>=latex,bend left=10] (v) to node[right]{\small $y-\alpha$} (7,1);
\draw[->,>=latex,bend right=10] (v) to node[right]{\small $y'-\alpha$} (7,-1);
\end{tikzpicture}}

\newcommand{\decorateddddegreetwo}[1]{\begin{tikzpicture}
\def \r {#1};
\draw (0,0) circle (\r);
\foreach \k in {1,2} \node[circle,inner sep=1.5pt,fill=black] (\k) at (-90-360*\k/3:\r){};
\foreach \k in {1,2} \node at (-90-360*\k/3:1.3*\r){$\k$};
\node[circle,inner sep=1.5 pt,draw=black,fill=white] (root) at (-90:\r){};
\draw[->,>=latex] (1) to node[right]{\small $a_1$} (root);
\draw[->,>=latex] (2) to node[below]{\small $a_2$} (1);
\foreach \k in {0,1,2} \node at (-90-360/6-360*\k/3:1.2*\r){\small $b_\k$};
\end{tikzpicture}}

\newcommand{\dddegthreeexampledecorated}[1]{\begin{tikzpicture}
\def \r {#1};
\draw (0,0) circle (\r);
\foreach \k in {1,...,3} \node[circle,inner sep=1.5pt,fill=black]  (\k) at (-90-360*\k/4:\r) {};
\foreach \k in {1,...,3} \node at (-90-360*\k/4:1.3*\r) {$\k$};
\node at (-135:1.2*\r){\small $b_0$};
\foreach \k in {1,...,3} \node at (-135-360*\k/4:1.2*\r) {\small $b_\k$};
\node[circle,inner sep=1.5pt,draw=black,fill=white] (root) at (-90:\r) {};
\draw[->,>=latex] (1) to node[right]{\small $a_1$} (root) ;
\draw[->,>=latex] (2) to node[right]{\small $a_2$} (1);
\draw[->,>=latex,very thick] (3) to node[above, near start]{\small $a_3$} (1);
\end{tikzpicture}}

\newcommand{\dddegonedirecteddecorated}[4]{\begin{tikzpicture}
\def \r {#1};
\draw (0,0) circle (\r);
\node[circle, inner sep=1.5pt, fill=black] (1) at (90:\r) {};
\node[circle,inner sep=1.5pt,draw=black,fill=white] (root) at (-90:\r) {};
\draw[->,>=latex] (1) to node[left]{#2} (root); 
\node[anchor=east] at (180:\r){#3};
\node[anchor=west] at (0:\r){ #4};
\end{tikzpicture}}

\newcommand{\dddegonedirecteddecoratedwrong}[4]{\begin{tikzpicture}
\def \r {#1};
\draw (0,0) circle (\r);
\node[circle, inner sep=1.5pt, fill=black] (1) at (90:\r) {};
\node[circle,inner sep=1.5pt,draw=black,fill=white] (root) at (-90:\r) {};
\draw[->,>=latex] (root) to node[left]{#2} (1); 
\node[anchor=east] at (180:\r){#3};
\node[anchor=west] at (0:\r){#4};
\end{tikzpicture}}

\newcommand{\decoratedlogarithm}[1]{\begin{tikzpicture}
\def \r {#1};
\draw (0,0) circle (\r);
\node[circle, inner sep=1.5pt, fill=black] (1) at (90:\r) {};
\node[circle,inner sep=1.5pt,draw=black,fill=white] (root) at (-90:\r) {};
\draw (1) to node[left]{\small $a_1$} (root); 
\node at (180:1.3*\r){\small $b_0$};
\node at (0:1.3*\r){\small $b_1$};
\end{tikzpicture}}

\newcommand{\decoratedcorollaFour}[1]{\begin{tikzpicture}
\def \r {#1};
\draw (0,0) circle (\r);
\foreach \k in {1,...,4} \node[circle, inner sep=1.5pt, fill=black] (\k) at (-90-360*\k/5:\r) {};
\node[circle,inner sep=1.5pt,draw=black,fill=white] (root) at (-90:\r) {};
\draw (1) to node[above]{\small $a_1$} (root);
\draw (2) to (root);
\node at (120:.5*\r) {\small $a_2$};
\draw (3) to  (root);
\node at (20:.7*\r) {\small $a_3$};
\draw (4) to node[above]{\small $a_4$} (root);
\node at (-90-360/10:1.3*\r){\small $-a_0$};
\node at (-90+360/10:1.3*\r){\small $a_5$};
\foreach \k in {1,...,3} \node at (-90-360/10-360*\k/5:1.2*\r) {\small $0$};
\end{tikzpicture}}

\newcommand{\decoratedpathtreeFour}[1]{\begin{tikzpicture}
\def \r {#1};
\draw (0,0) circle (\r);
\foreach \k in {1,...,4} \node[circle, inner sep=1.5pt, fill=black] (\k) at (-90-360*\k/5:\r) {};
\node[circle,inner sep=1.5pt,draw=black,fill=white] (root) at (-90:\r) {};
\draw (1) to  (2); \draw (2) to (3); \draw (3) to (4); \draw (4) to  (root);
\node at (-90-360/10:1.3*\r){\small $-b$};
\foreach \k in {1,...,4} \node at (-90-360/10-360*\k/5:1.2*\r) {\small $0$};
\foreach \k in {1,...,4} \node at (-90-360/10-360*\k/5:.6*\r) {\small $a_\k$};
\end{tikzpicture}}

\newcommand{\relationtranslation}{\begin{tikzpicture}
\node[circle, inner sep=1.5pt,fill=black] (vm) at (0,0){};
\node at (0,.4){$i$};
\draw[->,>=latex,bend right=10] (-1,-1) to node[left]{\small $x'$} (vm);
\draw[->,>=latex,bend left=10] (-1,1) to node[left]{\small $x$} (vm);
\draw[->,>=latex,bend left=10] (vm) to node[right]{\small $y$} (1,1);
\draw[->,>=latex,bend right=10] (vm) to node[right]{\small $y'$} (1,-1);
\node at (1.7,0){$\leadsto$};
\node[circle, inner sep=1.5pt,fill=black] (v) at (4,0){};
\node at (4,0.4){$i$};
\draw[->,>=latex,bend right=10] (3,-1) to node[left]{\small $x'+\lambda$} (v);
\draw[->,>=latex,bend left=10] (3,1) to node[left]{\small $x+\lambda$} (v);
\draw[->,>=latex,bend left=10] (v) to node[right]{\small $y-\lambda$} (5,1);
\draw[->,>=latex,bend right=10] (v) to node[right]{\small $y'-\lambda$} (5,-1);
\end{tikzpicture}}

\newcommand{\relationstokes}[1]{\begin{tikzpicture}
\def \r {#1};
\draw (0,0) circle (\r);
\foreach \k in {1,...,5} \node[circle,inner sep=1.5pt,fill=black] (\k) at (-90-360*\k/6:\r){};
\node[circle,inner sep=1.5 pt,draw=black,fill=white] (root) at (-90:\r){};
\draw (1) to (2); \draw (1) to (root); \draw (3) to (5); \draw (4) to (5);
\draw[bend left=30, very thick] (2) to (3);
\draw[bend left=30, very thick] (5) to (root);
\node at (90+360/12:1.2*\r){$i$};
\node at (-90+360/12:1.2*\r){$j$};
\end{tikzpicture}}

\newcommand{\relationorliksolomon}[1]{\begin{tikzpicture}
\def \r {#1};
\draw (0,0) circle (\r);
\foreach \k in {1,...,4} \node[circle,inner sep=1.5pt,fill=black] (\k) at (-90-360*\k/5:\r){};
\node[circle,inner sep=1.5 pt,draw=black,fill=white] (root) at (-90:\r){};
\draw (1) to (2); \draw[dashed] (1) to (3); \draw[dashed] (1) to (4); \draw[dashed] (3) to (4); \draw (4) to (root);
\end{tikzpicture}}

\newcommand{\reditintcaseone}[1]{\begin{tikzpicture}
\def \r {#1};
\draw (0,0) circle (\r);
\foreach \k in {1,...,4} \node[circle,inner sep=1.5pt,fill=black] (\k) at (-90-360*\k/5:\r){};
\node at (-90-360*2/5:1.3*\r){$k$};
\node at (-90-360*3/5:1.3*\r){$v_0$};
\node at (-90-360*4/5:1.3*\r){$v_1$};
\node at (0:.2*\r){$c$};
\node at (10:.65*\r){$c_0$};
\node at (-40:.65*\r){$c_1$};
\node[circle,inner sep=1.5 pt,draw=black,fill=white] (root) at (-90:\r){};
\draw (1) to (root); \draw (2) to (root); \draw[dashed] (3) to (root); \draw (3) to (4); \draw (4) to (root);
\node at (-90:1.7*\r){$D=\widehat{D}\setminus\{c\}$};
\end{tikzpicture}}

\newcommand{\reditintcaseonebis}[1]{\begin{tikzpicture}
\def \r {#1};
\draw (0,0) circle (\r);
\foreach \k in {1,...,4} \node[circle,inner sep=1.5pt,fill=black] (\k) at (-90-360*\k/5:\r){};
\node at (-90-360*2/5:1.3*\r){$k$};
\node at (-90-360*3/5:1.3*\r){};
\node at (-90-360*4/5:1.3*\r){};
\node at (0:.2*\r){$c$};
\node at (-40:.65*\r){$c_1$};
\node[circle,inner sep=1.5 pt,draw=black,fill=white] (root) at (-90:\r){};
\draw (1) to (root); \draw (2) to (root); \draw (3) to (root); \draw (4) to (root);
\node at (-90:1.7*\r){$\widehat{D}\setminus\{c_0\}$};
\end{tikzpicture}}

\newcommand{\reditintcaseoneter}[1]{\begin{tikzpicture}
\def \r {#1};
\draw (0,0) circle (\r);
\foreach \k in {1,...,4} \node[circle,inner sep=1.5pt,fill=black] (\k) at (-90-360*\k/5:\r){};
\node at (-90-360*2/5:1.3*\r){$k$};
\node at (-90-360*3/5:1.3*\r){};
\node at (-90-360*4/5:1.3*\r){};
\node at (0:.2*\r){$c$};
\node at (10:.65*\r){$c_0$};
\node[circle,inner sep=1.5 pt,draw=black,fill=white] (root) at (-90:\r){};
\draw (1) to (root); \draw (2) to (root); \draw (3) to (root); \draw (3) to (4);
\node at (-90:1.7*\r){$\widehat{D}\setminus\{c_1\}$};
\end{tikzpicture}}

\newcommand{\reditintcasetwo}[1]{\begin{tikzpicture}
\def \r {#1};
\draw (0,0) circle (\r);
\foreach \k in {1,...,4} \node[circle,inner sep=1.5pt,fill=black] (\k) at (-90-360*\k/5:\r){};
\node at (-90-360*2/5:1.3*\r){$k$};
\node[circle,inner sep=1.5 pt,draw=black,fill=white] (root) at (-90:\r){};
\draw (1) to (root); \draw (2) to (root); \draw (2) to (3); \draw (2) to (4);
\node at (-90:1.7*\r){$D=\partial_k\widetilde{D}$};
\end{tikzpicture}}

\newcommand{\reditintcasetwobis}[1]{\begin{tikzpicture}
\def \r {#1};
\draw (0,0) circle (\r);
\foreach \k in {1,...,5} \node[circle,inner sep=1.5pt,fill=black] (\k) at (-90-360*\k/6:\r){};
\node at (-90-360*2/6:1.3*\r){};
\node[circle,inner sep=1.5 pt,draw=black,fill=white] (root) at (-90:\r){};
\draw (1) to (root); \draw (2) to (root); \draw (3) to (4); \draw (3) to (5);
\draw[very thick,bend left=20] (2) to (3);
\node at (-90-360/12-360*2/6:1.2*\r){$k$};
\node at (-90+360/12:1.2*\r){$l$};
\draw[very thick,bend left=20] (5) to (root);
\node at (-90:1.7*\r){$\widetilde{D}$};
\end{tikzpicture}}

\newcommand{\reditintcasetwoter}[1]{\begin{tikzpicture}
\def \r {#1};
\draw (0,0) circle (\r);
\foreach \k in {1,...,4} \node[circle,inner sep=1.5pt,fill=black] (\k) at (-90-360*\k/5:\r){};
\node at (-90-360*2/5:1.3*\r){$k$};
\node[circle,inner sep=1.5 pt,draw=black,fill=white] (root) at (-90:\r){};
\draw (1) to (root); \draw (2) to (root); \draw (3) to (root); \draw (3) to (4);
\node at (-90:1.7*\r){$\partial_l\widetilde{D}$};
\end{tikzpicture}}

\newcommand{\exappendixA}[1]{\begin{tikzpicture}
\def \n {11};
\def \r {#1};
\pgfmathsetmacro\nn{\n+1}; 
\draw[dotted] (0,0) circle (\r);
\foreach \k in {1,...,\n} \node[circle,inner sep=1.5pt,fill=black]  (\k) at (-90-360*\k/\nn:\r) {};
\node[circle,inner sep=1.5pt,draw=black,fill=white] (root) at (-90:\r) {};
\foreach \k in {1,...,\n} \node at (-90-360*\k/\nn:1.2*\r) {$\k$};
\draw[->,>=latex] (1) to (root); \draw[->,>=latex] (3) to (2); \draw[->,>=latex] (4) to (3); \draw[->,>=latex] (5) to (4); \draw[->,>=latex] (6) to (4); \draw[->,>=latex] (8) to (7); \draw[->,>=latex] (9) to (1); \draw[->,>=latex] (10) to (11);
\end{tikzpicture}}

\newcommand{\proofappendixA}[1]{\begin{tikzpicture}
\def \n {10};
\def \r {#1};
\pgfmathsetmacro\nn{\n+1}; 
\draw (0,0) circle (\r);
\foreach \k in {2,4,7,10} \node[circle,inner sep=1.5pt,fill=black]  (\k) at (-90-360*\k/\nn:\r) {};
\node[circle,inner sep=1.5pt,draw=black,fill=white] (root) at (-90:\r) {};
\node at (-90-360*10/\nn:1.2*\r){$m=m_0$};
\node at (-90-360*7/\nn:1.2*\r){$m_1$};
\node at (-90-360*4/\nn:1.2*\r){$m_2$};
\node at (-90-360*2/\nn:1.2*\r){$m_3$};
\draw[->,>=latex] (10) to (7); \draw[->,>=latex] (7) to (4); \draw[->,>=latex] (4) to (2);
\end{tikzpicture}}

\newcommand{\proofbisappendixA}[1]{\begin{tikzpicture}
\def \n {10};
\def \r {#1};
\pgfmathsetmacro\nn{\n+1}; 
\draw (0,0) circle (\r);
\node (0,0){$\widetilde{\Pi}(\alpha)$};
\foreach \k in {2,4,7,9} \node[circle,inner sep=1.5pt,fill=black]  (\k) at (-90-360*\k/\nn:\r) {};
\node[circle,inner sep=1.5pt,draw=black,fill=white] (root) at (-90:\r) {};
\node at (-90-360*9/\nn:1.2*\r){$i_{-1}$};
\node at (-90-360*7/\nn:1.2*\r){$i_0$};
\node at (-90-360*4/\nn:1.2*\r){$i_1$};
\node at (-90-360*2/\nn:1.2*\r){$i_2$};
\draw[->,>=latex] (9) to (7); \draw[->,>=latex] (7) to (4); \draw[->,>=latex] (4) to (2);
\end{tikzpicture}}

\begin{document}

%\begin{frontmatter}

\title{The combinatorial Hopf algebra of motivic dissection polylogarithms}
\author{Cl\'{e}ment Dupont}
\address{Institut de Math\'{e}matiques de Jussieu - Paris Rive Gauche\\4, place Jussieu \\75005 Paris, France}
\email{clement.dupont@imj-prg.fr}
%\ead{clement.dupont@imj-prg.fr}
%\address{Institut de Math\'{e}matiques de Jussieu-Paris Rive Gauche, 4 place Jussieu, 75005 Paris, France.}

\maketitle

\begin{abstract}
We introduce a family of periods of mixed Tate motives called dissection polylogarithms, that are indexed by combinatorial objects called dissection diagrams. The motivic coproduct on the former is encoded by a combinatorial Hopf algebra structure on the latter. This generalizes Goncharov's formula for the motivic coproduct on (generic) iterated integrals. Our main tool is the study of the relative cohomology group corresponding to a bi-arrangement of hyperplanes.
\end{abstract}

%\begin{keyword}
%mixed Tate motives\sep motivic periods \sep combinatorial Hopf algebras
%\MSC[2010] 05E99 \sep 14C15 \sep 14F42\sep 14N20 \sep 16T05
%\end{keyword}

%\end{frontmatter}

\section{Introduction}

	\subsection{Mixed Tate motives, periods, and combinatorial Hopf algebras}
	
	This article is part of a general attempt to understand mixed Tate motives and their periods through combinatorics.\\
	
	 Let $F$ be a number field and let $\mathrm{MTM}(F)$ be the category of mixed Tate motives over $F$ with coefficients in $\Q$ \cite{levinetatemotives}. It is a tannakian category, which means that it is equivalent to the category of finite-dimensional representations of an affine group scheme $G_{\mathrm{MTM}(F)}$ defined over $\Q$:
	 $$\mathrm{MTM}(F)\cong \mathrm{Rep}(G_{\mathrm{MTM}(F)}).$$
	 The motivic Galois group $G_{\mathrm{MTM}(F)}$ is abstractly defined by the tannakian formalism; a general programme is then to understand it and its action on mixed Tate motives in concrete terms. In practice, it is usually easier to work with the ring of functions $\H_{\mathrm{MTM}(F)}$ on (the pro-unipotent part of) $G_{\mathrm{MTM}(F)}$, which is a Hopf algebra defined over $\Q$ (see \S\ref{parMTM}). There are at least two motivations for such a programme.\\
	 The first motivation is to give explicit complexes to compute the rational algebraic $K$-theory of the field $F$. Indeed, the rational $K$-groups of $F$ are related to the Ext groups in the category $\mathrm{MTM}(F)$ \cite[1.6]{delignegoncharov}, and hence may be algebraically deduced from the Hopf algebra $\H_{\mathrm{MTM}(F)}$. This point of view has already appeared at many places in the literature, with the notions of \enquote{polylogarithmic complex} or \enquote{motivic complex} (see \cite{goncharovpolylogsarithmeticgeometry} for a survey).\\
	 The second motivation is related to the (conjectural) Galois theory of periods \cite{andrebook,andregaloismotivestransc}. More precisely, if we start with a period $p$ of $\mathrm{MTM}(F)$, one may lift it to a motivic period $p^\H$ which lives in the Hopf algebra $\H_{\mathrm{MTM}(F)}$ (see \S\ref{sectionmotivicperiods} for a discussion on the natural setting for motivic periods). The coproduct of $p^\H$, also referred to as the motivic coproduct of $p$, gives the (conjectural) action of the motivic Galois group $G_{\mathrm{MTM}(F)}$ on $p$. These ideas have been made popular by the pioneering work of A. B. Goncharov (\cite{goncharovmultiplepolylogs,goncharovgaloissym}, see also Deligne and Goncharov \cite{delignedroiteprojective,delignegoncharov}), who was able to compute the motivic coproduct of iterated integrals on the punctured complex line. In this framework, F. Brown was able to prove \cite{brownMTMZ} the Deligne-Ihara conjecture and the Hoffman basis conjecture for motivic multiple zeta values.\\
	 
	 In this article, we compute the motivic coproduct for a family of periods called dissection polylogarithms. These periods, which are indexed by combinatorial objects called dissection diagrams, generalize Goncharov's (generic) iterated integrals. We show that their motivic coproduct is related to a combinatorial Hopf algebra on dissection diagrams. This Hopf algebra is part of a growing family of Hopf algebras based on combinatorial objects, whose most famous representative is the Connes-Kreimer Hopf algebra (\cite{conneskreimer}, see \cite{lodayroncocombinatorial} for a tentative definition of the term \enquote{combinatorial Hopf algebra} and references).

	\subsection{Main results of this article}

The combinatorial objects that we consider are called \textit{dissection diagrams}. A dissection diagram of degree $n$ is a set of $n$ non-intersecting chords of a rooted oriented polygon (the polygons will always be drawn as circles) with $(n+1)$ vertices such that the graph formed by the chords is acyclic. See Figure \ref{introfigdd}\\

 \begin{figure}[h!]
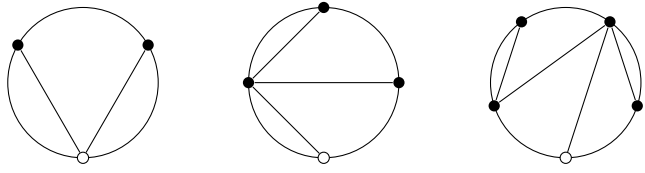

	\begin{center}\corolla{2}{1} \hspace{1cm}\dddegthreeexTen{1} \hspace{1cm}\dddegfourexintro{1}\end{center}
	\caption{Examples of dissection diagrams of respective degrees $2$, $3$, $4$. The polygons are drawn as circles. Every polygon has a distinguished vertex, which is called the root and is depicted as a white dot.}\label{introfigdd}\end{figure}	

Let $\D$ be the free commutative $\Q$-algebra generated by dissection diagrams, with a grading given by the degrees of the dissection diagrams. We give $\D$ the structure of a graded Hopf algebra. The coproduct $\Delta:\D\rightarrow \D\otimes\D$ is uniquely defined by its value on the dissection diagrams $D$, and is given by the formula
\begin{equation}\label{introformulacoproduct}
\Delta(D)=\sum_{C\subset\c(D)} \pm q_C(D)\otimes r_C(D).
\end{equation}
The terms in this formula are defined in \S\ref{sectionoperations}. For now let us just mention that $\c(D)$ denotes the set of chords of $D$, $q_C(D)$ is a product of dissection diagrams obtained by taking the quotient of $D$ by the chords in $C$ (contraction of chords), and $r_C(D)$ is a single dissection diagram obtained by keeping only the chords in $C$ (deletion of chords). The form of the coproduct is reminiscent of several other combinatorial Hopf algebras, such as the Connes-Kreimer Hopf algebra of rooted trees \cite{conneskreimer}.\\

There is a decorated version $\D(\C)$ of this Hopf algebra, where we attach complex numbers to each side of the polygon and each chord of the dissection diagram.

\begin{figure}[!h]
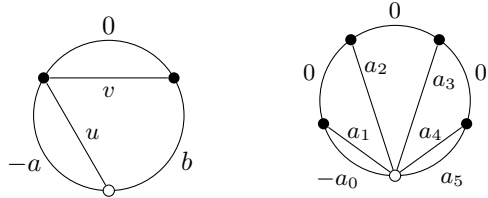

\begin{center}\decorateddddegreetwointro{1} \hspace{1cm} \decoratedcorollaFour{1}\end{center}
\caption{Decorated dissection diagrams. On the right, a decorated corolla.}\label{introfigdecorateddd}\end{figure}

 We mainly consider the Hopf subalgebra $\D^{gen}(\C)\subset\D(\C)$ generated by the decorated dissection diagrams that satisfy a genericity condition on the decorations. To each such generic decorated dissection diagram $D$, we associate an absolutely convergent integral
$$I(D)=\int_{\Delta_D}\omega_D$$
called a \textit{dissection polylogarithm}, where $\omega_D$ is a meromorphic form on $\C^n$ and $\Delta_D$ is a singular simplex in $\C^n$ that does not meet the polar locus of $\omega_D$.
For example, for $D$ the first decorated dissection diagram in Figure \ref{introfigdecorateddd}, we get 
$$\omega_D=\frac{1}{(2i\pi)^2}\frac{dx\wedge dy}{(x-u)(y-x-v)}$$
and $\Delta_D:\Delta^2\rightarrow\C^2$ is a continuous image of a triangle which is bounded by the lines $\{x=a\}$, $\{x=y\}$ and $\{y=b\}$. In general, the form $\omega_D$ is determined by the combinatorial data of the decorated chords and the domain $\Delta_D$ is determined by the decorations of the sides of the polygon (the integral $I(D)$ depends on the choice of $\Delta_D$, see \S\ref{defdissectionpolylogs} for a discussion).\\

The dissection polylogarithms generalize the (generic) iterated integrals on the punctured complex plane $\I(a_0;a_1,\ldots,a_n;a_{n+1})$ studied by Goncharov \cite{goncharovgaloissym}, which correspond to the special case of corollas (dissection diagrams where the chords are all linked to the root vertex, see Figure \ref{introfigdecorateddd}). In this special case, the genericity condition dictates that the decorations $a_i$ are pairwise distinct.\\

The geometric meaning of the integral $I(D)$ is the following. Let $L$ be the polar locus of $\omega_D$ and let $M$ be the Zariski closure of $\partial\Delta_D$; $L$ and $M$ are unions of hyperplanes inside $\C^n$. The dissection polylogarithm $I(D)$ is thus a period of the $\Q$-mixed Hodge structure
$$H(D)=H^n(\C^n\setminus L,M\setminus M\cap L)$$
which is of mixed Tate type (its weight graded pieces have type $(p,p)$).
If the decorations belong to a fixed number field $F$ embedded in $\C$, then $H(D)$ is the Hodge realization of a mixed Tate motive over $F$ denoted $\wt{H}(D)$. %According to the general philosophy, the motives $\wt{H}(D)$ should contain all the arithmetic information about dissection polylogarithms.
 Adding the extra data of the classes of $\omega_D$ and $\Delta_D$, we define a framed version
$$I^\H(D)=(H(D),[\omega_D],[\Delta_D])$$
which is an algebro-geometric avatar of the complex number $I(D)$. It shall be called the \textit{motivic dissection polylogarithm}.\\ 
This framed version (or more precisely, its equivalence class) naturally lives in a graded Hopf algebra $\H$, which is the fundamental Hopf algebra of the tannakian category of $\Q$-mixed Hodge-Tate structures (or, if we replace $H(D)$ by $\wt{H}(D)$, in the Hopf algebra $\H_{\mathrm{MTM}(F)}$ of mixed Tate motives over $F$). The main result of this article is the computation of the coproduct of the motivic dissection polylogarithms. More precisely, we show that they generate a Hopf subalgebra of $\H$ and that their coproduct can be computed combinatorially using formula (\ref{introformulacoproduct}).

\begin{thm}\label{intromaintheorem}[see Theorem \ref{maintheorem}]\\
Let $D$ be a generic decorated dissection diagram. Then the coproduct of the corresponding motivic dissection polylogarithm in $\H$ is given by formula (\ref{introformulacoproduct}):
$$\Delta(I^\H(D))=\sum_{C\subset\c(D)}\pm I^\H(q_C(D))\otimes I^\H(r_C(D)).$$
In other words, the morphism $$\D^{gen}(\C)\rightarrow \H \,\, , \,\, D\mapsto I^\H(D)$$ is a morphism of graded Hopf algebras.
\end{thm}

The particular case of iterated integrals has previously been worked out by Goncharov \cite{goncharovgaloissym} in the framework of motivic fundamental groupoids of the punctured complex line.\\

In order to prove Theorem \ref{intromaintheorem}, one needs to have a good understanding of the relative cohomology groups $H(D)$. We introduce the notion of bi-arrangements of hyperplanes and compute the corresponding relative cohomology groups in what we call the \textit{affinely generic} case (see \S\ref{paragraphrelativeBOS}). Our main technical tool is the following theorem.

\begin{thm}\label{introrelativeBOS}[see Theorems \ref{relativeBOS} and \ref{functoriality} for more precise statements]\\
Let $\{L_1,\ldots,L_l,M_1,\ldots,M_m\}$ be a set of hyperplanes of $\C^n$. We write $L=L_1\cup\cdots\cup L_l$ and $M=M_1\cup\cdots\cup M_m$ and we assume that $L\cup M$ is a normal crossing divisor inside $\C^n$. Then for every $k$ we have an explicit presentation
$$\gr^W_{2k}H(L;M)\cong \left(\Lambda^k(e_1,\ldots,e_l)\otimes\Lambda^{n-k}(f_1,\ldots,f_m)\right) / R_k(L;M).$$ 
This presentation is functorial in $(L;M)$.
\end{thm}

The functoriality statement in Theorem \ref{introrelativeBOS}, that is made more precise in Theorem \ref{functoriality}, is crucial. Indeed, it allows us to relate the geometric situation coming from $D$ and the geometric situations coming from the terms $q_C(D)$ and $r_C(D)$ in formula (\ref{introformulacoproduct}).\\

It has to be noted (Remark \ref{remdegenerateinfty}) that the configurations of hyperplanes $\{L_1,\ldots,L_l,M_1,\ldots,M_m\}$ that we are looking at are normal crossing divisors inside $\C^n$, but are highly degenerate at infinity when viewed inside $\mathbb{P}^n(\C)$. The affine context enables us to take products of configurations of hyperplanes, an operation which is more involved in the projective setting.

	\subsection{Organization of the article}

	In \S 2 we introduce the dissection diagrams and the Hopf algebra $\D$, as well as its decorated variants. This section is purely combinatorial and requires no special knowledge of algebraic geometry. In \S 3 we focus on bi-arrangements of hyperplanes and prove Theorem \ref{introrelativeBOS}. This section can be read independently from the rest of the article. In \S 4 we introduce the dissection polylogarithms $I(D)$ and discuss some of their algebraic relations. In \S 5 we define the motivic dissection polylogarithms $I^\H(D)$ and prove Theorem \ref{intromaintheorem}. Three appendices (\ref{appA}, \ref{appB}, \ref{appC}) are devoted to the proofs of technical lemmas.

	\subsection{Conventions and notation}
	
	\begin{enumerate}
	\item \textit{(Coefficients)} Unless otherwise stated, all vector spaces, algebras, and Hopf algebras are defined over $\Q$, as well as tensor products of such objects. All (mixed) Hodge structures are defined over $\Q$.
	\item \textit{(Cohomology)} The cohomology groups $H^\bullet(X)$ and relative cohomology groups $H^\bullet(X,Y)$ implicitly denote the singular cohomology groups with coefficients in $\Q$. We will simply write $H^\bullet(X)\otimes\C$ for the singular cohomology groups with $\C$-coefficients. If $X$ is a manifold, the latter are naturally isomorphic, via the de Rham isomorphism, to the (analytic) de Rham cohomology groups tensored with $\C$, hence we allow ourselves to use smooth differential forms as representatives for cohomology classes.
	\item \textit{(Signs)} If $I$ and $J$ are disjoint subsets of a linearly ordered set $\{1,\ldots,n\}$, we define a sign $\sgn(I,J)\in\{\pm 1\}$ as follows. In the exterior algebra on $n$ independent generators $x_1,\ldots,x_n$, we write $x_I=x_{i_1}\wedge\cdots\wedge x_{i_k}$ for $I=\{i_1<\cdots<i_k\}$. Then $\sgn(I,J)$ is defined by the equation $x_{I\sqcup J}=\sgn(I,J) x_I\wedge x_J$. For example we get $\sgn(\{i_r\},I\setminus\{i_r\})=(-1)^{r-1}$.
	\end{enumerate}
	
	\subsection{Connections with other articles}
	
	The use of combinatorial Hopf algebras in the theory of mixed Tate motives and $K$-theory has already appeared in references such as \cite{blochkriz} and \cite{ganglgoncharovlevin}. The integrals that we study are part of a family first studied by K. Aomoto \cite{aomotostructure,aomotoaddition}, and whose relationship with mixed Tate motives and $K$-theory was investigated in \cite{bvgs}. We are particularly indebted to the ideas of A. B. Goncharov on motivic iterated integrals and the article \cite{goncharovgaloissym}.

	\subsection{Acknowledgements}
	The author thanks Francis Brown for many discussions and helpful comments on a preliminary version of this text, and the anonymous referee for their suggestions. This work was partially supported by ERC grant 257638 \enquote{Periods in algebraic geometry and physics}.

\section{A combinatorial Hopf algebra on dissection diagrams}

	\subsection{The combinatorics of dissection diagrams}
	
		For every integer $n$ we consider a regular oriented $(n+1)$-gon $\Pi_n$ with a distinguished vertex called the \textit{root}. We draw the polygons as circles so that $\Pi_0$ and $\Pi_1$ also make sense, hence the \textit{sides} of $\Pi_n$ are drawn as arcs between two consecutive vertices. A \textit{chord} of $\Pi_n$ is a line between two distinct vertices.
	
		\begin{defi}
		A {\bf dissection diagram} of degree $n$ is a set of $n$ non-intersecting chords of $\Pi_n$ such that the graph formed by the chords is acyclic.
		\end{defi}
		
		In all the examples the polygons will be drawn with a clockwise orientation. The root will be drawn at the bottom as a white dot, whereas the non-root vertices will be drawn as black dots.\\
		Since there are $n$ chords and $(n+1)$ vertices, the graph formed by the chords is actually a tree that passes through all $(n+1)$ vertices; in other words, it is a spanning tree of the complete graph on the $(n+1)$ vertices of $\Pi_n$.\\
		All the dissection diagrams of degree $\leq 3$ are pictured in Figure \ref{figdissectiondiagrams}.
		
		\begin{figure}[h]
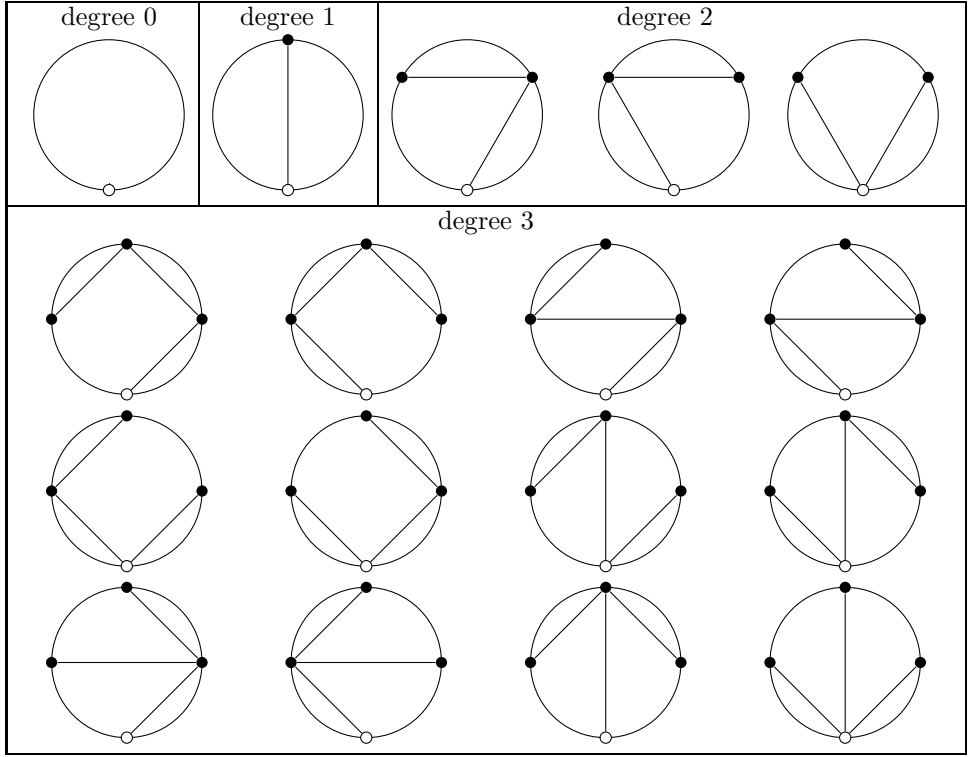

		\begin{center}\tabledissectiondiagrams{1}\end{center}
		\caption{The dissection diagrams of degree $\leq 3$}\label{figdissectiondiagrams}
		\end{figure}
		
		\begin{lem}
		The number of dissection diagrams of degree $n$ is $$d_n=\frac{1}{2n+1}\binom{3n}{n}.$$
		\end{lem}
		
		\begin{proof}
		We will not use this lemma in the rest of this article, so we just give a sketch of the proof. The sequence $(d_n)_{n\geq 0}$ counting the dissection diagrams in each degree satisfies the recurrence relation, for $n\geq 1$:
		\begin{equation}\label{eqrecdn}
		d_n=\sum_{\substack{i_1,i_2,i_3\geq 0 \\ i_1+i_2+i_3=n-1}}d_{i_1}d_{i_2}d_{i_3}.
		\end{equation}
		The reason for this recurrence relation is that a dissection diagram $D$ of degree $n$ is uniquely determined by a triple $(D_1,D_2,D_3)$ of dissection diagrams of respective degrees $(i_1,i_2,i_3)$ such that $i_1+i_2+i_3=n-1$.
		\begin{center}\isoternarytrees{1.8}\end{center}
		In the above picture, $\rho$ is the first (in clockwise order, starting at the root) non-root vertex of $\Pi_n$ that is attached to the root by a chord of $D$. Let $c$ be the chord between $\rho$ and the root.\\ 
		The $i_1$ chords that are on the left-hand side of $c$ form a rooted tree with $i_1$ internal vertices, whose root is $\rho$. Since the internal vertices of this tree are all on the polygon $\Pi_n$, we may view it as a dissection diagram $D_1$ of degree $i_1$.\\
		The $(n-1-i_1)$ chords that are on the right-hand side of $c$ form two connected components: one of cardinality $i_2$ that is attached to $\rho$, the other of cardinality $i_3$ that is attached to the root of $\Pi_n$. In the same fashion as above, we get dissection diagrams $D_2$ and $D_3$ of respective degrees $i_2$ and $i_3$, with $i_2+i_3=n-1-i_1$.\\
		Now let $d(x)=\sum_{n\geq 0}d_nx^n$ be the ordinary generating series for the enumeration of dissection diagrams.  
		The recurrence relation (\ref{eqrecdn}), together with $d_0=1$, implies the functional equation $d(x)=1+xd(x)^3$. Thus the Lagrange inversion formula \cite[Theorem 5.4.2]{stanleyvol2} applied to $f(x)=d(x)-1$ gives the result.
		\end{proof}
		
		\begin{rem}
		It is well-known that $d_n$ is also the number of ternary trees (planar rooted trees in which every internal vertex has exactly $3$ incoming edges) with $n$ internal vertices. The proof goes along the same lines: a planar ternary rooted tree $T$ is completely determined by its subtrees $T_1$, $T_2$, $T_3$ attached to the root. 
		\begin{center}\ternarytree \end{center}
		Thus we may recursively build a bijection between ternary trees and dissection diagrams. %For example, the three dissection diagrams of degree $2$ from the above table respectively correspond to the three ternary trees with $2$ internal vertices.
		%\begin{center}\ternarytreeleftdegreetwo\hspace{1cm}\ternarytreemiddegreetwo\hspace{1cm}\ternarytreerightdegreetwo \end{center}
		\end{rem}

		\vspace{.15cm}
		
		We let $\D$ be the free commutative unital algebra (over $\Q$) on the set of dissection diagrams of positive degree. The degrees of the dissection diagrams induce a grading
		$$\D=\bigoplus_{n\geq 0}\D_n$$
		on $\D$.
		The unit $1$ of $\D$ will be identified with the dissection diagram $\vc{\corolla{0}{.3}}$ of degree $0$.\\ 
		In small degree, we have	
		\begin{center} $\D_0=\Q $ \hspace{1cm}  $\D_1=\Q\,\vc{\corolla{1}{0.3}}$ \hspace{1cm} $\D_2=\Q\,\vc{\dddegtworight{0.3}}\oplus\Q\,\vc{\dddegtwoleft{0.3}}\oplus\Q\,\vc{\corolla{2}{0.3}}\oplus\Q\,\vc{\corolla{1}{.3}\,\corolla{1}{0.3}}$\end{center}
		where $\vc{\corolla{1}{.3}\,\corolla{1}{0.3}}$ represents the square of the only dissection diagram $\vc{\corolla{1}{.3}}$ of degree $1$. For every $n\geq 0$, $\D_n$ is a finite-dimensional vector space.\\
		
			\paragraph{\textit{Conventions on dissection diagrams}}
			
			We introduce some labeling conventions on dissection diagrams. An example is shown in Figure \ref{figconventions}.\\
			
			The non-root vertices of $\Pi_n$ are labeled $1,\ldots,n$ following the orientation, $1$ being just after the root. The sides of $\Pi_n$ are labeled $0,1,\ldots,n$ in such a way that the side labeled $0$ is between the root and the vertex $1$. This side plays a special role in the sequel and is called the \textit{root side}. The other sides are called the \textit{non-root sides}: for $i=1,\ldots,n-1$, the side labeled $i$ is between the vertices $i$ and $i+1$, and the side labeled $n$ is between the vertex $n$ and the root.\\
			
		In a dissection diagram of degree $n$,  the $n$ chords form a spanning tree of the complete graph on the $(n+1)$ vertices of $\Pi_n$. There is thus a preferred orientation of all the chords, towards the root. We may then label the chords with $1,\ldots,n$ such that the chord labeled $i$ leaves the vertex labeled $i$.\\
		
		The sides of $\Pi_n$ are also implicitly oriented following the orientation of $\Pi_n$ (clockwise, in all our figures). Thus when we consider the $(n+1)$ sides of $\Pi_n$ together with the $n$ chords of a dissection diagram $D$, we get a directed graph with $(n+1)$ vertices and $(2n+1)$ edges that is denoted $\Gamma(D)$ and called the \textit{total directed graph} of $D$.
			
		\begin{rem}
		Even though we will not always include them in the pictures, the orientations of the sides and the chords, as well as the labelings of the vertices, sides and chords of a dissection diagram are implicit.
		\end{rem}	
			
		In the sequel it will be more convenient to consider dissection diagrams $D$  where the chords are labeled by some abstract set $\mathscr{C}(D)$ of cardinality $n$, and the sides of the polygon by some other abstract set $\mathscr{S}(D)$ of cardinality $(n+1)$, which are both linearly ordered.\\
		If we set $\mathscr{S}^+(D)=\mathscr{S}(D)\setminus\{\min(\mathscr{S(D)})\}$ for the set of non-root sides, the linear orderings give bijections
		\begin{equation}\label{conventionbijection}
		\mathscr{C}(D)\simeq\{1,\ldots,n\}\simeq\mathscr{S}^+(D).
		\end{equation}
		
		\begin{rem}
		When the context is clear, we will drop the dissection diagram $D$ from the notation and simply write $\mathscr{C}$, $\mathscr{S}$, $\mathscr{S}^+$, $\Gamma$.
		\end{rem}
		
		\begin{figure}[h]
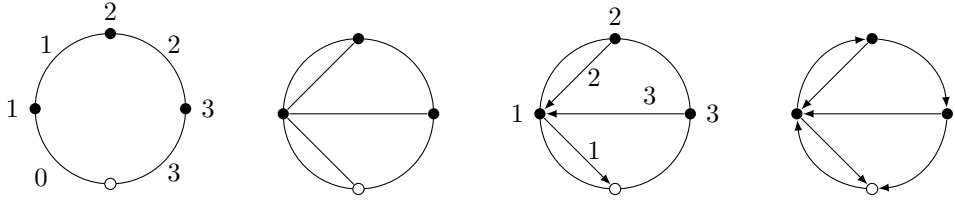

		\begin{center}\dddegthreelabelingpolygons{1}\hspace{.6cm} \dddegthree{1} \hspace{.6cm} \dddegthreelabelingarrows{1} \hspace{.6cm} \dddegthreedirected{1} \end{center}
		\caption{Conventions on dissection diagrams. Labeling the vertices and the sides of $\Pi_3$; a dissection diagram $D$ of degree $3$; the natural orientation and labeling of the chords of $D$; the total directed graph $\Gamma(D)$ of $D$.}\label{figconventions}\end{figure}

	\subsection{Operations on dissection diagrams}\label{sectionoperations}
	
		Let $D$ be a dissection diagram of degree $n$. We fix a subset $C\subset\c$ of chords of $D$. We introduce the notations $\s_C^+$, $q_C(D)$, $r_C(D)$, $\k_C(D)$ and $k_C(D)$ that will allow us to make sense of formula (\ref{defdelta}) below for the coproduct in $\D$. The reader may refer to Figure \ref{figexamplesoperations} for a special case.\\
		
		\paragraph{\textit{The set $\s_C^+\subset\s^+$}}	\hspace*{\fill}\\
		
		We first define a subset $\s_C^+\subset\s^+$ of the non-root sides of $D$, of cardinality $n-|C|$. It plays a sort of \enquote{dual role} to $C$, see Proposition \ref{propbasis}. In some simple cases (see Example \ref{excoproduct} below), $\s_C^+$ will simply be the complement $\overline{C}$ of $C$ in $\s^+$, using the identification (\ref{conventionbijection}).\\
		The planar graph $C\cup\s$ has $|C|+1$ faces. Each such face $\alpha$ is the interior of a polygon that we denote $\wt{\Pi}(\alpha)$, whose sides are sides of $\Pi_n$ and chords of $D$. If we denote by $\s_C(\alpha)$ the set of sides of $\Pi_n$ that are sides of $\wt{\Pi}(\alpha)$, we get a partition 
		\begin{equation}\label{partition}
		\s=\bigsqcup_\alpha \s_C(\alpha).
		\end{equation}
		
		\begin{lem}\label{lemSTacyclic}
		Let $J\subset\s$ be a subset of edges of $\Pi_n$, with $|C|+|J|=n$. Then the undirected graph $C\cup J$ is acyclic if and only if $J$ has the form
		$$J=\bigsqcup_\alpha J(\alpha) \,\textnormal{ \textit{with} }\, J(\alpha)=\s_C(\alpha)\setminus\{u_\alpha\}$$ for some choice of $u_\alpha\in\s_C(\alpha)$.
		\end{lem} 		
		
		\begin{proof}
		Let us write $J=\sqcup_\alpha J(\alpha)$ with $J(\alpha)\subset\s_C(\alpha)$. If there exists an $\alpha$ such that $J(\alpha)=\s_C(\alpha)$ then $C\cup J$ contains the whole boundary of the polygon $\wt{\Pi}(\alpha)$, hence contains a cycle. Hence if $C\cup J$ is acyclic, then all the inclusions $J(\alpha)\subset\s_C(\alpha)$ are strict. Since $|J|=n-|C|$, we necessarily have $|J(\alpha)|=|\s_C(\alpha)|-1$ for each $\alpha$, hence  $J(\alpha)=\s_C(\alpha)\setminus\{u_\alpha\}$. We leave it to the reader to show that in that case, $J\cup C$ is indeed acyclic.
		\end{proof}
		
		Let us set
		$$\s_C^+=\bigsqcup_\alpha \s_C^+(\alpha)  \,\textnormal{ with }\, \s_C^+(\alpha)=\s_C(\alpha)\setminus\{\min(\s_C(\alpha))\}.$$
		It is a subset of $\s^+$ and has cardinality $n-|C|$.

		 Let $\overline{C}=\c\setminus C$ denote the set of the chords of $D$ which are not in $C$. Since the chords do not intersect each other, we have a partition $$\overline{C}=\bigsqcup_\alpha \overline{C}(\alpha)$$ where $\overline{C}(\alpha)$ is the set of chords of $D$ which are inside the polygon $\wt{\Pi}(\alpha)$.\\
		 It is clear that $\s_C^+(\alpha)$ and $\overline{C}(\alpha)$ have the same cardinality $|\s_C^+(\alpha)|=|\overline{C}(\alpha)|=n(\alpha)$, with $\sum_\alpha n(\alpha)=n-|C|$.
		 
		 \begin{rem}
		 Despite the notation, $\s_C^+$ does not depend only on $C$ but also on the dissection diagram $D$.
		 \end{rem}
		 
		 \begin{ex}\label{exoperations}
		 Let us focus on the dissection diagram $D$ of Figure \ref{figconventions} and put $C=\{3\}$ consisting only of the horizontal chord.
		 \begin{center}\dddegthreeexample{1} \end{center}
		 Then the partition $\s=\bigsqcup_\alpha\s_C(\alpha)$ is $\{0,1,2,3\}=\{1,2\}\sqcup\{0,3\}$
		 hence we get $\s_C^+=\{2\}\sqcup\{3\}=\{2,3\}$. The corresponding partition of $\overline{C}=\{1,2\}$ is $\overline{C}=\{2\}\sqcup\{1\}$.
		 \end{ex}
		
		\paragraph{\textit{The dissection diagrams $q_C^\alpha(D)$ and their product $q_C(D)$}}	\hspace*{\fill}\\
		
		Starting from the dissection diagram $D$, let us contract the chords from $C$. The resulting picture is a \enquote{cactus} of dissection diagrams glued together. These dissection diagrams are denoted by $q_C^\alpha(D)$ and we write $$q_C(D)=\prod_\alpha q_C^\alpha(D)$$ for their product in $\D$.\\
		
		More precisely, let us consider an individual polygon $\wt{\Pi}(\alpha)$ and contract all its sides that are chords of $C$. We get a polygon $\Pi(\alpha)$ that is naturally oriented. The dissection diagram $q_C^\alpha(D)$ naturally lives in $\Pi(\alpha)$. The set of its non-root sides is $\s_C^+(\alpha)$ and the set of its chords is $\overline{C}(\alpha)$. The degree of $q_C^\alpha(D)$ is $n(\alpha)$, hence the degree of $q_C(D)$ is $\sum_\alpha n(\alpha)=n-|C|$.\\
		
		Let us recall that we identify the dissection diagram $\vc{\corolla{0}{0.3}}$ of degree $0$ with the unit $1$ of $\D$, so that we do not write the dissection diagrams $q_C^\alpha(D)$ of degree $n(\alpha)=0$ in the product $q_C(D)$.
		
		\begin{ex}\label{exQ}
		We come back to the dissection diagram $D$ of degree $3$ from Example \ref{exoperations} with $C=\{3\}$. Contracting the horizontal chord labeled $3$ gives the picture
		$$\vc{\dddegthreeexampleQ{1}}
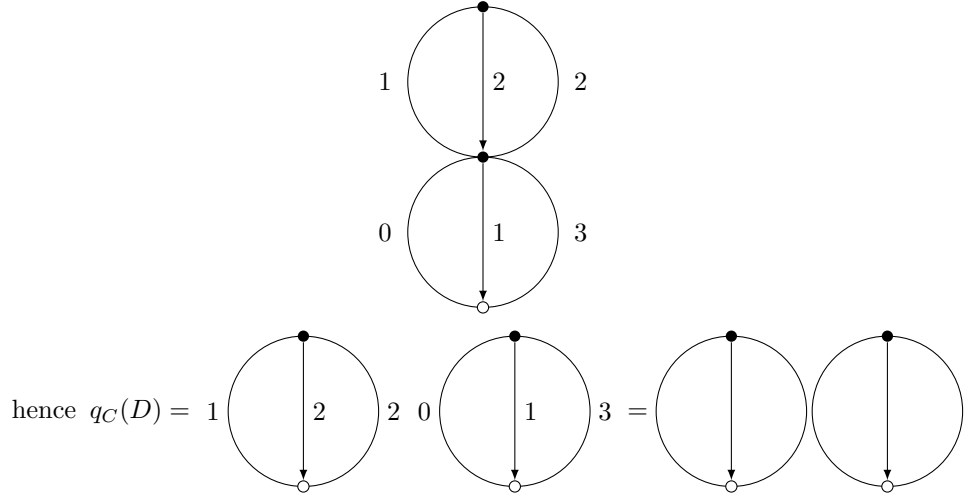
$$
		$$\textnormal{ hence } \, q_C(D)=\vc{\dddegthreeexampleQbis{1}}\vc{\dddegthreeexampleQter{1}}
		= \vc{\directedcorolla{1}{1}}\,\vc{\directedcorolla{1}{1}}$$
		is the square of the dissection diagram of degree $1$.
		\end{ex}
		
		\paragraph{\textit{The dissection diagram $r_C(D)$; the set $\k_C(D)$ and its cardinality $k_C(D)$}}	\hspace*{\fill}\\		
		
		Going back to the initial dissection diagram $D$, let us look at the graph obtained by keeping only the chords from $C$ and contracting the sides from $\s_C^+$. By Lemma \ref{lemSTacyclic}, this process does not lead to cycles between the chords from $C$ and hence gives a dissection diagram whose set of chords is $C$ and whose set of non-root sides is $\overline{\s_C^+}=\s^+\setminus \s_C^+$. We call this dissection diagram $r_C(D)$. Its degree is $|C|$.\\
		
		It has to be noted that in general the directions of the chords in $r_C(D)$ may differ from the directions of the chords in $D$. We let $\k_C(D)\subset C$ be the set of these chords that one has to flip in the process of computing $r_C(D)$, and write $k_C(D)=|\k_C(D)|$ for its cardinality.
		
		\begin{ex}\label{exR}
		We come back to the dissection diagram $D$ of degree $3$ from Example \ref{exoperations} with $C=\{3\}$.  Keeping only the horizontal chord labeled $3$ gives the picture
		\begin{center}\dddegthreeexampleR{1} \end{center}
		and hence contracting the sides from $\s_C^+=\{2,3\}$ gives the picture
		$$r_C(D)=\vc{\dddegthreeexampleRbis{1}}=\vc{\directedcorolla{1}{1}}$$
		hence $r_C(D)$ is (unsurprisingly) the dissection diagram of degree $1$. Since in the above picture we had to flip the chord labeled $3$, we get $\k_C(D)=\{3\}$ and $k_C(D)=1$.
		\end{ex}
		
		\vspace{.1cm}
		
		\begin{figure}[h]\dddegthreeexamplefinal{1} \hspace{.2cm} \examplestable{.5}\caption{The computations of $\s_C^+$, $q_C(D)$, $r_C(D)$, $\k_C(D)$ and $k_C(D)$ for the dissection diagram $D$ from Example \ref{exoperations}}\label{figexamplesoperations}\end{figure}
		
	\subsection{Definition of the Hopf algebra}
	
		We define a map $$\Delta:\D\rightarrow \D\otimes\D$$ by setting
		\begin{equation}\label{defdelta} \Delta(D)=\sum_{C\subset\c(D)} (-1)^{k_C(D)}q_C(D)\otimes r_C(D)
		\end{equation}
		for $D$ a dissection diagram, and extending it to all of $\D$ as a morphism of algebras.\\
		For a dissection diagram $D$ of degree $n$, $q_C(D)$ has degree $n-|C|$ and $r_C(D)$ has degree $|C|$. Thus the coproduct $\Delta$ is compatible with the grading of $\D$, with components
		$$\Delta_{n-k,k}:\D_n\rightarrow\D_{n-k}\otimes\D_k$$
		corresponding to the subsets $C\subset\c(D)$ of cardinality $k$.\\
		
		For $C=\c(D)$ we get $\s_C=\varnothing$, $q_C(D)=1$, $r_C(D)=D$, $\k_C(D)=\varnothing$ and $k_C(D)=0$, hence the corresponding term in formula (\ref{defdelta}) is $\Delta_{0,n}(D)=1\otimes D$. For $C=\varnothing$, we get the term $\Delta_{n,0}(D)=D\otimes 1$.
		
		\begin{prop}\label{prophopfalgebra}
		Formula (\ref{defdelta}) gives $\D$ the structure of a graded connected commutative Hopf algebra.
		\end{prop}
		
		\begin{proof}
		All there is to prove is that $\Delta$ is coassociative, since it is well-known that given a graded connected bialgebra there exists a unique antipode that makes it into a Hopf algebra. Let us fix a dissection diagram $D$ of degree $n$ and prove that $(\mathrm{id}\otimes\Delta)(\Delta(D))=(\Delta\otimes\mathrm{id})(\Delta(D))$.\\
		On the one hand we have $$(\mathrm{id}\otimes\Delta)(\Delta(D))=\sum_{C\subset C'\subset\c(D)} (-1)^{k_{C'}(D)+k_C(r_{C'}(D))}q_{C'}(D)\otimes q_C(r_{C'}(D))\otimes r_C(r_{C'}(D)).$$
		On the other hand we have $$(\Delta\otimes\mathrm{id})(\Delta(D))=\sum_{C\subset\c(D)}(-1)^{k_C(D)}\Delta(q_C(D))\otimes r_C(D).$$
		Let us recall that $q_C(D)=\prod_\alpha q_C^\alpha(D)$. For a given $\alpha$, the set of chords of $q_C^\alpha(D)$ is $\overline{C}(\alpha)$, hence  
		$$\Delta(q_C^\alpha(D))=\sum_{C'_\alpha\subset\overline{C}(\alpha)}(-1)^{k_{C'_\alpha}(q_C^\alpha(D))}q_{C'_\alpha}(q_C^\alpha(D))\otimes r_{C'_\alpha}(q_C^\alpha(D)).$$
		Let us perform the change of summation indices $C'=C\sqcup\bigsqcup_\alpha C'_\alpha$. The result then follows from the following lemma.
		\end{proof}

		\begin{lem}\label{lemhopfalgebra}
		\begin{enumerate}
		\item $q_{C'}(D)=\prod_\alpha q_{C'_\alpha}(q_C^\alpha(D))$.
		\item $q_C(r_{C'}(D))=\prod_\alpha r_{C'_\alpha}(q_C^\alpha(D))$.
		\item $r_C(r_{C'}(D))=r_C(D)$.
		\item $k_{C'}(D)+k_C(r_{C'}(D))=k_C(D)+\sum_\alpha k_{C'_\alpha}(q_C^\alpha(D))$.
		\end{enumerate}
		\end{lem}
		
		\begin{proof}
		See \ref{appA}.
		\end{proof}
		
		\begin{ex}
		We may use the computations of Figure \ref{figexamplesoperations} in order to get \exampleformulacoproduct{.5}
		\end{ex}
		
		\begin{ex}\label{excoproduct}
		\begin{enumerate}
		\item For all $n\geq 0$ let $X_n$ be the dissection diagram of degree $n$ (\enquote{corolla}, see Figure \ref{figcorollapathtree}) with all chords pointing towards the root, with the convention $X_0=1$.\\
		 Then the formula for the coproduct is:
		$$\Delta(X_n)=\sum_{k=0}^n\left(\sum_{i_0+\cdots+i_k=n-k}X_{i_0}\cdots X_{i_k}\right)\otimes X_k.$$	
		Indeed, for a subset $C=\{i_0+1,i_0+i_1+2,\ldots,i_0+i_1+\cdots+i_{k-1}+k\}$, we get $\s_C^+=\overline{C}$, $q_C(X_n)=X_{i_0}\cdots X_{i_k}$, $r_C(X_n)=X_k$, $\k_C(X_n)=\varnothing$ and $k_C(X_n)=0$.\\
		For instance we get
		$$\Delta(X_3)=1\otimes X_3+3X_1\otimes X_2+(2X_2+X_1X_1)\otimes X_1+X_3\otimes 1.$$
		\item For all $n\geq 0$, let $Y_n$ be the dissection diagram of degree $n$ (\enquote{path tree}, see Figure \ref{figcorollapathtree}) consisting of the chords between $1$ and $2$, $2$ and $3$, $\ldots$, $(n-1)$ and $n$, $n$ and the root, with the convention $Y_0=1$. Then the formula for the coproduct is:
		$$\Delta(Y_n)=\sum_{k=0}^n \binom{n}{k}Y_{n-k}\otimes Y_k.$$
		Indeed, for any subset $C\subset\{1,\ldots,n\}$ of cardinality $k$, we get $\s_C^+=\overline{C}$, $q_C(Y_n)=Y_{n-k}$, $r_C(Y_n)=Y_k$, $\k_C(Y_n)=\varnothing$ and $k_C(Y_n)=0$.\\
		The above formula is reminiscent of the formula for the coproduct in the Hopf algebra $\Q[t]$ of functions on the additive group $\mathbb{G}_a$.
		\end{enumerate}		
		\end{ex}
		
		\begin{figure}[h]
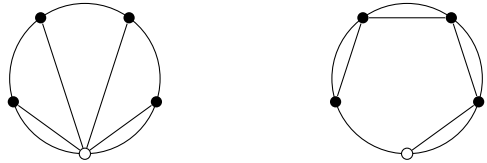
\begin{center}\corolla{4}{1} \hspace{2cm} \pathtreeFour{1}
		\end{center}	\caption{The corolla $X_4$ and the  path tree $Y_4$.}\label{figcorollapathtree}\end{figure}
		
		\begin{rem}
		The Hopf algebra $\D$ is a right-sided combinatorial Hopf algebra in the sense of \cite[5.7]{lodayroncocombinatorial}. According to Theorem 5.8 of \cite{lodayroncocombinatorial}, there is thus a structure of graded pre-Lie algebra on the free vector space spanned by dissection diagrams of positive degree (more precisely, its graded dual). It would be interesting to know if this pre-Lie structure has a simple presentation.
		\end{rem}
		
		\paragraph{\textit{A family of Hopf algebras}}
		
		Let $x$ be a fixed rational number. If one changes formula (\ref{defdelta}) to
		\begin{equation}\label{defdeltalambda} 
		\Delta^{(x)}(D)=\sum_{C\subset\c(D)} x^{k_C(D)}q_C(D)\otimes r_C(D)
		\end{equation}
		then the proof of Proposition \ref{prophopfalgebra} (replace $-1$ by $x$) shows that this defines a (graded connected commutative) Hopf algebra $\D^{(x)}$.\\
		Apart from the choice $x=-1$ which gives back $\D^{(-1)}=\D$, there are two other natural choices: for $x=1$ there is no sign in the formula; for $x=0$ (with the convention $0^0=1$) there is no sign and the sum is restricted to the subsets $C$ with $k_C(D)=0$. The formulas of Example \ref{excoproduct} are valid for any choice of $x$ since we always have $k_C(D)=0$.\\
		%It would be interesting to know if the Hopf algebras $\D^{(x)}$ appear naturally as in Theorem \ref{maintheorem}. Another question is to classify them in terms of the parameter $x$.\\
		We may also consider $x$ as a formal parameter and view formula (\ref{defdeltalambda}) as a map of $\Q[x]$-algebras
		\begin{equation}
		\Q[x]\otimes\D\rightarrow \Q[x]\otimes\D\otimes\D\cong\left(\Q[x]\otimes\D\right)\otimes_{\Q[x]}\left(\Q[x]\otimes\D\right)
		\end{equation}
		 given by 
		$$D\mapsto \sum_{C\subset\c(D)}x^{k_C(D)}\otimes q_C(D)\otimes r_C(D).$$ 
		In terms of algebraic geometry, we get an algebraic family of affine group schemes parametrized by the affine line
		$$\mathrm{Spec}(\Q[x]\otimes\D)=\mathbb{A}^1\times\mathrm{Spec}(\D)\rightarrow \mathbb{A}^1=\mathrm{Spec}(\Q[x])$$
		with constant underlying scheme $\mathrm{Spec}(\D)$.
	
	\subsection{Decorations on dissection diagrams}
	
		In this paragraph we fix an abelian group $\Lambda$. We define a decorated version $\D(\Lambda)$ of the Hopf algebra $\D$.
	
		\subsubsection{Decorated directed graphs}\label{conventions}		
		
		Let $\Gamma$ be a directed graph. A \textit{$\Lambda$-decoration} on $\Gamma$ is the data of an element of $\Lambda$ for each edge of $\Gamma$. While performing operations on directed graphs, we will always keep in mind the two following rules for the decorations:
		\begin{itemize}
		\item Let us flip an edge, i.e. change its direction. We then multiply its decoration by $-1$.
		
		\begin{center}\changedirection \end{center}
		
		\item Let us contract an edge going from a vertex $v_-$ to a vertex $v_+$ which is decorated by an element $\alpha\in\Lambda$. For any edge of $\Gamma$ going to $v_-$, we replace its decoration $x$ by the decoration $x+\alpha$; for any edge of $\Gamma$ leaving from $v_-$, we replace its decoration $y$ by the decoration $y-\alpha$. The other decorations (including the decorations of the edges that touch $v_+$) stay unchanged.
		
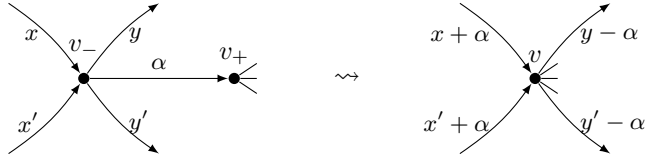
\vspace{.2cm}
		\begin{center}\contraction \end{center}
		\vspace{.1cm}
		We leave it to the reader to check that if one contracts a set of edges (that does not contain a loop), the resulting decorated graph does not depend on the order in which we perform the contractions.
		\end{itemize}

		\subsubsection{Decorated dissection diagrams and the Hopf algebra $\D(\Lambda)$}
		
		\begin{figure}[h]
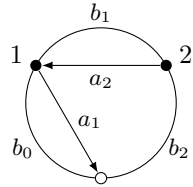
 \begin{center}\decorateddddegreetwo{1}\end{center}\caption{A decorated dissection diagram of degree $2$}\label{figdecorateddddegreetwo}\end{figure}
		
			A $\Lambda$-\textit{decorated dissection diagram} of degree $n$ is a dissection diagram $D$ of degree $n$ together with a $\Lambda$-decoration on the total directed graph $\Gamma(D)$. For $i=1,\ldots,n$, we denote by $a_i\in\Lambda$ the decoration of the chord $i$, and for $j=0,\ldots,n$, we denote by $b_j\in\Lambda$ the decoration of the side $j$ (see Figure \ref{figdecorateddddegreetwo}). We use the same letter to denote the decorated dissection diagram and its underlying dissection diagram obtained by forgetting the decorations.\\
			
  			We let $\D(\Lambda)$ be the free commutative unital algebra (over $\Q$) on the set of $\Lambda$-decorated dissection diagrams of positive degree. If $\Lambda=0$ then we recover $\D(0)=\D$. We want to generalize the Hopf algebra structure on $\D$ to all the $\D(\Lambda)$'s.\\
  			We define the coproduct $$\Delta:\D(\Lambda)\rightarrow\D(\Lambda)\otimes\D(\Lambda)$$ as in $\D$ by formula (\ref{defdelta}). The terms $q_C(D)$ and $r_C(D)$ are understood through the conventions of \S \ref{conventions}.
  			
  			\begin{ex}\label{exdecoratedQR}
			Let $D$ be a decorated dissection diagram of degree $3$ whose underlying dissection diagram is the one from Examples \ref{exoperations}, \ref{exQ}, \ref{exR}:
			\begin{center}\dddegthreeexampledecorated{1}\end{center}
			Then for $C=\{3\}$ we get
			$$q_C(D)=\hspace{.1cm}\vc{\dddegonedirecteddecorated{1}{\small{$a_2$}}{\small{$b_1$}}{\small{$b_2$+$a_3$}}}\hspace{.15cm}\vc{\dddegonedirecteddecorated{1}	{\small{$a_1$}}{\small{$b_0$}}{\small{$b_3$-$a_3$}}}$$
			and 
			$$r_C(D)=\hspace{.1cm}\vc{\dddegonedirecteddecoratedwrong{1}{\small{$a_3$-$b_3$}}{\small{$b_0$}}{\small{$b_1$+$b_2$+$b_3$}}}\hspace{.1cm}=\hspace{.1cm}\vc{\dddegonedirecteddecorated{1}{\small{$b_3$-$a_3$}}{\small{$b_0$}}{\small{$b_1$+$b_2$+$b_3$.}}} $$
			The last equality is the application of the convention related to flipping an edge (\S\ref{conventions}).
			\end{ex}
  			
  			We leave it to the reader to check that properties $1.$, $2.$ and $3.$ from Lemma \ref{lemhopfalgebra} remain true, as well as property $4.$, which is independent of the decorations. Hence the proof of Proposition \ref{prophopfalgebra} can be copied word for word and gives the following extension.
		
			\begin{prop}\label{prophopfalgebradec}
			For any abelian group $\Lambda$, formula (\ref{defdelta}) gives $\D(\Lambda)$ the structure of a graded connected commutative Hopf algebra. Moreover, for any morphism $\Lambda\rightarrow\Lambda'$ of abelian groups, the corresponding morphism $\D(\Lambda)\rightarrow \D(\Lambda')$ is a morphism of Hopf algebras. In other words, $\Lambda\leadsto\D(\Lambda)$ is a functor from the category of abelian groups to the category of Hopf algebras.
			\end{prop}
		
			\begin{rem}
			The variant of formula (\ref{defdeltalambda}) remains valid with decorations.
			\end{rem}
		
		\subsubsection{Generic decorations and the Hopf algebra $\D^{gen}(\Lambda)$}		
		
			Let $\Gamma$ be a directed graph. A \textit{simple cycle} of $\Gamma$ is an undirected cycle in $\Gamma$ that does not pass twice through the same vertex. For a given simple cycle in $\Gamma$, the \textit{total decoration} of the simple cycle is the signed sum of the decorations in the cycle, the sign being $+1$ if and only if the direction of the edge agrees with the direction of the cycle. We say that a $\Lambda$-decoration on $\Gamma$ is \textit{generic} if for every simple cycle in $\Gamma$, the total decoration of the cycle is non-zero.
		
			We say that a $\Lambda$-decorated dissection diagram is \textit{generic} if the $\Lambda$-decoration on $\Gamma(D)$ is generic. For example, the decorated dissection diagram from Figure \ref{figdecorateddddegreetwo} is generic if and only if the quantities $b_0+b_1+b_2$, $b_1+b_2-a_1$, $b_0-a_2+b_2$, $b_0+a_1$, $b_1+a_2$, $b_2-a_1-a_2$ are all $\neq 0$.\\
		
  			We leave it to the reader to check that the operations of reversal of arrows and contraction of \S \ref{conventions} preserve the genericity condition. As a consequence, the generic $\Lambda$-decorated dissection diagrams of positive degree generate a Hopf subalgebra 
  			$$\D^{gen}(\Lambda)\hookrightarrow\D(\Lambda).$$
  			The functoriality assertion of Proposition \ref{prophopfalgebradec} is valid for the Hopf algebras $\D^{gen}(\Lambda)$ if we restrict to \textit{injective} morphisms $\Lambda\hookrightarrow\Lambda'$.

\section{Bi-arrangements of hyperplanes and relative cohomology}

	After recalling some classical results on arrangements of hyperplanes, we introduce and study bi-arrangements of hyperplanes, focusing on the affinely generic case. The systematic study of bi-arrangements of hyperplanes and the corresponding relative cohomology groups will appear in a subsequent article.

	\subsection{Affinely generic arrangements of hyperplanes}
	
	Let $L=\{L_1,\ldots,L_l\}$ be an arrangement of hyperplanes in $\C^n$. The hyperplanes do not necessarily pass through the origin. As the notation suggests, the set $L$ is implicitly linearly ordered. We will use the same letter $L$ to denote the union $$L=L_1\cup\cdots\cup L_l$$ of the hyperplanes. For a subset $I\subset\{1,\ldots,l\}$, the stratum of $L$ indexed by $I$ is the affine space $$L_I=\bigcap_{i\in I}L_i$$ with the convention $L_\varnothing=\C^n$.\\
	
	We say that $L$ is \textit{affinely generic} if it is a normal crossing divisor inside $\C^n$. It means that for all $I$, $L_I$ is either empty or has codimension the cardinality $|I|$ of $I$.
	
	\begin{rem} If the $L_i$'s are in general position in $\C^n$ then $L$ is affinely generic, but the converse is not true. For instance, two parallel lines in $\C^2$ constitute an affinely generic arrangement. In other words, if we work in the projective space $\mathbb{P}^n(\C)$ by adding a hyperplane $L_0$ at infinity, the projective arrangement of hyperplanes $L_0\cup L_1\cup\cdots\cup L_n$ is not necessarily normal crossing.
	\end{rem}
	
	In the sequel, \textit{we will only consider affinely generic hyperplane arrangements}. This class of hyperplane arrangements is stable under the operations of deletion, contraction and product that we now describe.\\
	The \textit{deletion} of $L$ with respect to the last hyperplane $L_l$ is the arrangement $L'=\{L_1,\ldots,L_{l-1}\}$ in $\C^n$. We have a natural morphism $H^\bullet(\C^n\setminus L')\rightarrow H^\bullet(\C^n\setminus L)$.\\
	The \textit{contraction} of $L$ with respect to $L_l$ is the arrangement $L''=\{L_l\cap L_1,\ldots,L_l\cap L_{l-1}\}$ in $L_l\cong\C^{n-1}$ consisting of all the intersections of $L_l$ with the $L_i$'s, $i=1,\ldots,l-1$. We have a residue morphism $H^\bullet(\C^n\setminus L)(1)\rightarrow H^{\bullet-1}(L_l\setminus L'')$, where $(1)$ denotes a Tate twist.\\
	If $L^{(1)}\subset \C^{n_1}$ and $L^{(2)}\subset \C^{n_2}$ are two hyperplane arrangements, then the \textit{product} arrangement $L^{(1)}\times L^{(2)}\subset\C^{n_1+n_2}$ consists of the hyperplanes $L^{(1)}_{i_1}\times \C^{n_2}$ followed by the hyperplanes $\C^{n_1}\times L^{(2)}_{i_2}$. There is a K\"{u}nneth isomorphism $H^\bullet(\C^{n_1}\setminus L^{(1)})\otimes H^\bullet(\C^{n_2}\setminus L^{(2)})\cong H^\bullet(\C^{n_1+n_2}\setminus L^{(1)}\times L^{(2)}).$\\
		
	Let $\Lambda^\bullet(e_1,\ldots,e_l)$ denote the exterior algebra over $\Q$ with a generator $e_i$ in degree $1$ for each hyperplane $L_i$. For a set $I=\{i_1<\cdots<i_k\}\subset\{1,\ldots,l\}$ we set $e_I=e_{i_1}\wedge\cdots\wedge e_{i_k}$ with the convention $e_\varnothing=1$.\\
	Let $R_\bullet(L)$ be the ideal of $\Lambda^\bullet(e_1,\ldots, e_l)$ generated by the elements $e_I$ for subsets $I\subset\{1,\ldots,l\}$ such that $L_I=\varnothing$.\\
	
	The following theorem is a particular case of the Brieskorn-Orlik-Solomon theorem (for a detailed proof of the general case, see \cite[Theorems 3.126 and 5.89]{orlikterao}).
	
	\begin{thm}\label{BOSgeneric} Let $L$ be an affinely generic hyperplane arrangement.
	\begin{enumerate} 
	\item There is an isomorphism of graded algebras 
	\begin{equation}\label{BOSgenericiso}
	\Lambda^\bullet(e_1,\ldots,e_l)/R_\bullet(L)\stackrel{\cong}{\longrightarrow} H^\bullet(\C^n\setminus L)
	\end{equation}
	that sends $e_i$ to the class of the form $\omega_i=\dfrac{1}{2i\pi}\dfrac{df_i}{f_i}$, where $f_i$ is any linear form that defines $L_i$.
	\item This isomorphism is functorial in the following sense:
		\begin{enumerate}
		\item the deletion morphism $H^\bullet(\C^n\setminus L')\rightarrow H^\bullet(\C^n\setminus L)$ is given by $e_I\mapsto e_I$ for $I\subset\{1,\ldots,l-1\}$;
		\item the contraction morphism $H^\bullet(\C^n\setminus L)(1)\rightarrow H^{\bullet-1}(L_l\setminus L'')$ is given, for $I$ such that $l\notin I$, by $e_I\mapsto 0$ and $e_I\wedge e_l\mapsto e_I$;
		\item the K\"{u}nneth isomorphism $H^\bullet(\C^{n_1}\setminus L^{(1)})\otimes H^\bullet(\C^{n_2}\setminus L^{(2)})\cong H^\bullet(\C^{n_1+n_2}\setminus L^{(1)}\times L^{(2)})$ is given by $e_{I_1}\otimes e_{I_2}\mapsto e_{I_1\sqcup I_2}$.
		\end{enumerate}
	\end{enumerate}
	\end{thm}
	
	\begin{rem}\label{remBOSTate}
	The first part of Theorem \ref{BOSgeneric} implies that the mixed Hodge structure underlying $H^k(\C^n\setminus L)$ is pure of weight $2k$ and of Tate type: it is a direct sum of a certain number of copies of $\Q(-k)$.
	\end{rem}
	
	\begin{rem}\label{remArtinvanishing}
	If $I\subset\{1,\ldots,l\}$ has cardinality $>n$ then $L_I=\varnothing$ by the definition of an affinely generic hyperplane arrangement. Hence (\ref{BOSgenericiso}) implies that $H^k(\C^n\setminus L)=0$ for $k>n$. This is also a consequence of Artin vanishing since $\C^n\setminus L$ is an affine algebraic variety of dimension $n$.
	\end{rem}
	
	\subsection{Affinely generic bi-arrangements of hyperplanes}\label{paragraphrelativeBOS}
	
	A \textit{bi-arrangement of hyperplanes} $(L;M)$ in $\C^n$ is the data of two disjoint sets $L=\{L_1,\ldots,L_l\}$ and $M=\{M_1,\ldots,M_m\}$ of hyperplanes in $\C^n$. Equivalently, it is a $2$-partition of the underlying hyperplane arrangement $L\cup M=\{L_1,\ldots,L_l,M_1,\ldots,M_m\}$. As the notation suggests, both $L$ and $M$ are linearly ordered. We say that $(L;M)$ is \textit{affinely generic} if $L\cup M$ is, which means that it is a normal crossing divisor in $\C^n$. In the sequel, \textit{we will only consider affinely generic bi-arrangements of hyperplanes}.\\
	
	%\begin{rem}\label{remnotsymmetric}
	%A better setting would be to consider bi-arrangements in the projective space $\mathbb{P}^n(\C)$ instead of the affine space $\C^n$. Indeed, in the projective setting, Poincar\'e duality exchanges the roles played by $L$ and $M$ in $H(L;M)$ (at least when $L\cup M$ is a normal crossing divisor), hence we get more symmetric results. Here we choose to work in the affine setting since in the sequel we will consider a family of bi-arrangements with explicit equations written in affine coordinates.
	%\end{rem}
	
	Among the relative cohomology groups $H^\bullet(\C^n\setminus L, M\setminus M\cap L)$, we will focus on the middle-degree one: we set $$H(L;M)=H^n(\C^n\setminus L,M\setminus M\cap L).$$
	According to Deligne \cite{delignehodge3}, $H(L;M)$ is endowed with a functorial mixed Hodge structure. It is clear (and will be re-proved in the proof of Theorem \ref{relativeBOS}) that this is actually a mixed Hodge-Tate structure. This means that for all $k$ we have $\gr_{2k+1}^WH(L;M)=0$, and $\gr_{2k}^WH(L;M)$ is isomorphic to a direct sum of the Tate structures $\Q(-k)$. The graded quotient $\gr_{2k}^WH(L;M)$ is $0$ for $k\notin\{0,\ldots,n\}$.\\
		
		\begin{thm}\label{relativeBOS}
		Let $(L;M)$ be an affinely generic bi-arrangement in $\C^n$. Then for all $k=0,\ldots,n$ we have a presentation 
			\begin{equation}\label{relativeBOSiso} 
			\gr^W_{2k}H(L;M)\cong \left(\Lambda^k(e_1,\ldots,e_l)\otimes\Lambda^{n-k}(f_1,\ldots,f_m)\right) / R_k(L;M) 
			\end{equation}
			where $R_k(L;M)$ is spanned by the elements
			\begin{itemize}
			\item $e_I\otimes f_J$ if $L_I\cap M_J=\varnothing$, $|I|=k$, $|J|=n-k$.
			\item $e_I\otimes \left(\displaystyle\sum_{j\notin J'}\sgn(\{j\},J')f_{J'\cup\{j\}}\right)$ for $|I|=k$, $|J'|=n-k-1$.
			\end{itemize}
		\end{thm}
		
		\begin{proof}
		Let us denote by $j:\C^n\setminus(L\cup M)\hookrightarrow \C^n\setminus L$ the natural open immersion. Then $H^\bullet(\C^n\setminus L, M\setminus M\cap L)$ is the cohomology of the sheaf $j_!\Q_{\C^n\setminus(L\cup M)}$. One readily checks that we have a resolution
		$$0\rightarrow j_!\Q_{\C^n\setminus(L\cup M)} \rightarrow \Q_{\C^n\setminus L}\rightarrow \bigoplus_{i}(\iota_i)_*\Q_{M_i\setminus M_i\cap L} \rightarrow \bigoplus_{i<j}(\iota_{i,j})_*\Q_{M_{ij}\setminus M_{ij}\cap L}\rightarrow \cdots$$
		where $\iota_J:M_J\setminus M_J\cap L \hookrightarrow \C^n\setminus L$ denotes the natural closed immersion. More precisely let us set
		$$\mathcal{K}^p=\bigoplus_{|J|=p}(\iota_J)_*\Q_{M_J\setminus M_J\cap L}$$ 
		and $d:\mathcal{K}^p\rightarrow\mathcal{K}^{p+1}$ is given by the natural restriction morphisms
		$$(\iota_J)_*\Q_{M_J\setminus M_J\cap L}\rightarrow (\iota_{J\cup \{j\}})_*\Q_{M_{J\cup\{j\}}\setminus M_{J\cup\{j\}}\cap L}$$ 
		for $j\notin J$, multiplied by the sign $\sgn(\{j\},J)$. We then have a quasi-isomorphism
		$$j_!\Q_{\C^n\setminus(L\cup M)} \cong \mathcal{K}^\bullet.$$
		Let $w$ be the descending filtration on $\mathcal{K}^\bullet$ given by $w^p\mathcal{K}^\bullet=\mathcal{K}^{\geq p}$. The corresponding hypercohomology spectral sequence is
		$$E_1^{p,q}=\bigoplus_{|J|=p}H^q(M_J\setminus M_J\cap L) \Longrightarrow E_\infty^{p,q}=\gr_w^{p}H^{p+q}(\C^n\setminus L, M\setminus M\cap L)$$ 
		On the $E_1$-term, the differential $d_1$ is given by the natural restriction morphisms 
		$$H^q(M_J\setminus M_J\cap L)\rightarrow H^q(M_{J\cup\{j\}}\setminus M_{J\cup\{j\}}\cap L)$$ for $j\notin J$, multiplied by the sign $\sgn(\{j\},J)$.\\
		According to Deligne \cite[8.3.5]{delignehodge3}, this spectral sequence is a spectral sequence of mixed Hodge structures. Since by Remark \ref{remBOSTate} the mixed Hodge structures $H^q(M_J\setminus M_J\cap L)$ are pure of weight $2q$, the spectral sequence degenerates at $E_2$: $E_\infty=E_2$. The same argument implies that on $H^\bullet(\C^n\setminus L, M\setminus L\cap M)$, $w$ is (up to a shift) the canonical weight filtration.\\ 
		According to Remark \ref{remArtinvanishing}, we have  $H^k(M_J\setminus M_J\cap L)=0$ for $|J|>n-k$. Thus in degree $n$ we get
		$$\gr_{2k}^WH(L;M)\cong\mathrm{Coker}\left( \bigoplus_{|J'|=n-k-1}H^k(M_{J'}\setminus M_{J'}\cap L) \stackrel{d_1}{\rightarrow} \bigoplus_{|J|=n-k}H^k(M_J\setminus M_J\cap L) \right)$$
		which is obviously $0$ if $k\notin\{0,\ldots,n\}$.
	   Introducing basis elements $f_J$, Theorem \ref{BOSgeneric} tells us that $H^k(M_J\setminus M_J\cap L)$ has a presentation given by generators $e_I\otimes f_J$, $|I|=k$, and relations $e_I\otimes f_J=0$ if $L_I\cap M_J=\varnothing$. Since the differential is given by $d_1(e_I\otimes f_{J'})=e_I\otimes\left(\displaystyle\sum_{j\notin J'} \sgn(\{j\},J')f_{J'\cup\{j\}}\right)$, this implies the theorem.
		\end{proof}
		
		\begin{rem}\label{remlogforms}
		In order to do explicit computations, we introduce a useful acyclic model for the complex of sheaves $\mathcal{K}_\C^\bullet:=\mathcal{K}^\bullet\otimes\C$.
		For $(L;M)$ an affinely generic hyperplane arrangement, let us define a double complex of sheaves on $X$ 
		$$\Omega^{p,q}_{(L;M)}=\bigoplus_{|J|=p} (i_{M_J}^{\C^n})_*\Omega^q_{M_J}(\log L)$$
		where $\Omega^\bullet_{M_J}(\log L)$ is the complex of logarithmic forms defined in \cite[3.1]{delignehodge2}, and $i_{M_J}^{\C^n}$ is the inclusion of $M_J$ inside $\C^n$. The horizontal differential $\Omega^{p,q}_{(L;M)}\rightarrow \Omega^{p+1,q}_{(L;M)}$ is given by the restriction morphisms $\Omega^q_{M_J}(\log L)\rightarrow (i_{M_J\cap M_j}^{M_J})_*\Omega^q_{M_{J\cup \{j\}}}(\log L)$ for $j\notin J$, multiplied by the sign $\sgn(\{j\},J)$. The vertical differential $\Omega^{p,q}_{(L;M)}\rightarrow \Omega^{p,q+1}_{(L;M)}$ is the exterior differential on forms. We let $\Omega^\bullet_{(L;M)}$ denote the total complex. Using \cite[3.1.8]{delignehodge2}, one easily proves that we have a quasi-isomorphism
		$$\mathcal{K}^\bullet \stackrel{\cong}{\longrightarrow} \Omega^\bullet_{(L;M)}.$$
		\end{rem}
		
		We adapt the notions of deletion and contraction to the setting of (affinely generic) bi-arrangements. Furthermore, we allow ourselves to iterate them. Thus, for a subset $I_0\subset\{1,\ldots,l\}$ (resp. $J_0\subset \{1,\ldots,m\}$) we may consider the deletion $(L(I_0);M)$ (resp. $(L;M(J_0))$) obtained by forgetting the hyperplanes $L_i$, $i\notin I_0$ (resp. the hyperplanes $M_j$, $j\notin J_0$), and the contraction $(L_{I_0}|L(\overline{I_0});M)$ (resp. $(M_{J_0}|L;M(\overline{J_0}))$) obtained by considering the intersections of the hyperplanes with $L_{I_0}$ (resp. with $M_{J_0}$).\\
		On the relative cohomology groups $H(L;M)$, we get natural deletion/contraction morphisms, which are computed in the next theorem.
			
		\begin{thm}\label{functoriality}
		The isomorphism (\ref{relativeBOSiso}) is functorial in the following sense. 
			\begin{enumerate}
			\item For a subset $J_0\subset\{1,\ldots,m\}$, the deletion morphism $H(L;M)\rightarrow H(L;M(J_0))$ is given by $e_I\otimes f_J\mapsto 0$ if $J\not\subset J_0$ and $e_I\otimes f_J\mapsto e_I\otimes f_J$ if $J\subset J_0$.
			\item For a subset $J_0\subset\{1,\ldots,m\}$, the contraction morphism $H(M_{J_0}|L;M(\overline{J_0}))\rightarrow H(L;M)$ is given, for $J\subset\overline{J_0}$, by $$e_I\otimes f_J\mapsto e_I\otimes (f_{J_0}\wedge f_J)=\sgn(J_0,J) e_I\otimes f_{J_0\cup J}.$$
			\item For a subset $I_0\subset\{1,\ldots,l\}$, the deletion morphism $H(L(I_0);M)\rightarrow H(L;M)$ is given, for $I\subset I_0$, by $e_I\otimes f_J\mapsto e_I\otimes f_J$.
			\item For a subset $I_0\subset\{1,\ldots,l\}$ of cardinality $k_0$, the contraction morphism $H(L;M)(k_0)\rightarrow H(L_{I_0}|L(\overline{I_0});M)$ is given, for $I_0\not\subset I$, by $e_I\otimes f_J\mapsto 0$, and $$e_{I\cup I_0}\otimes f_J=\sgn(I,I_0)(e_I\wedge e_{I_0})\otimes f_J\mapsto \sgn(I,I_0) e_I\otimes f_J.$$
			\item The K\"{u}nneth morphism $H(L^{(1)};M^{(1)})\otimes H(L^{(2)};M^{(2)})\rightarrow H(L^{(1)}\times L^{(2)};M^{(1)}\times M^{(2)})$ is given by $(e_{I_1}\otimes f_{J_1})\otimes (e_{I_2}\otimes f_{J_2})\rightarrow (e_{I_1\sqcup I_2})\otimes (f_{J_1\sqcup J_2})$.
			\end{enumerate}
		\end{thm}
		
		\begin{proof}
		\begin{enumerate}
		\item It is obvious.
		\item Let $k_0$ be the cardinality of $J_0$. Let us denote $\mathcal{K}_0$ the complex of sheaves corresponding to $(M_{J_0}|L;M(\overline{J_0}))$ as defined in the proof of Theorem \ref{relativeBOS}. The contraction morphism $H(M_{J_0}|L;M(\overline{J_0}))\rightarrow H(L;M)$ is defined by a morphism $\Phi:\mathcal{K}_0^{\bullet-k_0}\rightarrow\mathcal{K}^\bullet$.\\
		By definition we have $$\mathcal{K}_0^{p-k_0}=\bigoplus_{\substack{|J|=p-k_0\\J_0\cap J=\varnothing}}(\iota_{J_0\cup J})_*\Q_{M_{J_0\cup J}\setminus M_{J_0\cup J}\cap L}$$
		which is obviously a sub-sheaf of $\mathcal{K}^p$. We define $\Phi$ by multiplying the natural inclusion by the sign $\sgn(J_0,J)$ on the component indexed by $J$. We check that with this sign, $\Phi$ is a morphism of complexes of sheaves, and the claim follows.
		%On the models defined in Remark \ref{remlogforms}, this corresponds to a morphism $\Omega^{p-1,q}_{(L;M)}\rightarrow \Gamma(\C^n,\Omega^{p,q}_{\langle\C^n,(L;M)\rangle})$ defined by exactly the same formula. The assertion then follows from the equality $\sgn(\{m\},J)f_{J\cup\{m\}}=f_m\wedge f_J$.
		\item It is obvious.
		\item It is enough to do the proof over $\C$ and work with the models defined in Remark \ref{remlogforms}. The residue morphism is then given by morphisms (with obvious notation)
		$$\Omega^{p,q}_{(L;M)}\rightarrow (i_{L_{I_0}}^{\C^n})_*\Omega^{p,q-k_0}_{(L_{I_0}|L(\overline{I_0});M)}$$ which are induced by the residue morphisms
		$$\Omega^q_{M_J}(\log L) \rightarrow (i_{L_{I_0}\cap M_J}^{M_J})_*\Omega^{q-k_0}_{L_{I_0}\cap M_J}(\log L(\overline{I_0}))$$
		defined in \cite[3.1.5]{delignehodge2}. The formula follows.
		\item We also work over $\C$ with the models defined in Remark \ref{remlogforms}. It is easy to check that the K\"{u}nneth morphism is given by the cup-product
		$$(p_1)^*\Omega^{q_1}_{M^{(1)}_{J_1}}(\log L^{(1)})\otimes (p_2)^*\Omega^{q_2}_{M^{(2)}_{J_2}}(\log L^{(2)}) \rightarrow \Omega^{q_1+q_2}_{M^{(1)}_{J_1}\times M^{(2)}_{J_2}}(\log L^{(1)}\times L^{(2)})$$
		and the formula follows.
		\end{enumerate}
		\end{proof}

\section{Dissection polylogarithms}

	In this section, we focus on $\C$-decorated dissection diagrams, which we simply call decorated dissection diagrams.

	\subsection{The bi-arrangement attached to a decorated dissection diagram}
	
		\subsubsection{Definition}\label{defbiarrangementdd}
	
		We attach to any decorated dissection diagram $D$ of degree $n$ a bi-arrangement $(L;M)$ inside $\C^n$. The equations of the $L_i$'s depend on the chords of $D$ and their decorations, while the equations of the $M_j$'s depend on the decorations of the sides of the polygon (hence not on the combinatorics of $D$).\\
		Let us recall that the total directed graph $\Gamma(D)$ of $D$ is the graph whose vertices are the $(n+1)$ vertices of $\Pi_n$, and whose $(2n+1)$ directed edges are the chords of $D$ and the sides of $\Pi_n$, oriented clockwise.\\
		
		Let us work in the complex affine space $\C^n$ with coordinates $(t_1,\ldots,t_n)$. To each edge in $\Gamma(D$) we associate a hyperplane in $\C^n$ in the following way:
		\begin{itemize}
		\item To an edge $\stackrel{i}{\bullet}\stackrel{\alpha}{\longrightarrow} \stackrel{j}{\bullet}$ between two non-root vertices, we associate the hyperplane $t_i-t_j-\alpha=0$.
		\item To an edge $\stackrel{i}{\bullet}\stackrel{\alpha}{\longrightarrow} \circ$ that goes to the root, we associate the hyperplane $t_i-\alpha=0$.
		\item To an edge $\circ\stackrel{\alpha}{\longrightarrow} \stackrel{i}{\bullet}$ that comes from the root, we associate the hyperplane $-t_i-\alpha=0$.
		\end{itemize}
	
		Hence the rule is always the same: we interpret the vertex $i$ as the coordinate $t_i$, and the root as $0$. The third case above only occurs for the side labeled $0$. \\
		
		We label $L_1,\ldots,L_n$ the hyperplanes given by the chords of the decorated dissection diagram $D$, $L_i$ being given by the $i$-th chord (which by definition is the chord starting at the vertex $i$). Hence $L_i$ is defined by $t_i-t_j-a_i=0$ if the $i$-th chord goes to the $j$-th vertex, and by $t_i-a_i=0$ if it goes to the root.\\
		
		We label $M_0,M_1,\ldots,M_n$ the hyperplanes given by the sides of the polygon $\Pi_n$, in clockwise order. They are defined by
		$M_0=\{t_1=-b_0\}$, $M_j=\{t_j=t_{j+1}+b_j\}$ for $j=1,\ldots,n-1$, and $M_n=\{t_n=b_n\}$.\\
		
		This defines a bi-arrangement $$(L;M)=(L_1,\ldots,L_n;M_0,M_1,\ldots,M_n)$$ in $\C^n$.

		\begin{ex}
		Let us look at the decorated dissection diagram $D$ of degree $3$ from Example \ref{exdecoratedQR}. Then the bi-arrangement $(L_1,L_2,L_3;M_0,M_1,M_2,M_3)$ in $\C^3$ is defined by the equations
		$L_1=\{t_1-a_1=0\}$, $L_2=\{t_2-t_1-a_2=0\}$, $L_3=\{t_3-t_1-a_3=0\}$, $M_0=\{t_1=-b_0\}$, $M_1=\{t_1=t_2+b_1\}$, $M_2=\{t_2=t_3+b_2\}$, $M_3=\{t_3=b_3\}$.
		\end{ex}		
		
		The combinatorics of the bi-arrangement $(L;M)$ can be read directly off the dissection diagram, as the following lemma shows.
		
		\begin{lem}\label{lemrelationscycles}
		Let $D$ be a decorated dissection diagram with generic decorations.
		Let $I\subset\{1,\ldots,n\}$ be a set of chords of $D$ and $J\subset\{0,\ldots,n\}$ be a set of sides of $D$. We view $I\cup J$ as a subgraph of the total directed graph $\Gamma(D)$.
		\begin{enumerate}
		\item $L_I\cap M_J=\varnothing$ if and only if the graph $I\cup J$ contains an undirected cycle.
		\item If $L_I\cap M_J\neq \varnothing$, then $\mathrm{codim}(L_I\cap M_J)=|I|+|J|$.
		\end{enumerate}
		Thus the bi-arrangement $(L;M)$ is affinely generic.
		\end{lem}
		
		\begin{proof}
		If there is an undirected path of total decoration $\lambda$ from the vertex $i$ to the vertex $j$ in $I\cup J$, then for any point $(t_1,\ldots,t_n)\in L_I\cap M_J$, we get $t_i=t_j+\lambda$. If $I\cup J$ contains an undirected cycle, then it contains some simple cycle with total decoration $\lambda\neq 0$. Let $i$ be a non-root vertex inside this simple cycle. For a point $(t_1,\ldots,t_n)$ in $L_I\cap M_J$, we get by definition $t_i=t_i+\lambda$, which is impossible. Thus $L_I\cap M_J=\varnothing$.\\
		Conversely, one easily sees that if $I\cup J$ does not contain an undirected cycle then $L_I\cap M_J\neq\varnothing$ and $\mathrm{codim}(L_I\cap M_J)=|I|+|J|$.
		\end{proof}
		
		\subsubsection{Operations on dissection diagrams and bi-arrangements}
		
		We can now explain the conventions from \S \ref{conventions} on dissection diagrams.
		
		\begin{itemize}
		\item If we change the direction of an edge and multiply its decoration by $-1$, this does not change the equation given by this edge.
		\item The convention for the contraction of edges accounts for the contraction of hyperplanes in bi-arrangements. Indeed, let us look at a contracted bi-arrangement $(L_i|L_1,\ldots,\widehat{L_i},\ldots,L_n;M)$. If we choose the coordinates on $L_i\cong \C^{n-1}$ to be $(t_1,\ldots,\widehat{t_i},\ldots,t_n)$, then the equations of the hyperplanes in this restricted bi-arrangement are exactly given by the edges of the graph resulting from the contraction of the $i$-th chord, with the convention from \S \ref{conventions}. The same is of course true for a contraction of some hyperplane $M_j$, $j\geq 1$.
		\end{itemize}
		This allows us to reinterpret the operations $q_C$ and $r_C$ in terms of contraction and deletion of bi-arrangements.
		
		\begin{lem}\label{lembiarrangementsRQ}
		Let $D$ be a decorated dissection diagram of degree $n$ and $(L;M)$ the corresponding bi-arrangement in $\C^n$. Let $C\subset\c$ be a set of chords of $D$.
		\begin{enumerate}
		\item For each $\alpha$, let $(L^{(\alpha)};M^{(\alpha)})$ be the bi-arrangement corresponding to the dissection diagram $q_C^\alpha(D)$. We have an isomorphism of bi-arrangements
		$$\prod_\alpha \left(L^{(\alpha)};M^{(\alpha)}\right) \cong (L_C|L(\overline{C});M).$$
		\item The bi-arrangement of the dissection diagram $r_C(D)$ is $$\left(\left. M_{\s_C^+}\right\vert L(C);M_0,M\left(\overline{\s_C^+}\right)\right).$$
		\end{enumerate}
		\end{lem}	
		
		\begin{proof}
		\begin{enumerate}
		\item Let us recall the partition $\overline{C}=\bigsqcup_\alpha \overline{C}(\alpha)$. The equations of the bi-arrangement $(L^{(\alpha)};M^{(\alpha)})$ are written in coordinates $t_i$, $i\in\overline{C}(\alpha)$, hence the product $\prod_\alpha (L^{(\alpha)};M^{(\alpha)})$ is a bi-arrangement in an affine space with coordinates $t_i$, $i\in\overline{C}$. The same is true of $(L_C|L(\overline{C});M)$. We then describe the isomorphism.\\
		Let us denote by $D/C$ the graph obtained by contracting the chords from $C$; its non-root vertices are labeled by $\overline{C}$. For each $\alpha$, the root of $q_C^\alpha(D)$ in $D/C$ is either the root of $D$ or a non-root vertex $\rho(\alpha)\in\overline{C}$. We let $t(\alpha)=0$ in the first case, and $t(\alpha)=t_{\rho(\alpha)}$ in the second case. The isomorphism is then defined by the change of variables $t'_i=t_i+t(\alpha)$ for $i\in \overline{C}(\alpha)$.
		\item It is straightforward, if we choose the coordinates $t_i$, $i\notin C$, on $L_i$.
		\end{enumerate}
		\end{proof}	
		
		\begin{ex}
		Let us look at Example \ref{exdecoratedQR} and illustrate the first point of the above lemma. On $L_3$ with coordinates $(t_1,t_2)$, the change of variables is defined by $t_1=t'_1$, $t_2=t'_2-t'_1$. Then for instance the equation $t_2-a_2=0$ becomes $t'_2-t'_1-a_2=0$. 
		\end{ex}

	\subsection{Definition of the dissection polylogarithms}\label{defdissectionpolylogs}

		We fix a decorated dissection diagram $D$ of degree $n$ and assume that its decorations are generic.\\

		\paragraph{\textit{The differential form $\omega_D$}}	\hspace*{\fill}\\
		For $i=1,\ldots,n$, let $\varphi_i$ be the linear equation for the hyperplane $L_i$ defined in the previous paragraph, of the form $\varphi_i=t_i-t_j-a_i$ if the $i$-th chord goes to the $j$-th vertex, and by $\varphi_i=t_i-a_i$ if it goes to the root.
		We then set $$\omega_D=\dfrac{1}{(2i\pi)^n}\mathrm{dlog}(\varphi_1)\wedge\cdots\wedge\mathrm{dlog}(\varphi_n)=\dfrac{1}{(2i\pi)^n}\frac{dt_1\wedge\cdots\wedge dt_n}{\varphi_1\cdots\varphi_n}\cdot$$
		 It is a meromorphic $n$-form on $\C^n$ and its polar locus is exactly the union $L=L_1\cup\cdots\cup L_n$.\\
		
		\paragraph{\textit{The integration simplex $\Delta_D$}}	\hspace*{\fill}\\
		In the previous paragraph we have defined a family of hyperplanes $M_0=\{t_1=-b_0\}$, $M_j=\{t_j=t_{j+1}+b_j\}$ for $j=1,\ldots,n-1$, and $M_n=\{t_n=b_n\}$. We set $M=M_0\cup M_1\cup\cdots\cup M_n$.\\
		We fix a singular $n$-simplex $\Delta_D$ inside $\C^n\setminus L$ such that for all $j=0,\ldots,n$, $\partial_j\Delta_D\subset M_j$. The existence of such a simplex is guaranteed by the fact that the decorations being generic, $L\cup M$ is a normal crossing divisor inside $\C^n$ (Lemma \ref{lemrelationscycles}).
		
		\begin{defi}
		We set $$I(D)=\int_{\Delta_M}\omega_D\,\,\,\in\C$$ and call it the \textit{dissection polylogarithm} attached to the dissection diagram $D$.
		\end{defi}
		
		The above integral is absolutely convergent since the integration simplex $\Delta_D$ does not meet the polar locus $L$ of the form $\omega_D$.\\
		
		As the examples in the next paragraph will show, the integral $I(D)$ really depends on $\Delta_D$ (though only via its homology class $[\Delta_D]\in H_n(\C^n\setminus L,M\setminus M\cap L)$). Thus, the notation $I(D)$ is abusive. We allow ourselves that abuse for at least two reasons. Firstly, there is no canonical way of choosing $[\Delta_D]$ for \textit{all} decorated dissection diagrams; if one looks at specific families of dissection diagrams and/or decorations (see the examples in the next paragraph) then this may sometimes be achieved. Secondly, we will replace (see Definition \ref{defmotivicdissectionpolylog}) the dissection polylogarithms $I(D)$ by motivic versions $I^\H(D)$ that only depend on the decorated diagram $D$, and not on the homology class of $\Delta_D$.
		
		\begin{rem}\label{remdegenerateinfty}
		A dissection polylogarithm is a special case of an Aomoto polylogarithm in the sense of \cite{bvgs}. To make the connection with the setting of \cite{bvgs} precise, one has to work in the projective setting, adding the hyperplane at infinity $L_0$. One has to notice that in this case we get a pair of simplices $(L;M)$ inside $\mathbb{P}^n(\C)$ which is not in general position: it is highly degenerate at infinity. Thus we are not in the case studied by J. Zhao in \cite{zhaogeneric}.
		\end{rem}
		
	\subsection{Examples of dissection polylogarithms}\label{exdissectionpolylogs}
	We study some families of dissection polylogarithms.\\
	
		\begin{figure}[h]
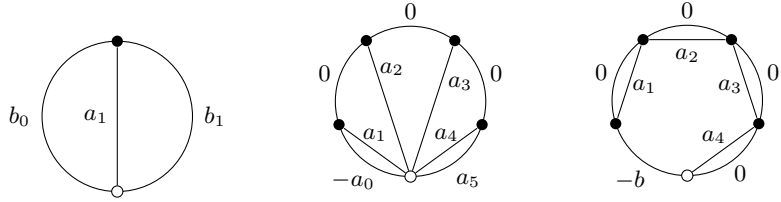

		\begin{center}\decoratedlogarithm{1}\hspace{1cm}\decoratedcorollaFour{1}\hspace{1cm}\decoratedpathtreeFour{1}\end{center}
		\caption{A decorated dissection diagram of degree $1$; the decorated corolla corresponding to the iterated integral $\mathbb{I}(a_0;a_1,a_2,a_3,a_4;a_5)$; the decorated path tree corresponding to the $\mathbb{J}$-integral $\mathbb{J}(b;a_1,a_2,a_3,a_4)$.}\label{figdecoratedcorollaFour}
		\end{figure}
		
		\paragraph{\textit{Degree $1$: logarithms}}	\hspace*{\fill}\\
		
		Let $D$ be a decorated dissection diagram of degree $1$ (see Figure \ref{figdecoratedcorollaFour}). The genericity assumption on the decorations reads: 
		$$a_1+b_0\neq 0 , a_1-b_1\neq 0 , b_0+b_1\neq 0$$ 
		We have $\varphi_1=t-a_1$ so that $L_1=\{a_1\}$ and $\omega_D=\frac{1}{2i\pi}\frac{dt}{t-a_1}$. We have $M_0=\{-b_0\}$ and $M_1=\{b_1\}$, so that $\Delta_D$ is any continuous path from $-b_0$ to $b_1$ in $\C\setminus\{a_1\}$. We then have
		$$I(D)=\dfrac{1}{2i\pi}\int_{-b_0}^{b_1} \frac{dt}{t-a_1}=\dfrac{1}{2i\pi}\log\left(\dfrac{a_1-b_1}{a_1+b_0}\right).$$
		As is well-known, this number is well-defined up to an integer, depending on the number of times that the path of integration winds around $a_1$.\\
		
		\paragraph{\textit{Corollas and iterated integrals}}	\hspace*{\fill}\\
		
		This example generalizes the previous one. Let us consider the case where $D$ is a corolla of degree $n$, which is the case when all chords of $D$ point towards the root. In this case we have $\varphi_i=t_i-a_i$ for all $i=1,\ldots,n$, so that $$\omega_D=\dfrac{1}{(2i\pi)^n}\dfrac{dt_1\wedge\cdots\wedge dt_n}{(t_1-a_1)\cdots(t_n-a_n)}\cdot$$
		By performing the change of variables $$t'_1=t_1,t'_2=t_2+b_1,t'_3=t_3+b_1+b_2,\ldots,t'_n=t_n+b_1+\cdots+b_n$$	
		we can always assume that the decorations on the sides of $D$ are all $0$ except for the first and the last one. We then put $a_0=-b_0$ and $a_{n+1}=b_n$, so that the genericity condition reads: $a_i\neq a_j$ for $0\leq i\neq j\leq n+1$ (see Figure \ref{figdecoratedcorollaFour}).\\
		We denote by $\Delta(a_0,a_{n+1})$ the integration simplex $\Delta_D$. Its boundary is given by the hyperplanes $t_1=a_0$, $t_j=t_{j+1}$ for $j=1,\ldots,n$, and $t_n=a_{n+1}$. The corresponding dissection polylogarithm is denoted
		 \begin{equation}\label{defitint}
		 \I(a_0;a_1,\ldots,a_n;a_{n+1})=\dfrac{1}{(2i\pi)^n}\int_{\Delta(a_0,a_{n+1})}\dfrac{dt_1\cdots dt_n}{(t_1-a_1)\cdots(t_n-a_n)}\cdot
		 \end{equation}
		 These integrals are iterated integrals on the space $\C\setminus\{a_1,\ldots,a_n\}$ and have been much studied by A. B. Goncharov (see \cite{goncharovmultiplepolylogs} and \cite{goncharovgaloissym}).\\
		  In this particular case, one way of choosing the integration simplex $\Delta(a_0,a_{n+1})$ is by specifying a continuous path $\gamma:[0,1]\rightarrow\C\setminus\{a_1,\ldots,a_n\}$ such that $\gamma(0)=a_0$ and $\gamma(1)=a_{n+1}$, and considering the singular simplex
		   $$\{0\leq t_1\leq\cdots\leq t_n\leq 1\}\rightarrow (\C\setminus\{a_1,\ldots,a_n\})^n , (t_1,\ldots,t_n)\mapsto (\gamma(t_1),\ldots,\gamma(t_n)).$$
		   For that reason, the study of the integrals (\ref{defitint}) is closely related to the algebro-geometric properties of the fundamental groups of the spaces $\C\setminus\{a_1,\ldots,a_n\}$.\\
		   In general, the dissection polylogarithms cannot be interpreted directly as iterated integrals in the above sense (however, see Theorem \ref{reductionitint} for an abstract statement on a reduction to iterated integrals).\\
		
		\paragraph{\textit{Path trees and $\mathbb{J}$-polylogarithms}}	\hspace*{\fill}\\
		
		 Let us consider the case where $D$ is a path tree of degree $n$. In this case we get $$\omega_D=\dfrac{1}{(2i\pi)^n}\frac{dt_1\wedge\cdots\wedge dt_n}{(t_1-t_2-a_1)(t_2-t_3-a_2)\cdots (t_{n-1}-t_n-a_{n-1})(t_n-a_n)}\cdot$$
	 	As in the previous example, we may perform a change of variables so that the edge decorations are $b_0=-b$, $b_1=\cdots=b_{n-1}=b_n=0$ (see Figure \ref{figdecoratedcorollaFour}).\\
	 	Let us write, for all $I\subset\{1,\ldots,n\}$, $a_I=\displaystyle\sum_{i\in I}a_i$. Then the genericity condition on the decorations reads: for all $i=1,\ldots,n$, $a_i\neq 0$ and for all $I\subset\{1,\ldots,n\}$, $a_I\neq b$ (which includes the condition $b\neq 0$ for $I=\varnothing$).\\ 
	 	The corresponding dissection polylogarithm is denoted
	 	$$\mathbb{J}(b;a_1,\ldots,a_n)=\dfrac{1}{(2i\pi)^n}\int_{\Delta(b,0)}\frac{dt_1\cdots dt_n}{(t_1-t_2-a_1)\cdots (t_{n-1}-t_n-a_{n-1})(t_n-a_n)}\cdot$$
		 	
	\subsection{Relations among dissection polylogarithms}\label{sectionrelations}
	
		We describe certain families of relations between dissection polylogarithms that one can describe combinatorially on the dissection diagrams.\\
	
		\paragraph{\textit{Translations}} \hspace*{\fill}\\
		
		Let $D$ be a decorated dissection diagram of degree $n$; let us fix a non-root vertex $i\in\{1,\ldots,n\}$ and $\lambda\in\C$. Let $\tau_i(\lambda).D$ be the decorated dissection diagram obtained from $D$ by adding $\lambda$ to the decoration of every edge of $\Gamma(D)$ going to $i$, and substracting $\lambda$ from the decoration of every edge of $\Gamma(D)$ leaving $i$:
		
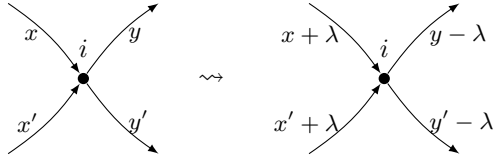
		\begin{center}\relationtranslation\end{center}
		
		\begin{prop}\label{r1}
		The decorations of $\tau_i(\lambda).D$ are generic if and only if the decorations of $D$ are generic. In this case we have the equality
		 $$(R1):\,I(D)=I(\tau_i(\lambda).D).$$
		\end{prop}
		
		\begin{proof}
		The first statement is straightforward. The equality follows from the change of variables $t_i\mapsto t_i-\lambda$ in the integral defining $I(D)$. Of course the simplices $\Delta_D$ and $\Delta_{\tau_i(\lambda).D}$ are chosen in a compatible way: $\Delta_{\tau_i(\lambda).D}$ is the image of $\Delta_D$ under $t_i\mapsto t_i-\lambda$.
		\end{proof}
	
		\paragraph{\textit{Rotating a dissection diagram}} \hspace*{\fill}\\
		
		Let $D$ be a decorated dissection diagram of degree $n$; let $D^+$ be the dissection diagram obtained from $D$ by rotating the labels of the $(n+1)$ vertices of $D$ in clockwise order. One has to flip a certain number of chords so that all the chords in $D^+$ point towards the root (which was formerly vertex $1$). The rule for flipping chords is given in \S \ref{conventions}.
		
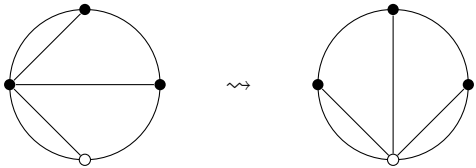
		\begin{center}$$\vc{\dddegthreeexTen{1}}\hspace{.7cm}\leadsto\hspace{.7cm}\vc{\dddegthreeexTwelve{1}}$$
		$$\,\,\,D\hspace{3.7cm} D^+$$\end{center}
		
		\begin{prop}\label{r2}
		The decorations of $D^+$ are generic if and only if the decorations of $D$ are generic. In this case, let $\varepsilon$ be the signature of the permutation relating the orders of the chords in $D$ and in $D^+$. Then we have the equality
		$$(R2):\,I(D)=(-1)^n\varepsilon\, I(D^+).$$
		\end{prop}
		
		\begin{proof}
		The first statement is straightforward since $\Gamma(D)=\Gamma(D^+)$ as decorated directed graphs.\\
		Let us perform the change of variables $f(t_1,\ldots,t_n)=(t_2-t_1,t_3-t_1,\ldots,t_{n-1}-t_1,-t_1)$ in the integral defining $I(D)$:
		$$I(D)=\int_{\Delta_D}\omega_D=\int_{f^{-1}(\Delta_D)}f^*\omega_D.$$
		Now $\Delta_{D^+}$ is chosen to be $f^{-1}(\Delta_D)$, but with the orientation multiplied by $(-1)^n$: indeed, we perform a cyclic permutation of the $(n+1)$ faces of the simplex. As the differential forms are concerned, we get by definition $f^*\omega_D=\varepsilon\,\omega_{D^+}$, hence the result.
		\end{proof}
		
		\paragraph{\textit{Stokes' theorem}}\hspace*{\fill}\\
		
		Let us consider a set of $n$ non-intersecting chords in $\Pi_{n+1}$ such that the graph created by the chords is acyclic. For such a diagram $\wt{D}$ and a side $s\in\{0,\ldots,n+1\}$ of $\Pi_{n+1}$, we let $\partial_s\wt{D}$ be the graph obtained by contracting the side $s$. One easily checks that there exist exactly two sides $i$ and $j$ of $\wt{D}$ such that $\partial_i\wt{D}$ and $\partial_j\wt{D}$ are dissection diagrams.
		\begin{center}\relationstokes{1}\end{center}
		Now let us suppose that the chords of $\wt{D}$ are directed and that we are given a decoration on the total directed graph of $\wt{D}$. Then $\wt{D}$ gives a bi-arrangement $(L;M)=(L_1,\ldots,L_n;M_0,M_1,\ldots,M_{n+1})$ in the same fashion as in \S \ref{defdissectionpolylogs}. For a side $s\in\{0,\ldots,n+1\}$, the bi-arrangement given by $\partial_s\wt{D}$ is exactly the contraction $(M_s|L;M_0,\ldots,\widehat{M_s},\ldots,M_{n+1})$, with natural coordinates $(t_1,\ldots,\widehat{t_s},\ldots,t_{n+1})$ on $M_s$ (the rule for contracting edges is given in \S \ref{conventions}).\\
		To sum up, $\partial_i\wt{D}$ and $\partial_j\wt{D}$ are decorated dissection diagrams. One way of coherently choosing the singular simplices $\Delta_{\partial_i\wt{D}}$ and $\Delta_{\partial_j\wt{D}}$ is to choose a singular simplex $\wt{\Delta}$ such that $\partial_s\wt{\Delta}\subset M_s$ for all $s$, and to put $\Delta_{\partial_i\wt{D}}=\partial_i\wt{\Delta}$ and $\Delta_{\partial_j\wt{D}}=\partial_j\wt{\Delta}$.
		
		\begin{prop}\label{r3}
		If the decorations on $\wt{D}$ are generic then the decorations on $\partial_i\wt{D}$ and $\partial_j\wt{D}$ are generic too. In this case, let $\varepsilon_i$ (resp. $\varepsilon_j$) be the signature of the permutation relating the orders of the chords in $\wt{D}$ and in $\partial_i\wt{D}$ (resp. $\partial_j\wt{D}$). Then we have the equality 
		$$(R3):\,(-1)^i\varepsilon_i\,I(\partial_i\wt{D})+(-1)^j\varepsilon_j\,I(\partial_j\wt{D})=0.$$
		\end{prop}
		
		\begin{proof}
		The first statement is straightforward. Let $\omega$ be the differential $n$-form on $\C^{n+1}$ given by the $n$ decorated chords of $\wt{D}$ as in \ref{defdissectionpolylogs}. It is a closed form so by Stokes' theorem we get
		$$\sum_{s=0}^{n+1}(-1)^s\int_{\partial_s\wt{\Delta}}\omega_{|\partial_s\wt{\Delta}}=0.$$
		For $s\notin\{i,j\}$ we get $\omega_{|\partial_s\wt{\Delta}}=0$; the result then follows from the equalities $\omega_{|\partial_i\wt{\Delta}}=\varepsilon_i\,\omega_{\partial_i\wt{D}}$ and $\omega_{|\partial_j\wt{\Delta}}=\varepsilon_j\,\omega_{\partial_j\wt{D}}$.
 		\end{proof}
		
		\paragraph{\textit{Orlik-Solomon relations}}\hspace*{\fill}\\
		
		Let us consider a set of $(n+1)$ non-intersecting chords in $\Pi_n$ such that the graph created by the chords has Betti number $1$. For such a diagram $\widehat{D}$, let $C$ be the unique simple cycle. For every chord $c\in C$, we get a dissection diagram $\widehat{D}\setminus \{c\}$ by deleting $c$.
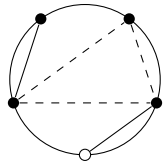
		\begin{center} \relationorliksolomon{1} \end{center}
		Now let us suppose that the chords of $\widehat{D}$ are linearly ordered by $\{1,\ldots,n+1\}$ and directed and that we are given a decoration on the total graph of $\widehat{D}$. Then for every chord $c\in C$ we get a decorated dissection diagram $\widehat{D}\setminus \{c\}$. It has to be noted that one may have to reorder the chords in $\widehat{D}\setminus \{c\}$. One may compute all the dissection polylogarithms $I(\widehat{D}\setminus \{c\})$ using the same choice of integration simplex.
		
		\begin{prop}\label{r4}
		Let us suppose that among all simple cycles in the total graph of $\widehat{D}$, $C$ is the only one whose total decoration is $0$. Then for every chord $c\in C$, the decorations on $\widehat{D}\setminus \{c\}$ are generic. For $c\in C$, let us denote by $\varepsilon(c)$ the product of the signs $\sgn(\{c\},C\setminus\{c\})$, $\sgn(C\setminus\{c\},\{1,\ldots,n+1\}\setminus C)$, and the signature of the permutation reordering the chords in $\widehat{D}\setminus\{c\}$. We then have the equality:
		$$(R4):\,\sum_{c\in C}\varepsilon(c)\,I(\widehat{D}\setminus \{c\})=0.$$
		\end{prop}
		
		\begin{proof}
		The first statement is straightforward. Since the total decoration of $C$ is $0$, one easily sees that the hyperplanes $L_c$, for $c\in C$, are linearly dependent. Thus, we have the Orlik-Solomon relation \cite[Lemma 3.119]{orlikterao}
		$$\sum_{c\in C}\sgn(\{c\},C\setminus\{c\})\,\omega_{C\setminus \{c\}}=0.$$
		Multiplying on the right by $\omega_{\{1,\ldots,n+1\}\setminus C}$ we get
		$$\sum_{c\in C}\sgn(\{c\},C\setminus\{c\})\sgn(C\setminus\{c\},\{1,\ldots,n+1\}\setminus C)\,\omega_{\{1,\ldots,n+1\}\setminus \{c\}}=0.$$
		The result then follows from the fact that $\omega_{\{1,\ldots,n+1\}\setminus \{c\}}$ is $\omega_{\widehat{D}\setminus \{c\}}$ up to the sign implied by the reordering of the chords in $\widehat{D}\setminus\{c\}$.
		\end{proof}
		
		\begin{rem}
		All the above relations are special cases of the \enquote{scissors congruence relations} between Aomoto polylogarithms \cite[2.1]{bvgs}. The translation relation $(R1)$ and the rotation relation $(R2)$ are special cases of projective invariance under particular subgroups of $\mathrm{PGL}_{n+1}(\C)$. Stokes' theorem $(R3)$ is a particular case of the intersection additivity relation with respect to $M$, which has been shown \cite[Proposition 2.4]{zhaogeneric} to follow from the scissors congruence relations. The Orlik-Solomon relation $(R4)$ is a particular case of the additivity relation with respect to $L$.
		\end{rem}
		
	\subsection{Reduction to iterated integrals}		
		
		\begin{thm}\label{reductionitint}
		Let $D$ be a generic decorated dissection diagram. Then the dissection polylogarithm $I(D)$ can be written as a linear combination with integer coefficients of generic iterated integrals $\mathbb{I}(a_0;a_1,\ldots,a_n;a_{n+1})$ where the $a_i$'s are linear combinations with integer coefficients of the decorations of $D$.
		\end{thm}	
		
		\begin{proof}
		It is enough to prove that using relations $(R2)$, $(R3)$, $(R4)$, one can write $I(D)$ as a linear combination with integer coefficients of dissection polylogarithms $I(X)$ for $X$ a corolla with generic decorations as in the statement of the theorem. Indeed, using relation $(R1)$, one can always perform a change of variables so that any $I(X)$ is an iterated integral.\\
		Because the chords of $D$ do not cross, at least one chord has to connect consecutive vertices of $\Pi_n$. Thus, using relation $(R2)$, one can assume that in $D$ there is a chord between $1$ and the root.\\
		We prove by induction on $k=1,\ldots,n$ that using relations $(R3)$ and $(R4)$, one may write $I(D)$ as a linear combination with integer coefficients of generic dissection polylogarithms involving dissection diagrams where the first $k$ non-root vertices are linked to the root, with decorations as in the statement of the theorem. The case $k=1$ has already been settled, and the case $k=n$ gives the theorem.\\
		Let us suppose that in $D$ all vertices between $1$ and $k$ are linked to the root by a chord. There are two cases to consider.\\
		\textit{Case $1$}: there is no chord going to the vertex $k$. Let us then consider the chord from the vertex $v_0=k+1$. If its endpoint is the root then we are done. Else, its endpoint must be a vertex $v_1\in \{k+2,\ldots,n\}$. Let us consider the sequence of chords
		$$\stackrel{v_0}{\bullet} \stackrel{c_0}{\longrightarrow} \stackrel{v_1}{\bullet} \stackrel{c_1}{\longrightarrow} 
		\stackrel{v_2}{\bullet} \stackrel{c_2}{\longrightarrow} \cdots \stackrel{v_r}{\bullet} \stackrel{c_{r}}{\longrightarrow}\circ$$
		going to the root, where $c_i$ has decoration $\alpha_i$. Let $\widehat{D}$ be the diagram obtained by adding to $D$ a chord $c$ from $k+1$ to the root decorated by the sum $\alpha_0+\alpha_1+\alpha_2+\cdots+\alpha_r$. Then we have created a simple cycle $C=(c,c_0,c_1,\ldots,c_r)$ and we are in the situation where we can apply relation $(R4)$. Since $\widehat{D}\setminus\{c\}=D$ by definition we get		
		$$I(D)=\sum_{i=0}^r \pm I(\widehat{D}\setminus\{c_i\})$$
		 and for every $i=0,\ldots,r$,  $\widehat{D}\setminus\{c_i\}$ is a dissection diagram in which all vertices between $1$ and $k+1$ are linked to the root. Moreover its decorations are linear combinations with integer coefficients of the decorations of $D$.
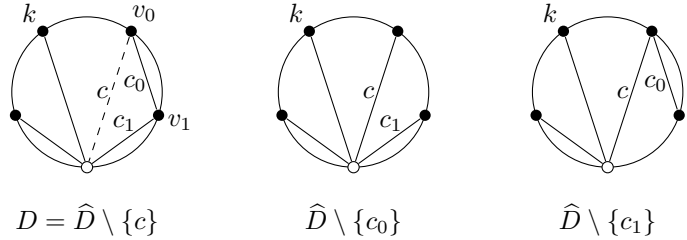
		 \begin{center}\reditintcaseone{1}\hspace{1cm}\reditintcaseonebis{1}\hspace{1cm}\reditintcaseoneter{1}\end{center}
		\textit{Case $2$}: there are chords going to the vertex $k$. We are going to use relation $(R3)$ to reduce to Case $1$. Let $\wt{D}$ be the diagram obtained by opening the angle between the chord going from $k$ to the root and the first of the chords going to $k$, as in the picture below. The decoration of the new edge is $0$. Then by definition we get $\partial_k\wt{D}=D$. The other $\partial_l\wt{D}$ that is a dissection diagram has no chord arriving at $k$. Thus, relation $(R3)$ gives
		$$I(D)=\pm I(\partial_l\wt{D}).$$ The decorations of $\partial_l\wt{D}$ are linear combinations with integer coefficients of the decorations of $D$, and we are reduced to Case $1$.
		\begin{center}\reditintcasetwo{1}\hspace{1cm}\reditintcasetwobis{1}\hspace{1cm}\reditintcasetwoter{1}\end{center}
		\end{proof}
		
		\begin{rem}\label{remreditint}
		The algorithm defined in the above proof is not canonical in any sense. It is worth noting that the number of iterated integrals that appear in the final sum is between $1$ (for $D$ a corolla) and $(n-1)!$ (for $D$ a path graph).\\
		For these reasons, Theorem \ref{reductionitint} should be taken as a technical tool, and not as an abstract statement on the internal structure of dissection polylogarithms.
		\end{rem}

\section{Motivic dissection polylogarithms and their coproduct}

	As in the previous section, the decorations on dissection diagrams are implicitly taken in $\C$.

	\subsection{Mixed Hodge-Tate structures and the motivic Hopf algebra $\H$}
	
		We review here the tannakian formalism for the category of mixed Hodge-Tate structures. The reader may want to refer to the \cite[\S 8]{goncharovgaloissym} for more details on framed objects in a mixed Tate category, and to \cite[\S 2]{brownsinglevalued} for a clear exposition of the different frameworks where motivic periods appear.
		
			\subsubsection{The tannakian category of mixed Hodge-Tate structures}
	
			A \textit{mixed Hodge-Tate structure} is a mixed Hodge structure $H$ \cite{delignehodge3} such that for any $k$, $\gr_{2k+1}^WH=0$ and $\gr_{2k}^WH$ is a sum of copies of the Tate structure $\Q(-k)$.\\
			If $H$ is a mixed Hodge-Tate structure and $F$ denotes the Hodge filtration on $H_\C=H\otimes\C$, then for all $k$ we have a natural isomorphism
			$$W_{2k}H_\C\cap F^kH_\C\stackrel{\simeq}{\rightarrow}\gr_{2k}^WH_\C$$
			which gives a canonical splitting of the weight filtration over $\C$:
			\begin{equation}\label{canonicalsplitting}
			s_H:\bigoplus_k\gr_{2k}^WH_\C\stackrel{\simeq}{\rightarrow}H_\C.
			\end{equation}		
			
			 The category of mixed Hodge-Tate structures is denoted $\mathrm{MHTS}$. It is a tannakian category which has a canonical fiber functor
			$$\omega(H)=\bigoplus_{k}\gr_{2k}^WH$$
			with values in the category of finite-dimensional vector spaces.\\
			It follows that we have an equivalence of categories
			\begin{equation}\label{MHTSrepG}
			\mathrm{MHTS}\cong \mathrm{Rep}(G)
			\end{equation}
			between $\mathrm{MHTS}$ and the category of finite-dimensional representations of a group scheme $G$. One easily sees \cite[3.1]{goncharovmultiplepolylogs} that we have a semi-direct product decomposition
			\begin{equation}\label{semidirect}
			G=\mathbb{G}_m\ltimes U
			\end{equation}
			where $\mathbb{G}_m$ is the multiplicative group and $U$ is a pro-unipotent group scheme.
			
			\subsubsection{The fundamental Hopf algebra $\H$}
		
			Let $\H$ be the Hopf algebra of functions on $U$. Since $\mathbb{G}_m$ acts on $U$, $\H$ is graded. It follows from (\ref{MHTSrepG}) and (\ref{semidirect}) that $\mathrm{MHTS}$ is equivalent to the category of finite-dimensional graded comodules on $\H$:
			$$\mathrm{MHTS}\cong \mathrm{grComod}(\H).$$

			The fact that the extension groups $\mathrm{Ext}_{\mathrm{MHTS}}^1(\Q(0),\Q(-n))$ are $0$ for $n\geq 0$ implies that $\H$ is positively graded and connected:
			$$\H=\bigoplus_{n\geq 0}\H_n \,\,\,\, , \,\,\,\, \H_0=\Q.$$
		
			An element of $\H_n$ is an equivalence class of triples $(H,v,\varphi)$ where 
			\begin{itemize}
			\item $H$ is a mixed Hodge-Tate structure,
			\item $v\in \gr^W_{2n}H$,
			\item $\varphi\in \left(\gr^W_0H\right)^\vee$.
			\end{itemize}
		
			The equivalence relation is generated by: $(H,v,\varphi)\equiv (H',v',\varphi')$ if there exists a morphism of mixed Hodge-Tate structures $f:H\rightarrow H'$ such that $\gr^W_{2n}f(v)=v'$ and $\varphi'\circ\gr^W_0f=\varphi$.\\
			A triple $(H,v,\varphi)$ as above is called an \textit{$n$-framed mixed Hodge-Tate structure}, $v$ and $\varphi$ being called the \textit{framings}. The expression $(H,v,\varphi)$ is linear in $v$ and $\varphi$.\\
		
			 The product in $\H$ is defined via the tensor product: 
		 $$(H,v,\varphi)(H',v',\varphi')=(H\otimes H',v\otimes v',\varphi\otimes\varphi').$$
		 
			 The coproduct $\Delta_{n-k,k}:\H_n\rightarrow\H_{n-k}\otimes \H_k$ is abstractly defined by the formula
			 \begin{equation}\label{defabstractcoproduct}
			 \Delta_{n-k,k}(H,v,\varphi)=\sum_i (H(k),v,b_i^\vee)\otimes (H,b_i,\varphi)
			 \end{equation} 
			 where $(b_i)$ is any basis of $\gr_{2k}^WH$ and $(b_i^\vee)$ the dual basis.\\
		 
			\subsubsection{A variant: the algebra $\P$ of Hodge-Tate periods}
		
			Parallel to the above construction of $\H$ is the construction of the algebra $\P$ of Hodge-Tate periods. It is a graded algebra $$\P=\bigoplus_{n\geq 0} \P_n.$$
			An element of $\P_n$ is an equivalence class of triples $(H,v,\delta)$ where
			\begin{itemize}
			\item $H$ is a mixed Hodge-Tate structure with non-negative weights: $W_{-1}H=0$,
			\item $v\in \gr_{2n}^WH$,
			\item $\delta\in H^\vee$.
			\end{itemize}
			The product is defined in the same way as in $\H$.\\
			We may define maps $\rho_{n-k,k}:\P_n\rightarrow \H_{n-k}\otimes\P_k$ by the formula
			$$\rho_{n-k,k}(H,v,\delta)=\sum_i (H(k),v,b_i^\vee)\otimes (H,b_i,\delta)$$
			where $(b_i)$ is any basis of $\gr_{2k}^WH$ and $(b_i^\vee)$ the dual basis. They endow $\P$ with the structure of a graded comodule over $\H$. This coaction is dual to an action of the group scheme $G$ on the algebra $\P$ of Hodge-Tate periods.\\
		
			 There is a surjective morphism of graded algebras
			\begin{equation}\label{surjectionPH}
			\P\twoheadrightarrow \H
			\end{equation}
			which sends $(H,v,\delta)$ to $(H,v,\varphi)$ where $\varphi$ is the image of $\delta$ via the map $H^\vee\twoheadrightarrow\left(\gr_0^WH\right)^\vee$. This is well-defined since by assumption $W_{-1}H=0$. The surjection (\ref{surjectionPH}) is compatible with the structures of graded comodules over $\H$.\\
			
			There is a morphism of algebras 
			$$per:\P\rightarrow\C$$
			called the \textit{period map} which is defined by
			$$per(H,v,\delta)=\langle\delta, s_H(v)\rangle$$
			where $s_H$ is the canonical splitting (\ref{canonicalsplitting}).
			
			\begin{rem}
			We compare our framework with the notation of \cite[\S 2]{brownsinglevalued}. We set $\mathcal{M}=\mathrm{MHTS}$, $\omega_{dR}=\omega$, and $\omega_B$ is the natural forgetful functor from $\mathrm{MHTS}$ to the category of finite-dimensional vector spaces (that forgets the weight and Hodge filtrations). Then in \cite{brownsinglevalued} the Hopf algebra $\H$ is denoted by $\mathcal{P}^{\mathfrak{a}}$, the algebra $\P$ of Hodge-Tate periods is denoted by $\P^{\mathfrak{m},+}$, and the surjection (\ref{surjectionPH}) is denoted by $\pi_{\mathfrak{a},\mathfrak{m}+}$.
			\end{rem}
		
	\subsection{Motivic dissection polylogarithms}\label{sectionmotivicperiods}
	
		Let $D$ be a generic decorated dissection diagram. Let $(L;M)=(L_1,\ldots,L_n;M_0,M_1,\ldots,M_n)$ be the bi-arrangement in $\C^n$ corresponding to $D$ (see \ref{defbiarrangementdd}). With the notation of \ref{paragraphrelativeBOS}, we set
		$$H(D)=H(L;M).$$ 
		According to Theorem \ref{relativeBOS}, it is a mixed Hodge-Tate structure with non-negative weights between $0$ and $2n$ and we have
		\begin{itemize}
		\item $\gr^W_{2n}H(D)\cong \Lambda^n(e_1,\ldots,e_n)$ which is one-dimensional with basis $$v(D)=e_1\wedge\cdots\wedge e_n.$$
		\item $\gr^W_0H(D)$ is isomorphic to the quotient of $\Lambda^n(f_0,f_1,\ldots,f_n)$ by the vector space spanned by the elements
		$$(-1)^if_0\wedge\cdots\wedge\widehat{f_i}\wedge\cdots\wedge f_n -(-1)^j f_0\wedge\cdots\wedge\widehat{f_j}\wedge\cdots\wedge f_n$$ for $0\leq i<j\leq n$. Hence it is one-dimensional with basis $f_1\wedge\cdots\wedge f_n$. We let 
		$$\varphi(D)=f_1^\vee\wedge\cdots\wedge f_n^\vee $$
		be the dual linear form in $\left(\gr^W_0H(D)\right)^\vee$. 
		\end{itemize}
		
		\begin{defi}\label{defmotivicdissectionpolylog}
		We set $$I^\H(D)=(H(D),v(D),\varphi(D))\in\H_n$$ and call it the \textit{motivic dissection polylogarithm} corresponding to $D$.
		\end{defi}
		
		More geometrically, we get
		$$\gr_{2n}^WH(D)\cong H^n(\C^n\setminus L)=\Q[\omega_D]$$
		so that $v(D)$ is the cohomology class of the $n$-form $\omega_D$. We also have a commutative diagram
		
		$$
		\xymatrix{
		H(D)^\vee \ar[r]^>>>>>>>{\cong} \ar@{->>}[d] &H_n(\C^n\setminus L, M\setminus M\cap L) \ar@{->>}[d]^{\mu}\\
		\left(\gr_0^WH(D)\right)^\vee \ar[r]^{\cong}  & H_n(\C^n,M)
		}
		$$		
		Let $\Delta_D$ be any integration simplex for $I(D)$, and $[\Delta_D]$ its homology class in 	$H_n(\C^n\setminus L, M\setminus M\cap L)$. Then 
		$\varphi(D)=\mu\left([\Delta_D]\right)\in H_n(\C^n,M)$ is canonical and does not depend on the choice of $[\Delta_D]$. It corresponds to an oriented simplex in $\C^n$ whose boundary is contained in $M$.\\
		
		To sum up, we have
		$$I^\H(D)=\left(H(D),[\omega_D],\mu\left([\Delta_D]\right)\right).$$
		
		\begin{rem}\label{remP}
		In a particular situation where one has a preferred choice of $[\Delta_D]$, then a more natural thing to do is to work in the algebra $\P$ and not in $\H$. A candidate for the motivic dissection polylogarithm is then
		$$I^\P(D)=\left(H(D),[\omega_D],[\Delta_D]\right) \in\P_n.$$
		Its period is $$per(I^\P(D))=I(D)$$ computed with the same choice of $\Delta_D$.\\
		Via the surjection (\ref{surjectionPH}), $I^\P(D)$ is mapped to $I^\H(D)$. We stress the fact that $I^\H(D)$ only depends on the (generic) decorated diagram $D$, whereas $I(D)$ and $I^\P(D)$ also depend on the choice of a homology class $[\Delta_D]\in H_n(\C^n\setminus L,M\setminus M\cap L)$.
		\end{rem}
		
		\begin{thm}
		The relations $(R1)$, $(R2)$, $(R3)$, $(R4)$ from Propositions \ref{r1}, \ref{r2}, \ref{r3}, \ref{r4} remain true if we replace the dissection polylogarithms $I(D)$ by their motivic versions $I^\H(D)$. Thus, Theorem \ref{reductionitint} is also true in the motivic setting.
		\end{thm}
		
		\begin{rem}
		Of course, Remark \ref{remreditint} also applies in this setting. Furthermore, it is worth noting that the reduction to iterated integrals does not tell us anything about the combinatorial shape of the coproduct of the motivic dissection polylogarithms (Theorem \ref{maintheorem} below).
		%The above Theorem gives evidence for Conjecture 17 of \cite{goncharovpolylogsarithmeticgeometry}.
		\end{rem}
		
		\begin{proof}
		We will not use this result in the sequel so we just sketch the proof that $(R3)$ is true in the motivic setting. Let $(L;M)=(L_1,\ldots,L_n;M_0,M_1,\ldots,M_{n+1})$ be the bi-arrangement of hyperplanes given by $\wt{D}$. 
		By definition $\varepsilon_i\,I^\H(\partial_i\wt{D})$ is the triple
		$$(H(M_i|L;M_0,\ldots,\widehat{M_i},\ldots,M_{n+1}),e_1\wedge\cdots\wedge e_n,f_1^\vee\wedge\cdots\wedge \widehat{f_i^\vee}\wedge\cdots\wedge f_{n+1}^\vee).$$
		Thus, using the natural morphism
		$$H(M_i|L;M_0,\ldots,\widehat{M_i},\ldots,M_{n+1})\rightarrow H(L;M)$$
		from Theorem \ref{functoriality}, we see that this triple is equivalent to 
		$$(H(L;M),(e_1\wedge\cdots\wedge e_n)\otimes f_i,f_i^\vee\wedge f_1^\vee\wedge\cdots\wedge \widehat{f_i^\vee}\wedge\cdots\wedge f_{n+1}^\vee),$$
		hence $(-1)^i\varepsilon_i\,I^\H(\partial_i\wt{D})$ is the triple
		$$(H(L;M),(e_1\wedge\cdots\wedge e_n)\otimes f_i,-f_1^\vee\wedge\cdots\wedge f_{n+1}^\vee)$$
		and the sum $(-1)^i\varepsilon_i\,I^\H(\partial_i\wt{D})+(-1)^j\varepsilon_j\,I^\H(\partial_j\wt{D})$ is the triple
		$$(H(L;M),(e_1\wedge\cdots\wedge e_n)\otimes (f_i+f_j),-f_1^\vee\wedge\cdots\wedge f_{n+1}^\vee).$$
		Thus it is enough to prove that $(e_1\wedge\cdots\wedge e_n)\otimes (f_i+f_j)=0$.\\
		For $s\notin\{i,j\}$, $(\{1,\ldots,n\};\{s\})$, $L_1\cap\cdots\cap L_n\cap M_s=\varnothing$ because the corresponding subgraph in the total graph of $\wt{D}$ has a cycle. Hence the first relation of Theorem \ref{relativeBOS} gives $(e_1\wedge\cdots\wedge e_n)\otimes f_s=0$ and
		$$(e_1\wedge\cdots\wedge e_n)\otimes (f_i+f_j)=(e_1\wedge\cdots\wedge e_n)\otimes\sum_{s=0}^{n+1}f_s=0$$
		using the second relation.
		\end{proof}
		
		\begin{rem}
		The above theorem is also valid if we work in $\P$ with the elements $I^\P(D)$ (see Remark \ref{remP}; in this setting the integration simplices have to be chosen coherently as in \S\ref{sectionrelations}).
		\end{rem}
	
	\subsection{The computation of the coproduct}

		\begin{prop}\label{propbasis}
		Let $D$ be a generic decorated dissection diagram and $k\in\{0,\ldots,n\}$. The classes of the elements $$b_C=e_C\otimes f_{\s_C^+}$$ for $C\subset\c\simeq\{1,\ldots,n\}$, $|C|=k$, form a basis of $\gr_{2k}^WH(D)$.
		\end{prop}
	
		\begin{proof}
		From Theorem \ref{relativeBOS} we get a presentation
		$$\gr^W_{2k}H(L;M)\cong \bigoplus_{|I|=k}\Q e_I\otimes \left(\Lambda^{n-k}(f_1,\ldots,f_m)/R_I(L;M)\right)$$
		where $R_I(L;M)$ is spanned by the elements
		\begin{enumerate}
		\item $f_J$ if $L_I\cap M_J=\varnothing$.
		\item $\displaystyle\sum_{j\notin J'}\sgn(\{j\},J')f_{J'\cup\{j\}}$ for $|J'|=n-k-1$.
		\end{enumerate}
		Let us fix a subset $I\subset\{1,\ldots,n\}$, viewed as a subset $C\subset\c$ of chords of $D$. We want to prove that the quotient of $\Lambda^{n-k}(f_0,f_1,\ldots,f_n)$ by relations 1. and 2. above is one-dimensional with basis element $f_{\s_C^+}$.
		\begin{enumerate}
		\item
		Let us write $\{0,\ldots,n\}=\s=\s_C(0)\sqcup\cdots\sqcup \s_C(k)$ the partition (\ref{partition}) of $\s$ given by the dissection defined by $C$.		
		 From Lemma \ref{lemrelationscycles} and Lemma \ref{lemSTacyclic}, we see that the only subsets $J\subset\{0,\ldots,n\}$ such that $L_I\cap M_J\neq \varnothing$ are 
		$$J(u_0,\ldots,u_k)=(\s_C(0)\setminus\{u_0\})\sqcup\cdots\sqcup (\s_C(k)\setminus\{u_k\})$$
		for some choice of $u_\alpha\in\s_C(\alpha)$.\\
		Thus the quotient of $\Lambda^{n-k}(f_0,f_1,\ldots,f_n)$ by relation $1.$ has a natural basis consisting of the elements $f_{J(u_0,\ldots,u_k)}$.
		\item Let us write $f(u_0,\ldots,u_k)=f_{J(u_0,\ldots,u_k)}$ for simplicity. We investigate the relations between the elements $f(u_0,\ldots,u_k)$ implied by relation $2$. The only non-trivial ones come from subsets
		$$J'=(\s_C(0)\setminus\{u_0\})\sqcup\cdots \sqcup(\s_C(i)\setminus\{a_i,b_i\})\sqcup\cdots\sqcup (\s_C(k)\setminus\{u_k\})$$
		with $a_i\neq b_i$, and are of the form
		\begin{equation}\label{formsigns}
		\sgn(\{a_i\},J')f(u_0,\ldots,a_i,\ldots,u_k)+\sgn(\{b_i\},J')f(u_0,\ldots,b_i,\ldots,u_k)=0
		\end{equation}
		Hence in the quotient of $\Lambda^{n-k}(f_0,\ldots,f_n)$ by relations 1. and 2., all the elements $f(u_0,\ldots,u_k)$ are equal up to a sign, hence this quotient is spanned by any of these elements. If we choose $u_\alpha=\min(\s_C(\alpha))$ for each $\alpha$, we get $J(u_0,\ldots,u_k)=\s_C^+$ by definition. Thus all there is to prove is that the elements $f(u_0,\ldots,u_k)$ are all non-zero in the quotient. This follows from a compatibility between the signs in formula (\ref{formsigns}), which is the content of the next lemma.
		\end{enumerate}
		\end{proof} 
		
		\begin{lem}\label{lemsignsgraphs}
		Let us define a graph whose vertices are the tuples $(u_0,\ldots,u_k)$ with $u_\alpha\in \s_C(\alpha)$ for every $\alpha=0,\ldots,k$. We put an edge between the pairs of the form $(u_0,\ldots,a_i,\ldots,u_k)$ and $(u_0,\ldots,b_i,\ldots,u_k)$ for $a_i\neq b_i$ in $\s_C(i)$. Let us decorate such an edge by the sign
		$$-\sgn(\{a_i\},J')\,\sgn(\{b_i\},J')$$ with $J'=(\s_C(0)\setminus\{u_0\})\sqcup\cdots \sqcup(\s_C(i)\setminus\{a_i,b_i\})\sqcup\cdots\sqcup (\s_C(k)\setminus\{u_k\})$.\\
		Then for every loop in this graph, the product of the signs of the edges of the loop is $1$.
		\end{lem}
		
		\begin{proof}
		See \ref{appB}.
		\end{proof}
	
		\begin{thm}\label{maintheorem}
		The coproduct of the motivic dissection polylogarithms is given by the formula 
		\begin{equation}\label{maintheoremformula}
		\Delta(I^\H(D))=\sum_{C\subset \c(D)}(-1)^{k_C(D)}I^\H(q_C(D))\otimes I^\H(r_C(D))
		\end{equation}
		where $I^\H(q_C(D))$ is understood as the product $\prod_\alpha I^\H(q_C^\alpha(D))$.\\
		In other words, the morphism $$\D^{gen}(\C)\rightarrow \H \,\, , \,\, D\mapsto I^\H(D)$$ is a morphism of graded Hopf algebras.
		\end{thm}
	
		\begin{proof}
		According to formula (\ref{defabstractcoproduct}) and Proposition \ref{propbasis}, we get
		$$\Delta_{n-k,k}(I^\H(D))=\sum_{\substack{C\subset\{1,\ldots,n\}\\|C|=k}} (H(D)(k),v(D),b_C^\vee)\otimes(H(D),b_C,\varphi(D))$$ 
		\begin{enumerate}
		
		\item We show that $(H(D)(k),v(D),b_C^\vee)=\pm I^\H(q_C(D))$.\\
		First, let us look at the bi-arrangement $$(L_C|L(\overline{C});M).$$ By Lemma \ref{lembiarrangementsRQ} and the K\"unneth isomorphism, we have an isomorphism
		$$H(L_S|L(\overline{S});M)\cong \bigotimes_\alpha H(q_S^\alpha(D))$$
		hence the graded $0$ part 
		$$\gr_0^WH(L_C|L(\overline{C});M)$$
		is one-dimensional and spanned by the vector $\bigwedge_\alpha f_{\s_C^+(\alpha)}=\pm f_{\s_C^+}$.\\
		Let us consider the residue morphism (Theorem \ref{functoriality})
		$$H(D)(k)\rightarrow H(L_C|L(\overline{C});M).$$
		On the $\gr^W_{2(n-k)}$ part, it sends $v(D)=e_1\wedge\cdots\wedge e_n$ to $\sgn(\overline{C},C)e_{\overline{C}}$.\\
		On the $\gr^W_0$ part, it sends $b_C=e_C\otimes f_{\s_C^+}$ to $f_{\s_C^+}$ and all the other basis elements $b_{C'}$ to $0$.\\
		Thus it gives an identification
		$$(H(D)(k),v(D),b_C^\vee)=\sgn(\overline{C},C)(H(L_C|L(\overline{C});M),e_{\overline{C}},f_{\s_C^+}^\vee).$$
		
		For each $\alpha$, let $\nu^\alpha_C:\overline{C}(\alpha)\stackrel{\simeq}{\rightarrow}\s_C^+(\alpha)$ be the bijection (\ref{conventionbijection}) given by the dissection diagram $q_C^\alpha(D)$, and let $\nu_C:\overline{C}\stackrel{\simeq}{\rightarrow}\s_C^+$ be the bijection induced by the $\nu^\alpha_C$'s. This bijection accounts for the reordering of the hyperplanes, and gives a sign
		$$(H(L_C|L(\overline{C});M),e_{\overline{C}},f_{\s_C^+}^\vee)=\sgn(\nu_C)\prod_\alpha I^\H(q_C^\alpha(D))$$
		hence the equality
		$$(H(D)(k),v(D),b_C^\vee)=\sgn(\overline{C},C)\,\sgn(\nu_C) I^\H(q_C(D)).$$
		
		\item We show that $(H(D),b_C,\varphi(D))=\pm I^\H(r_C(D))$.\\		
		First let us consider the bi-arrangement of hyperplanes 
		$$\left(\left. M_{\s_C^+}\right\vert L(C);M_0,M\left(\overline{\s_C^+}\right)\right).$$
		According to Lemma \ref{lembiarrangementsRQ}, it is exactly the one given by the dissection diagram $r_C(D)$, so we get
		$$H\left(\left. M_{\s_C^+}\right\vert L(C);M_0,M\left(\overline{\s_C^+}\right)\right)\cong H(r_C(D))$$	
		and the graded $0$ part
		$$\gr_0^WH\left(\left. M_{\s_C^+}\right\vert L(C);M_0,M\left(\overline{\s_C^+}\right)\right)$$
		is one-dimensional and spanned by the vector $f_{\s_C^+}$.\\
		Let us consider the morphism (Theorem \ref{functoriality})
		\begin{equation}\label{morphismface}
		H\left(\left. M_{\s_C^+}\right\vert L(C);M_0,M\left(\overline{\s_C^+}\right)\right)\rightarrow H(L;M).
		\end{equation}
		On the $\gr_{2k}^W$ part, it sends $e_C$ to $b_C=e_C\otimes f_{\s_C^+}$.\\
		On the $\gr_0^W$ part, it sends $f_{\s_C^+}$ to $\sgn(\s_C^+,\overline{\s_C^+})f_1\wedge\cdots\wedge f_n$.\\
		Thus it gives an identification
		$$(H(D),b_C,\varphi(D))=\sgn(\s_C^+,\overline{\s_C^+}) \left(H\left(\left. M_{\s_C^+}\right\vert L(C);M_0,M\left(\overline{\s_C^+}\right)\right),e_C,f_{\s_C^+}^\vee\right).$$ 
		Because of the ordering conventions, we have 
		$$\left(H\left(\left. M_{\s_C^+}\right\vert L(C);M_0,M\left(\overline{\s_C^+}\right)\right),e_C,f_{\s_C^+}^\vee\right)=\sgn(\eta_C)I^\H(r_C(D))$$
		where $\eta_C:C\stackrel{\simeq}{\rightarrow}\overline{\s_C^+}$ is the bijection (\ref{conventionbijection}) given by $r_C(D)$. Hence we have the equality
		$$(H(D),b_C(D),\varphi(D))=\sgn(\s_C^+,\overline{\s_C^+})\,\sgn(\eta_C) I^\H(r_C(D)).$$
		\item Putting the two first steps together, it only remains to check that the signs are correct. This is done in the next lemma.
		\end{enumerate}				
		\end{proof}
		
		\begin{lem}\label{lemproductsigns}
		We have the equality between signs:
		$$\sgn(\overline{C},C)\,\sgn(\nu_C)\,\sgn(\s_C^+,\overline{\s_C^+})\,\sgn(\eta_C) = (-1)^{k_C(D)}.$$
		\end{lem}
		
		\begin{proof}
		See \ref{appC}.
		\end{proof}
		
		\begin{rem}
		If we work in $\P$ with the elements $I^\P(D)$ (see Remark \ref{remP}) then we get a similar formula for the coaction $\rho:\P\rightarrow \H\otimes \P$:
		$$\rho(I^\P(D))=\sum_{C\subset \c(D)}(-1)^{k_C(D)}I^\H(q_C(D))\otimes I^\P(r_C(D)).$$
		We only have to define the integration simplices for the elements $I^\P(r_C(D))$ in a coherent way. If $\Delta_D$ is the integration simplex for $I^\P(D)$, then the integration simplex for $I^\P(r_C(D))$ has to be the face $\partial_{\s_C^+}\Delta_D$ of $\Delta_D$. This is because the morphism (\ref{morphismface}) corresponds, on the singular homology groups, to a composition of face maps.
		\end{rem}
		
	\subsection{The setting of mixed Tate motives}	\label{parMTM}
		 
		 Let $F$ be a number field, and $\mathrm{MTM}(F)$ be the category of mixed Tate motives over $F$ \cite{levinetatemotives}. Then we have a graded fiber functor
		 $$\omega(H)=\bigoplus_k\mathrm{Hom}_{\mathrm{MTM}(F)}(\Q(-k),\gr^W_{2k}H)$$
		 so that $\mathrm{MTM}(F)$ is equivalent to the category of finite-dimensional graded comodules on a graded Hopf algebra $\H_{\mathrm{MTM}(F)}$, which is also positively graded and connected:
		$$\mathrm{MTM}(F)\cong \mathrm{grComod}(\H_{\mathrm{MTM(F)}}).$$		 
		 One may also describe an element in $\H_{\mathrm{MTM}(F)}$ of degree $n$ as equivalence classes of triples $(H,v,\varphi)$ with 
		 \begin{itemize}
		 \item $H$ a mixed Tate motive over $F$, 
		 \item $v\in \mathrm{Hom}_{\mathrm{MTM}(F)}(\Q(-n),\gr^W_{2n}H)$, 
		 \item $\varphi\in\mathrm{Hom}_{\mathrm{MTM}(F)}(\gr^W_0H,\Q(0))$.
		 \end{itemize}
		 The coproduct in $\H_{\mathrm{MTM}(F)}$ is given by the abstract formula (\ref{defabstractcoproduct}).
		 
		 If $\sigma:F\hookrightarrow\C$ is an embedding of $F$ into the complex numbers, then there is a Hodge realization functor \cite{huberrealization,huberrealizationcorrigendum}
		 $$real_\sigma:\mathrm{MTM}(F)\rightarrow\mathrm{MHTS}$$
		 which gives a morphism of Hopf algebras
		 \begin{equation}\label{realsigma}
		 real_\sigma:\H_{\mathrm{MTM}(F)}\rightarrow\H.
		 \end{equation}
		 
		%Let $\Sigma$ be the set of embeddings of $F$ into $\C$; for each $\sigma\in\Sigma$ let $\H_\sigma$ be a copy of $\H$, and let $F_\Sigma\H$ be the product of all these copies (in the category of Hopf algebras). Then the morphisms $real_\sigma$ give a morphism
		%\begin{equation}\label{real}
		%\H(F)\rightarrow F_\Sigma\H
		%\end{equation}
		
		%The following Lemma is proven in \cite{goncharovmultiplepolylogs} using the Borel regulator (\cite{borel}).
		 
		 %\begin{lem}
		 %The morphism of Hopf algebras (\ref{real}) is injective.
		 %\end{lem}
		 
		 %For instance, for $F=\Q$, we get an injection $\H(\Q)\hookrightarrow\H$.
		 
		If we now start with a generic $F$-decorated dissection diagram $D$, then we may define a bi-arrangement $(L;M)$ inside $\mathbb{A}^n_F$, the $n$-dimensional affine space over $F$. Then $H^n(\mathbb{A}^n_F\setminus L,M\setminus M\cap L)$ defines an object in the category $\mathrm{MTM}(F)$ (see \cite[Proposition 3.6]{goncharovperiodsmm}). We may then copy Definition \ref{defmotivicdissectionpolylog} to define elements $I^{\H_{\mathrm{MTM}(F)}}(D)\in \H_{\mathrm{MTM}(F)}$ inside the Hopf algebra of mixed Tate motives over $F$.\\
		The realization morphism (\ref{realsigma}) maps $I^{\H_{\mathrm{MTM}(F)}}(D)\in \H_{\mathrm{MTM}(F)}$ to $I^\H(\sigma(D))\in \H$ where $\sigma(D)$ is the $\C$-decorated diagram obtained from $D$ by applying $\sigma$ to all the decorations.\\
		
		Formula (\ref{maintheoremformula}) for the coproduct of the motivic dissection polylogarithms is also valid in the setting of mixed Tate motives for the elements $I^{\H_{\mathrm{MTM}(F)}}(D)$, and gives a morphism of graded Hopf algebras 
		$$\D^{gen}(F)\rightarrow \H_{\mathrm{MTM}(F)}.$$
	
	\subsection{Examples of computations}
	
		We present two special cases of Theorem \ref{maintheorem}.\\
		
		In \ref{exdissectionpolylogs} we have introduced the iterated integrals $\I(a_0;a_1,\ldots,a_n;a_{n+1})$ as the dissection polylogarithms corresponding to corollas. The motivic counterparts $\I^\H(a_0;a_1,\ldots,a_n;a_{n+1})\in \H_n$ have already been defined and studied by Goncharov in \cite[Theorem 1.1]{goncharovgaloissym}, in the framework of motivic fundamental groupoids. We leave it to the reader to check that Goncharov's definition agrees with ours. The coproduct of the motivic iterated integrals has been worked out by Goncharov.

		\begin{thm}[See \cite{goncharovgaloissym}, Theorem 1.2.]\label{coproditint}
		The coproduct of motivic generic iterated integrals is given by the formula
		\begin{eqnarray*}
		\Delta(\mathbb{I}^{\mathcal{H}}(a_0;a_1,\ldots, a_n;a_{n+1})) & = & \sum_{\substack{0\leq k\leq n\\0=i_0<i_1<\cdots<i_k<i_{k+1}=n+1}}  \\
		\prod_{s=0}^k \mathbb{I}^{\mathcal{H}}(a_{i_s};a_{i_s+1},\ldots, a_{i_{s+1}-1};a_{i_{s+1}}) & \otimes & \mathbb{I}^{\mathcal{H}}(a_0;a_{i_1},\ldots, a_{i_k};a_{n+1}).
		\end{eqnarray*}
		\end{thm}
		
		\begin{proof}
		It is the same computation as in Example \ref{excoproduct}, $1.$, but taking care of the decorations. The term indexed by $0=i_0<i_1<\cdots<i_k<i_{k+1}=n+1$ corresponds to the subset $C=\{i_1,\ldots,i_k\}$.
		\end{proof}
		
		In \ref{exdissectionpolylogs} we have introduced the $\mathbb{J}$-polylogarithms $\mathbb{J}(b;a_1,\ldots,a_n)$ as the dissection polylogarithms corresponding to path trees. We let  $\mathbb{J}^\H(b;a_1,\ldots,a_n)\in \H_n$ be their motivic counterparts. Their coproduct is given by a simple formula. 

		\begin{thm}
		The coproduct of motivic generic $\mathbb{J}$-polylogarithms is given by the formula 
		\begin{equation}
		\Delta(\mathbb{J}^\H(a_1,\ldots,a_n;b))=\sum_{I\subset\{1,\ldots,n\}}\mathbb{J}^\H(a(\overline{I});b-a_I)\otimes \mathbb{J}^\H(a(I);b).
		\end{equation}
		\end{thm}
		
		\begin{proof}
		It is the same computation as in Example \ref{excoproduct}, $2.$, but taking care of the decorations. Here we have to make a slight translation of variables on the left-hand side of the tensor product so that it looks like the above formula. The details are left to the reader.
		\end{proof}
		
	\subsection{Genericity and regularization}
	
		In this paragraph we discuss the extension of our results to non-generic dissection diagrams/polylogarithms. The genericity condition on the decorations of a dissection diagram is a sufficient, but not necessary condition, for the existence of the corresponding dissection polylogarithm.\\
		
		Let us take the example of the iterated integrals $\I(a_0;a_1,\ldots,a_n;a_{n+1})$, for which the genericity condition reads $a_i\neq a_j$ for $i\neq j$. The convergence of the corresponding integral is actually guaranteed as soon as $a_0\neq a_1$ and $a_n\neq a_{n+1}$. For example, the multiple zeta values
		$$\zeta(n_1,\ldots,n_r)=\sum_{1\leq k_1<\cdots< k_r}\frac{1}{k_1^{n_1}\cdots k_r^{n_r}}$$
		defined for integers $n_1,\ldots,n_{r-1}\geq 1$ and $n_r\geq 2$, are special cases of these non-generic iterated integrals, as was first noticed by Kontsevich:
		$$\zeta(n_1,\ldots,n_r)=(-1)^r(2i\pi)^n\,\,\mathbb{I}(0;\underbrace{1,0,\ldots,0}_{n_1},\ldots,\underbrace{1,0,\ldots,0}_{n_r};1)$$
		for $n=n_1+\cdots+n_r$.\\
		The point is that in the formula for the coproduct of motivic iterated integrals (Theorem \ref{coproditint}), there may be non-convergent motivic iterated integrals on the right-hand side even if the left-hand side corresponds to a convergent one. For example, for $\I(0;1,0;1)=-(2i\pi)^{-1}\zeta(2)$, the formula would look like
		\begin{equation}\label{coprodzeta}
		\Delta_{1,1}(\I^\H(0;1,0;1))=\I^\H(1;0;1)\otimes\I^\H(0;1;1)+\I^\H(0;1;0)\otimes\I^\H(0;0;1).
		\end{equation}
		Goncharov showed that there is a regularization procedure that gives a meaning to (possibly non-convergent) iterated integrals $\I(a_0;a_1,\ldots,a_n;a_{n+1})$ (for all tuples $(a_0,\ldots,a_{n+1})$).\\
		Furthermore, he defined their motivic versions $\I^\H(a_0;a_1,\ldots,a_n;a_{n+1})$ and proved that Theorem \ref{coproditint} was valid without the genericity hypothesis. Thus, formula (\ref{coprodzeta}) makes sense (and in this particular case, the right-hand side is $0$).\\
		
		Building upon Goncharov's construction (see also \cite[\S 4]{goncharovperiodsmm}), one should be able to regularize all dissection polylogarithms and compute the coproduct of their motivic versions. The most naive hope would be that the formula for the coproduct would remain the same, hence extending Theorem \ref{maintheorem} to a morphism of Hopf algebras $\D(\C)\rightarrow \H$.
	
\appendix
	\section{Proof of Lemma \ref{lemhopfalgebra}}\label{appA}
	
	In this appendix we fix a dissection diagram $D$ of degree $n$. We use the identifications $\c=\{1,\ldots,n\}$, $\s=\{0,\ldots,n\}$ and $\s^+=\{1,\ldots,n\}$ for the sets of chords and sides of $D$.
	
	\begin{lem}\label{lemcns1}
	Let $C\subset\c$ be a subset of chords of $D$ and $c=\stackrel{i_0}{\bullet}\longrightarrow\stackrel{i_1}{\bullet}$ be a chord in $C$.\\
	Then $c$ is in $\k_C(D)$ if and only if the three following conditions are satisfied:
	\begin{enumerate}[(K1)]
	\item The path in $C$ $$\stackrel{i_0}{\bullet}\stackrel{c}{\longrightarrow} \stackrel{i_1}{\bullet}\longrightarrow\cdots\longrightarrow\stackrel{i_{M-1}}{\bullet}\longrightarrow \stackrel{i_M}{\bullet}$$ starting at $i_0$ does not go to the root.
	\item This path is decreasing: for all $k=1,\ldots,M$ we have $i_{k-1}>i_k$.
	\item For all $k=1,\ldots,M$, there is no chord $\stackrel{j}{\bullet}\longrightarrow\stackrel{i_k}{\bullet}$ in $C$ such that $j>i_0$.
	\end{enumerate}
	In particular, we have $i_0>i_1$, so that all the chords in $\k_C(D)$ are decreasing.
	\end{lem}	
	
	\begin{ex}\label{exampleappendix}
	In the following example, we have only drawn the chords from $C=\{1,3,4,5,6,8,9,10\}$, and drawn the circle with dots for a matter of comfort.
	\begin{center}\exappendixA{2}\end{center}
	We have $\s_C(D)=\{6,8,11\}$ and $\k_C(D)$ is made of the arrows $\,\,\stackrel{6}{\bullet}\longrightarrow\stackrel{4}{\bullet}\longrightarrow\stackrel{3}{\bullet}\longrightarrow\stackrel{2}{\bullet}\,\,$ and $\,\,\stackrel{8}{\bullet}\longrightarrow\stackrel{7}{\bullet}$.
	\end{ex}
	
	\begin{proof}
	For $c$ a chord in $C$, we denote by $(K)$ the conjunction of the three conditions $(K1)$, $(K2)$, $(K3)$ of the lemma.\\
	We investigate the process of contracting the edges from $\s^+_C$ decomposing it into steps where we contract only one edge. The number of steps is $r=n-|C|$. We label $e_1,\ldots,e_r$ the elements of $\s^+_C$, in decreasing order.\\
	Let $D^{(0)}$ be the diagram obtained from $D$ by forgetting the chords from $\overline{C}$ and only keeping the chords from $C$. The chords form a disjoint union of rooted trees.\\
	We define recursively diagrams $D^{(i)}$, $i=1,\ldots,r$.
	For $i=1,\ldots,r$, let $D^{(i)}$ be the diagram obtained from $D^{(i-1)}$ by contracting the side $e_i$, and possibly flipping chords so that the chords in $D^{(i)}$ still form a disjoint union of rooted trees. The number of connected components of this disjoint union decreases with $i$, and in the end we get a dissection diagram $D^{(r)}=r_C(D)$.\\
	We prove the following property by induction on $i=0,\ldots,r$:
	\begin{enumerate}[(i)]
	\item In the diagram $D^{(i)}$, among the chords that are attached to the root, the ones that have been flipped have only been flipped once, and they are exactly the ones that satisfy condition $(K)$.
	\item For a chord that is not attached to the root, it satisfies $(K)$ in $D$ if and only if it satisfies $(K)$ in $D^{(i)}$.
	\end{enumerate}
	The case $i=0$ is trivial, and the case $i=r$ will give the lemma. Hence we only need to pass from $(i-1)$ to $i$.\\
	Let us consider the diagram $D^{(i-1)}$ and let $m$ be the starting vertex of the side $e_i$. We assume that the end vertex of $e_i$ is the root of $D^{(i-1)}$, leaving to the reader the (very similar) case where it is another non-root vertex $(m+1)$.
	Let us denote
	$$\stackrel{m=m_0}{\bullet}\stackrel{}{\longrightarrow} \stackrel{m_1}{\bullet}\longrightarrow\cdots\longrightarrow\stackrel{m_{N-1}}{\bullet}\longrightarrow \stackrel{m_N}{\bullet}$$
	the (possibly empty) path in $C$ starting at $m$. 
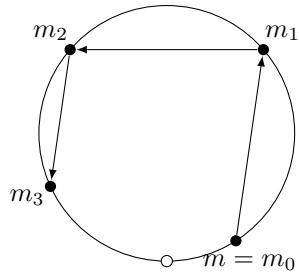
	\begin{center}\proofappendixA{1.7}\end{center}	
	When we contract $e_i$, $m$ is merged with the root and then we have to flip all these arrows. It is easy to see that they are the only ones. Hence we have to prove two things: these chords satisfy $(K)$, and all the other chords in their connected component in $D^{(i-1)}$ do not satisfy $(K)$.\\
	It is trivial that $m_0>m_1$ since $m_0$ is maximal in $D^{(i-1)}$. Since the chords cannot intersect each other, one easily proves by induction on $k$ that $m_{k-1}>m_k$ for all $k$. If the path in $C$ starting at $m$ goes to the root, then we cannot have $e_i\in \s_C^+$, which is a contradiction. Condition $(K3)$ cannot happen because $m$ is maximal in $D^{(i-1)}$. Hence we have proved that all the chords $\stackrel{m_{k-1}}{\bullet}\stackrel{}{\longrightarrow} \stackrel{m_k}{\bullet}$ satisfy $(K)$.\\
	Now let $c\in C$ be another chord in the same connected component of $D^{(i-1)}$ that satisfies $(K)$. Then $c$ lies between $m_{k}$ and $m_{k-1}$ for some $k$, or between the root and $m_{N}$. Let us suppose that we are in the first case; since the path starting with $c$ is decreasing, it has to go through $m_k$ because the chords cannot intersect each other. But then the chord $\stackrel{m_{k-1}}{\bullet}\stackrel{}{\longrightarrow} \stackrel{m_k}{\bullet}$ shows that condition $(K3)$  is not satisfied by $c$. In the second case, one sees that the path starting at $c$ has to end at the root, which is also a contradiction. Thus we are done with $(i)$. Statement $(ii)$ is straightforward since we have not affected the other connected components of $D^{(i-1)}$. This ends the induction.
	\end{proof}
	
	\begin{lem}\label{lemcns2}
	Let $C\subset\c$ be a subset of chords of $D$ and $c=\stackrel{i_0}{\bullet}\longrightarrow\stackrel{i_1}{\bullet}$ be a chord in $C$.\\
	Then $c$ is in $\k_C(D)$ if and only if there exists a path
	$$\stackrel{i_{-N}}{\bullet} \longrightarrow\cdots\longrightarrow \stackrel{i_{-1}}{\bullet}\stackrel{}{\longrightarrow}\stackrel{i_0}{\bullet}\stackrel{c}{\longrightarrow} \stackrel{i_1}{\bullet}$$
	of chords in $C$ such that $i_{-N}\in \s_C^+$.
	\end{lem}
	
	\begin{proof}
	We prove the equivalence with the condition $(K)$ of Lemma \ref{lemcns1}.\\
	If $c$ satisfies $(K)$, then we define $i_{-1}$ to be the highest vertex $>i_0$ such that there exists a chord $\stackrel{i_{-1}}{\bullet}\stackrel{}{\longrightarrow} \stackrel{i_0}{\bullet}$ in $C$, and so on. The process stops at a vertex $i_{-N}$ and we want to prove that $i_{-N}\in \s_C^+$. By construction and by condition $(K3)$, the chords $\stackrel{i_{k-1}}{\bullet}\stackrel{}{\longrightarrow} \stackrel{i_k}{\bullet}$, for $k=-N+1,\ldots,M$, are sides of the same polygon $\wt{\Pi}(\alpha)$ in the dissection defined by $C$, as well as the side labeled $i_{-N}$. 
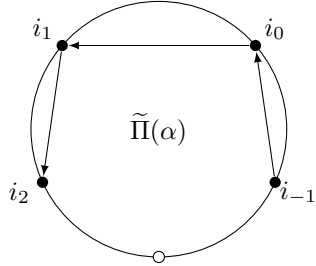
	\begin{center}\proofbisappendixA{1.7}\end{center}	
	Because of conditions $(K1)$ and $(K2)$, there is a side of this $\wt{\Pi}(\alpha)$ that is a side of $\Pi_n$  and that is less than $i_M$. Hence by definition $i_{-N}\in \s_C^+$.\\
	Conversely, under the assumption of the lemma, one easily sees that if any of conditions $(K1)$, $(K2)$, $(K3)$ is satisfied, then $i_{-N}\notin \s_C^+$.
	\end{proof}
	
	For the remainder of this appendix we use the unambiguous notation $\s_C^+=\s_C^+(D)$ to avoid any confusion.
	
	\begin{lem}\label{techlem2}
	Let $C\subset \c(D)$ be a subset of chords of $D$ and $\overline{C}=\bigsqcup_\alpha\overline{C}(\alpha)$ the partition (\ref{partition}) of $\overline{C}$ determined by $C$. Let us fix $C'_\alpha\subset \overline{C}(\alpha)$ for each $\alpha$ and $C'=C\sqcup\bigsqcup_\alpha C'_\alpha$.
	\begin{enumerate}
	\item $\s_C^+(D)=\bigsqcup_\alpha \s_{C'_\alpha}^+(q_C^\alpha(D))$.
	\item $\s_C^+(D)=\s_{C'}^+(D)\sqcup \s_C^+(r_{C'}(D))$.
	\item $\k_{C'}(D)\sqcup \k_C(r_{C'}(D))=\k_C(D)\sqcup\bigsqcup_\alpha \k_{C'_\alpha}(q_C^\alpha(D))$.
	\end{enumerate}
	\end{lem}

	\begin{proof}
	\begin{enumerate}
	\item It is straightforward, since the partition of $\s$ given by $C'$ refines the one given by $C$.
	\item The fact that $\s_{C'}^+(D)\subset \s^+_C(D)$ and $\s^+_C(r_{C'}(D))\subset\s^+_C(D)$ are easy. Then the fact that $\s^+_{C'}(D)\cap \s^+_C(r_{C'}(D))=\varnothing$ is straightforward since by definition $\s^+_C(r_{C'}(D))$ is a subset of non-root edges of $r_{C'}(D)$, which are precisely the elements from $\s^+(D)\setminus\s^+_{C'}(D)$. Then we get $\s^+_{C'}(D)\sqcup \s^+_C(r_{C'}(D))\subset\s^+_C(D)$. The equality follows from a cardinality argument: $|\s^+_C(D)|=n-|C|$, $|\s^+_{C'}(D)|=n-|C'|$ and $|\s^+_C(r_{C'}(D))|=|C'|-|C|$.
	\item Since by Lemma \ref{lemcns1} the chords that one has to flip are all decreasing, we necessarily have $\k_{C'}(D)\cap\k_C(r_{C'}(D))=\varnothing$.
		\begin{enumerate}
		\item We prove that $\k_{C'}(D)\cap C'_\alpha=\k_{C'_\alpha}(q_C^\alpha(D))$. Let  $c=\stackrel{i_0}{\bullet}\longrightarrow\stackrel{i_1}{\bullet}$ be a chord in $\k_{C'}(D)\cap C'_\alpha$, and
		$$\stackrel{i_{-N}}{\bullet} \longrightarrow\cdots\longrightarrow \stackrel{i_{-1}}{\bullet}\stackrel{}{\longrightarrow}\stackrel{i_0}{\bullet}\stackrel{c}{\longrightarrow} \stackrel{i_1}{\bullet}$$
		the path in $C'$ given by Lemma \ref{lemcns2}, with $i_{-N}\in \s^+_{C'}(D)$. As has been noted in the proof of Lemma \ref{lemcns2}, the chords $\stackrel{i_{-k+1}}{\bullet}\stackrel{}{\longrightarrow}\stackrel{i_{-k}}{\bullet}$ are sides to the same polygon $\wt{\Pi}(\alpha)$. Hence $i_{-N}\in\s^+_{C'_\alpha}(q_C^\alpha(D))$ and then Lemma \ref{lemcns2} implies that $c\in\k_{C'_\alpha}(q_C^\alpha(D))$. The converse is straightforward.
		\item We prove that $\k_C(D)\subset \k_{C'}(D)\sqcup \k_C(r_{C'}(D))$. Let  $c=\stackrel{i_0}{\bullet}\longrightarrow\stackrel{i_1}{\bullet}$ be a chord in $\k_C(D)$, and
		$$\stackrel{i_{-N}}{\bullet} \longrightarrow\cdots\longrightarrow \stackrel{i_{-1}}{\bullet}\stackrel{}{\longrightarrow}\stackrel{i_0}{\bullet}\stackrel{c}{\longrightarrow} \stackrel{i_1}{\bullet}$$
		the path given by Lemma \ref{lemcns2}.
		We know that $i_{-N}\in\s^+_C(D)$. According to $2.$, we have two possibilities: either $i_{-N}\in\s^+_{C'}(D)$ and then Lemma \ref{lemcns2} implies that $c\in\k_{C'}(D)$, or $i_{-N}\in \s^+_C(r_{C'}(D))$ and then Lemma \ref{lemcns2} implies that $c\in\k_S(r_T(D))$. 
		\item We prove that $\k_{C'}(D)\cap C\subset \k_C(D)$. This is straightforward using the characterization of Lemma \ref{lemcns1}.
		\item We prove that $\k_C(r_{C'}(D))\subset\k_C(D)$. We use the characterization of Lemma \ref{lemcns1}. If a chord $c\in C$ is not in $\k_C(D)$, then one of the conditions $(K1)$, $(K2)$, $(K3)$ is not satisfied.\\
		If $(K1)$ is not satisfied in $D$, this means that the path starting from $c$ in $C$ goes to the root. Then no chord in this path is in $\k_C(D)$, and a fortiori in $\k_{C'}(D)$. Thus no chord is this path is flipped in $r_{C'}(D)$ and condition $(K1)$ is not satisfied in $r_{C'}(D)$.\\
		If $(K2)$ is not satisfied in $D$, this means that in the path
		$$\stackrel{i_0}{\bullet}\stackrel{c}{\longrightarrow} \stackrel{i_1}{\bullet}\longrightarrow\cdots\longrightarrow\stackrel{i_{M-1}}{\bullet}\longrightarrow \stackrel{i_M}{\bullet}$$
		starting at $c$ in $C$, there is an increasing arrow $i_{k-1}<i_k$. Then for $l=1,\ldots,k-1$, the chord $\stackrel{i_{l-1}}{\bullet}\stackrel{}{\longrightarrow} \stackrel{i_l}{\bullet}$ is not in $\k_C(D)$, hence not in $\k_{C'}(D)$, then it is not flipped in $r_{C'}(D)$. The chord $\stackrel{i_{k-1}}{\bullet}\stackrel{}{\longrightarrow} \stackrel{i_k}{\bullet}$ is increasing so it cannot be flipped in $r_{C'}(D)$ according to Lemma \ref{lemcns2}. Thus we see that condition $(K2)$ is not satisfied in $r_{C'}(D)$.\\
		If $(K3)$ is not satisfied in $D$, it means that there exists a chord $c'=\stackrel{j}{\bullet}\stackrel{}{\longrightarrow} \stackrel{i_k}{\bullet}$, $c'\in C$, with $j>i_0$ for some $k=1,\ldots,M$. For the same reason as above, none of the chords $\stackrel{i_{l-1}}{\bullet}\stackrel{}{\longrightarrow} \stackrel{i_l}{\bullet}$ is flipped in $r_{C'}(D)$, for $l=1,\ldots,k$. Let us suppose that $c'$ is not flipped in $r_{C'}(D)$. Then condition $(K3)$ is still not satisfied in $r_{C'}(D)$. Now let us suppose that $c'$ is flipped in $r_{C'}(D)$. Then we necessarily have $j>i_k$, and then $c'$ becomes decreasing in $r_{C'}(D)$, hence condition $(K2)$ is not satisfied in $r_{C'}(D)$.\\
		In either case we have shown that $c\notin \k_C(r_{C'}(D))$.
		\end{enumerate}
	\end{enumerate}
	\end{proof}
	
	\begin{proof}[Proof of Lemma \ref{lemhopfalgebra}]
	\begin{enumerate}
	\item The left-hand side is obtained by contracting the chords from $C'$; the right-hand side is obtained by contracting the chords from $C$, then contracting the chords from $C'_\alpha$ for each $\alpha$. The result is thus the same since by definition $C'=C\sqcup\bigsqcup_\alpha C'_\alpha$.
	\item The left-hand side is obtained by contracting the edges from $\s^+_T(D)$, then the chords from $S$; the right-hand side is obtained by contracting the chords from $C$, then the edges from $\s^+_{C_\alpha}(q_C^\alpha(D))$ for each $\alpha$. The equality then follows from Lemma \ref{techlem2}, $1$.
	\item The left-hand side is obtained by contracting the edges from $\s^+_{C'}(D)$, then the edges from $\s^+_C(r_{C'}(D))$; the right-hand side is obtained by contracting the edges from $\s^+_C(D)$. The equality then follows from Lemma \ref{techlem2}, $2$.
	\item This follows from taking the cardinality in Lemma \ref{techlem2}, $3$.
	\end{enumerate}
	\end{proof}
		
	\section{Proof of Lemma \ref{lemsignsgraphs}}\label{appB}
	
	We leave it to the reader to check that it is enough to do the proof for three families of loops.
		\begin{enumerate}[1.]
		\item The trivial loops $$\xymatrix{(u_0,\ldots,a_i,\ldots,u_k) \ar@{-}@/^1pc/[r] \ar@{-}@/_1pc/[r] & (u_0,\ldots,b_i,\ldots,u_k)}$$
		The statement is trivial since the expression $$-\sgn(\{a_i\},J')\sgn(\{b_i\},J')$$ is symmetric in $a_i$ and $b_i$.
		\item The triangles $$\xymatrix{ (u_0,\ldots,a_i,\ldots,u_k)  \ar@{-}[rr] \ar@{-}[rd]& & (u_0,\ldots,b_i,\ldots,u_k) \ar@{-}[ld]  \\ & (u_0,\ldots,c_i,\ldots,u_k) &\\}$$
		The statement follows from the following equality, valid for any linearly ordered set $X$ and any set $\{a,b,c\}$ of pairwise disjoint elements of $X$:
		\begin{eqnarray*}
		\sgn(\{a\},X\setminus\{a,b\})\,\sgn(\{b\},X\setminus\{a,b\})&\\
		\sgn(\{b\},X\setminus\{b,c\})\,\sgn(\{c\},X\setminus\{b,c\})&\\
		\sgn(\{c\},X\setminus\{a,c\})\,\sgn(\{a\},X\setminus\{a,c\})&=-1.
		\end{eqnarray*}
		Indeed, we apply this equality to 
		$$X=(\s_C(0)\setminus\{u_0\})\sqcup\cdots\sqcup \s_C(i)\sqcup \cdots\sqcup(\s_C(k)\setminus\{u_k\}).$$
		\item The squares $$\xymatrix{(u_0,\ldots,a_i,\ldots,c_j,\ldots,u_k) \ar@{-}[r]\ar@{-}[d]  & (u_0,\ldots,b_i,\ldots,c_j,\ldots,u_k)   \\
		(u_0,\ldots,a_i,\ldots,d_j,\ldots,u_k)  &(u_0,\ldots,b_i,\ldots,d_j,\ldots,u_k) \ar@{-}[u] \ar@{-}[l] }$$		
		The statement follows from the following equality, valid for any linearly ordered set $X$ and any set $\{a,b,c,d\}$ of pairwise disjoint elements of $X$:
		\begin{eqnarray*}
		\sgn(\{a\},X\setminus\{a,b,c\})\,\sgn(\{b\},X\setminus\{a,b,c\})&\\
		\sgn(\{c\},X\setminus\{b,c,d\})\,\sgn(\{d\},X\setminus\{b,c,d\})&\\
		\sgn(\{a\},X\setminus\{a,b,d\})\,\sgn(\{b\},X\setminus\{a,b,d\})&\\
		\sgn(\{c\},X\setminus\{a,c,d\})\,\sgn(\{d\},X\setminus\{a,c,d\})&=1.
		\end{eqnarray*}
		Indeed, we apply this equality to 
		$$X=(\s_C(0)\setminus\{u_0\})\sqcup\cdots\sqcup \s_C(i) \sqcup \cdots\sqcup \s_C(j) \sqcup \cdots\sqcup(\s_C(k)\setminus\{u_k\}).$$
		\end{enumerate}

	\section{Proof of Lemma \ref{lemproductsigns}}\label{appC}
	
	Let $\sigma_C:\{1,\ldots,n\}\stackrel{\simeq}{\rightarrow}\{1,\ldots,n\}$ be the permutation defined by blocks via $\nu_C$ and $\eta_C$. Then we have
	$$\sgn(\overline{C},C)\,\sgn(\nu_C)\,\sgn(\s_C^+,\overline{\s_C^+})\,\sgn(\eta_C)=\sgn(\sigma_C)$$
	thus we only have to prove that 
	$$\sgn(\sigma_C)=(-1)^{k_C(D)}.$$
	This is a straightforward consequence of the next lemma and the fact that the signature of a cyclic permutation of length $(r+1)$ is $(-1)^r$.
	
	\begin{lem}
	Let $D$ be a dissection diagram and $C\subset\c$.
	\begin{enumerate}
	\item $\k_C(D)$ is a disjoint union of path graphs	
	$\stackrel{i_0}{\bullet}\longrightarrow \stackrel{i_1}{\bullet}\longrightarrow\cdots\longrightarrow\stackrel{i_{r-1}}{\bullet}\longrightarrow \stackrel{i_r}{\bullet}$
	with $i_0>i_1>\cdots>i_{r-1}>i_r$ and $i_0\in \s_C^+$.
	\item With these notations, $\sigma_C$ is the (commutative) product of the cycle permutations $(i_0\,\, i_1 \,\,\cdots\,\,  i_{r-1}\,\, i_r)$.
	\end{enumerate}
	\end{lem}
	
	\begin{ex}
	In the situation of Example \ref{exampleappendix}, we get $\sigma_C=(6\,\,4\,\,3\,\,2)(8\,\,7)$.
	\end{ex}
	
	\begin{proof}
	\begin{enumerate}
	\item Let us consider $\k_C(D)$ as a directed graph. If it contains a vertex attached to $l\geq 3$ chords, then there is one outcoming chord and $l-1\geq 2$ incoming chords. Hence after flipping those chords we get a vertex with $l-1\geq 2$ outcoming chords, which is impossible. Hence $\k_C(D)$ is a disjoint union of path graphs, and the same reasoning shows that in an individual component of $\k_C(D)$, all chords have the same direction. The statement then follows from Lemma \ref{lemcns2}.
	\item 
		Let us denote by $\varepsilon:\c\stackrel{\simeq}{\rightarrow}\s$ the bijection (\ref{conventionbijection}).
		\begin{enumerate} 
		\item  Let us first consider the bijection $\eta_C:C\stackrel{\simeq}{\rightarrow}\overline{\s_C^+}$ related to $r_C(D)$. If a chord $c\in C$ starting at the vertex $i$ is not in $\k_C(D)$ then it keeps the same direction in $r_C(D)$ and we get $\eta_C(c)=\varepsilon(c)$, which means $\eta(i)=i$. Now let us consider a chord $c=i_k=\stackrel{i_{k-1}}{\bullet}\longrightarrow \stackrel{i_k}{\bullet}$ with the notation of $1.$, $k=1,\ldots,r$. Then this chord changes direction in $r_C(D)$ and becomes $\stackrel{i_{k-1}}{\bullet}\longleftarrow \stackrel{i_{k}}{\bullet}$, hence $\eta_C(c)$ is the edge starting at $i_k$ and we get $\eta_S(i_{k-1})=i_k$.
		\item Let us now consider the bijection $\nu_C:\overline{C}\stackrel{\simeq}{\rightarrow}\s_C^+$ related to $q_C(D)$. Let $\stackrel{i_0}{\bullet}\longrightarrow \stackrel{i_1}{\bullet}\longrightarrow\cdots\longrightarrow\stackrel{i_{r-1}}{\bullet}\longrightarrow \stackrel{i_r}{\bullet}$ be a connected component of $\k_C(D)$ as in $1$. According to Lemma \ref{lemcns2}, the chord starting at $i_r$ is necessarily in $\overline{C}$. In $q_C(D)$, all the points $i_k$, $k=0,\ldots,r$, are identified, and thus the edge starting at $i_r$ is the edge starting at $i_0$, which means that $\nu_C(i_r)=i_0$. It is easy to check that for all other chords $c\in \overline{C}$ we get $\nu(c)=\varepsilon(c)$. This concludes the proof of the lemma.
		\end{enumerate}
	\end{enumerate}
	\end{proof}

\bibliography{biblio}

\end{document}